\documentclass[a4paper,leqno,12pt]{amsart} 
\usepackage{lmodern}
\usepackage[T1]{fontenc} %allows different characters in output
\usepackage[utf8]{inputenc} %allows accents and the like in input

\usepackage{amssymb}
\usepackage[mathcal]{euscript} 

\usepackage{color}

\newcommand{\az}{\color{blue}}

\newcommand{\bydef}{:=}

\newcommand{\ol}[1]{\overline{#1}}

\newcommand{\id}{\mathrm{id}}%identity map
\newcommand{\proj}{\mathrm{proj}}%projection map
\newcommand{\diag}{\mathrm{diag}}

\newcommand{\tr}{\mathrm{tr}}

\newcommand{\bi}{\mathbf{i}}%imaginary unit
\newcommand{\bj}{\mathbf{j}}
\newcommand{\bk}{\mathbf{k}}

\newcommand{\cC}{\mathcal{C}}

\newcommand{\cL}{\mathcal{L}}

\newcommand{\cS}{\mathcal{S}}

\newcommand{\cV}{\mathcal{V}}
\newcommand{\cW}{\mathcal{W}}

\newcommand{\sig}{\mathrm{sign}}
\newcommand{\frg}{{\mathfrak g}}%Lie algebras
\newcommand{\frh}{{\mathfrak h}} 
\newcommand{\frc}{{\mathfrak
c}} 
\newcommand{\frd}{{\mathfrak d}}
 
\newcommand{\fre}{{\mathfrak e}} 

\newcommand{\frs}{{\mathfrak s}} 
\newcommand{\fra}{{\mathfrak a}}

\newcommand{\ZZ}{\mathbb{Z}}

\newcommand{\RR}{\mathbb{R}} 
\newcommand{\CC}{\mathbb{C}}
\newcommand{\HH}{\mathbb{H}} 

\newcommand{\FF}{\mathbb{F}} 

\DeclareMathOperator{\Hom}{\mathrm{Hom}}
\DeclareMathOperator{\End}{\mathrm{End}}
\DeclareMathOperator{\inder}{\mathfrak{inder}}%inner derivations

\DeclareMathOperator{\Mat}{\mathrm{Mat}}

\DeclareMathOperator{\trace}{tr}
\DeclareMathOperator{\espan}{\mathrm{span}}

\newcommand{\Lie}{\mathrm{Lie}}%Lie functor

\newcommand{\ad}{\mathrm{ad}}
\newcommand{\Ad}{\mathrm{Ad}}

\newcommand{\frsl}{{\mathfrak{sl}}}
\newcommand{\frsp}{{\mathfrak{sp}}}
\newcommand{\frso}{{\mathfrak{so}}}

\newcommand{\frgl}{{\mathfrak{gl}}}

\newcommand{\frsu}{{\mathfrak{su}}}

\newcommand{\fru}{{\mathfrak{u}}}

\newcommand{\subo}{_{\bar 0}} 
\newcommand{\subuno}{_{\bar 1}}
\newcommand{\Cl}{\mathfrak{Cl}} %Clifford algebra
\newcommand{\bup}{\textup{b}}
\newcommand{\baup}{\textup{b}_{\textup{a}}}
\newcommand{\qup}{\textup{q}}

\newcommand{\hup}{\textup{h}}

\newcommand{\SU}{\mathrm{SU}}
\newcommand{\spf}{\mathfrak{sp}}
\newcommand{\der}{\mathfrak{der}}

\newcommand{\slf}{\mathfrak{sl}}
\newcommand{\gl}{\mathfrak{gl}}

\newcommand{\g}{\mathfrak{g}}

\newtheorem{theorem}{Theorem}[section]
\newtheorem{proposition}[theorem]{Proposition}
\newtheorem{lemma}[theorem]{Lemma}
\newtheorem{corollary}[theorem]{Corollary}

\theoremstyle{definition} 
\newtheorem{df}[theorem]{Definition}

\theoremstyle{remark} \newtheorem{remark}[theorem]{Remark}
\numberwithin{equation}{section}

\begin{document}

\title[ Real  symplectic triple systems ]{  Classification of real simple symplectic triple systems}
%title of paper and the running head option

\author[C.~Draper]{ %dedication if necessary
Cristina Draper${}^*$}
\address[C.\,D.]{%  
Departamento de Matem\'{a}tica Aplicada,  Escuela de Ingenier\'\i as Industriales,\endgraf
Universidad de M\'{a}laga, 
 29071 M\'{a}laga,  
Spain
}
\email{cdf@uma.es}

\author[A.~Elduque]{ %dedication if necessary
Alberto Elduque${}^\star$ 
%\bigskip \\ %first author's name and the running head option
}  
\address[A.\,E.]{Departamento de
Matem\'{a}ticas e Instituto Universitario de Matem\'aticas y
Aplicaciones, Universidad de Zaragoza, 50009 Zaragoza, Spain}
\email{elduque@unizar.es}

%%%%%%%%%%%%%%% footnote %%%%%%%%%%%%%%%%
\thanks{${}^*$ Supported by    Junta de Andaluc\'{\i}a  through projects  FQM-336, UMA18-FEDERJA-119, and PAIDI project P20\_01391, and  by the Spanish Ministerio de Ciencia e Innovaci\'on   through projects  PID2019-104236GB-I00 and PID2020-118452GB-I00, all of them with FEDER funds.} 
\thanks{${}^\star$ Supported by grant
MTM2017-83506-C2-1-P (AEI/FEDER, UE) and by grant E22\_17R (Gobierno de 
Arag\'on, Grupo de referencia ``\'Algebra
y Geometr{\'\i}a''), cofunded by Feder 2014-2020 ``Construyendo
Europa desde Arag\'on''.}

\begin{abstract}
The  simple symplectic triple systems over the real numbers are  classified up to isomorphism, and linear models of all of them are provided. 

Besides the split cases, one for each complex simple Lie algebra, there are two kinds of non-split real simple symplectic triple systems with classical enveloping algebra, called unitarian and quaternionic types, and five non-split
real simple symplectic triple systems with exceptional enveloping algebra.
\end{abstract}

\maketitle

%%%%%%%%%%%%%%%%%%%%%%%%%%%%%%%%%%%%%%%%%%%%%%%
%%%%%%%%%%%%%%%%%%%%%%%%%%%%%%%%%%%%%%%%%%%%%%%

\section{Introduction }

Symplectic triple systems constitute a ternary structure that appears widely in the literature, although often hidden under a variety of algebraic and geometric objects. They were introduced in 1975  in \cite{ternarias} by Yamaguti and Asano in their attempt to understand the ternary structures involved in Freudenthal's construction of the exceptional Lie algebras (see for instance,  \cite{FreuII}). Thus, for $T$ a symplectic triple system, then $T\oplus T$ is a Lie triple system, and, in general, symplectic triple systems are variations on the Freudenthal triple systems \cite{Mey}, or the balanced symplectic ternary algebras \cite{Ferrar}, among others.

Although several of these ternary structures are related, each of them with its own set of defining identities and most of them participating on the construction of Lie algebras,  the concrete point of view provided by  symplectic triple systems not only has algebraic interest but also has many applications to Differential Geometry. For instance, complex symplectic triple systems appear naturally in the study of 
3-Sasakian homogeneous manifolds $M=G/H$ \cite[Remark~4.11]{nues3Sas}: the complexification of the reductive pair $(\Lie(G), \Lie(H))$ coincides necessarily with the standard enveloping algebra and the inner derivation algebra of a simple complex symplectic triple system.
The paper \cite{hol} takes advantage of such facts to write, for some convenient invariant affine connections, the (complexification of the) curvature tensors of any 3-Sasakian homogeneous manifold  in terms of the related triple system  and to compute the corresponding holonomy algebras,  supporting in this way the choice of some specific affine connections suitable for studying  (not necessarily homogeneous) 3-Sasakian manifolds. Inspired by \cite{nues3Sas},   an Einstein manifold  is constructed related to each simple real  symplectic triple system in \cite{CM}. In fact, the proof in \cite{CM} that the Ricci tensor is a multiple of the metric is purely algebraic, making use only of the properties of the symplectic triple systems. Precisely this work    is part of our motivation to classify completely the simple real  symplectic triple systems, in order to obtain an algebraic description of such family of   Einstein manifolds (of their tangent spaces) which permits to do explicit computations and to give a concrete algebraic treatment. These manifolds appear independently in \cite{cortes1}, where Alekseevsky and Cort\'es study para-quaternionic K\"ahler manifolds from a quite different approach based on   some kind of twistor spaces. Our real symplectic triple systems   appear (see \cite{Wallach}) as a quaternionic  discrete series of representations of simple Lie groups $G$ such that the symmetric space $G/K$ (for $K$ its  maximal  compact  subgroup) has a $G$-invariant quaternionic structure, which forces $K$ to contain a normal subgroup   isomorphic to $\SU(2)$. This implies that the representation (corresponding to the triple system as a module for its inner derivation algebra)  can be constructed as the cohomology  of certain   vector bundles. Also, these representations correspond to quaternionic representations of the \emph{quaternionic real} forms of exceptional groups of real rank $4$ studied in  the work \cite{Loketesis}, which is a continuation of \cite{Wallach}.

The symplectic triple systems appear as the $1$-homogeneous components of \emph{contact gradings}. A  grading on a semisimple Lie algebra 
$\mathfrak{g}=\sum_{i=-2}^2\g_i$ is said to be \emph{contact} if $\dim\g_{-2}=1$ and the bracket 
$[\,,\,]\colon\g_{-1}\times \g_{-1}\to \g_{-2}$ is non-degenerate, that is, $[x,\g_{-1}]\ne0$ if $0\ne x\in \g_{-1}$.   These gradings appear naturally in Differential Geometry because they are related to parabolic geometries which have an underlying contact structure (see \cite[\S3.2.4]{parabolic} for the complex case and \cite[\S3.2.10]{parabolic} for the real case). Contact gradings can only exist on simple Lie algebras, and on each complex simple Lie algebra of rank larger than one there is a unique (up to inner automorphism) contact grading.
 For a nice survey   on differential systems associated with simple graded Lie algebras including contact gradings, the reader may consult \cite{Yamaguchi}. 
 
 Note that this strong   relationship between symplectic triple systems and Differential Geometry does not mean that simple symplectic triple systems are 
 well-known. On the contrary, their simple enveloping Lie algebras are well-known, while concrete models including expressions for  the triple products are more necessary than ever, specially models suitable to be used in applications. 
In this work, we   exhibit models based exclusively on Linear Algebra that do not make use of Jordan algebras or structurable algebras, with a view towards its potential applications. 
 
A last word about applications to Projective Geometry. Given the contact grading of a complex simple Lie algebra $\g$, the associated \emph{Freudenthal variety} 
(in $\mathbb P(\g_1)$)      is just the projectivization of the set $ \{0\ne x\in\g_1:\ad(x)^2\g_{-2}=0\}$    with its reduced structure. A  systematic study of the projective geometry of Freudenthal varieties is developed in \cite{proyFr}, making use precisely of the triple product provided by the symplectic triple system.
Some applications to incidence geometries come from abelian inner ideals.
Note   that $\g_{-2}$ is an (obviously minimal, being one-dimensional) abelian inner ideal of $\g$, that is, $[\g_{-2},[\g_{-2},\g]]\subset \g_{-2}$ (any abelian inner ideal of a simple complex Lie algebra $\g$  is the corner of a $\ZZ$-grading). The \emph{inner ideal geometry} of a Lie algebra $\g$ is the point-line geometry with the minimal inner ideals as point set, as lines the inner line ideals, and inclusion as incidence. This inner ideal geometry is a generalization of an extremal geometry.  
In fact, inner ideals are used to construct  Moufang sets, Moufang triangles and Moufang hexagons  in \cite{Jeroen}, which uses the structurable algebra related to symplectic triple systems coming from a  Jordan division algebra.

 Coming back to the algebraic origin of these ternary structures as pieces involved in the construction of simple (5-graded) Lie algebras,  this role can be extended to arbitrary fields of characteristic different from 2 and 3, but in characteristic 3,  symplectic triple systems are used surprisingly to construct simple Lie superalgebras  (see \cite{Eld06} or the review \cite{alb_resumen}). Conversely, orthogonal triple systems, which provide   constructions of Lie superalgebras if the characteristic of the field is different from 2 and 3,
  are used in characteristic 3  to construct simple Lie  algebras. That is, symplectic triple systems 
  are a source of nice algebraic constructions. 
 
 Also the interplay among several families of ternary structures provides different landscapes (for instance, Freudenthal triple systems constitute the ternary structure used in \cite{CM}  to describe a tensor  measuring how far is the manifold from being of constant sectional curvature equal to $1$).  Some relations between ternary non-associative structures in general  (and symplectic triple systems in particular) and Physics are explained in \cite{kerner}.

 There is a strong connection between symplectic triple systems and structurable algebras of skew-dimension one. Structurable algebras form a class of non-associative algebras  with involution introduced by Bruce Allison in \cite{Allison78}, which always gives rise to a 5-graded Lie algebra. (They include, for instance, Jordan algebras with trivial involution, and associative algebras with involution.) The structurable algebras of skew-dimension one  have been constructed by means of non-linear isotopies of cubic norm structures   and also by means of hermitian cubic norm structures (see the review \cite{skew1_a}), but it would be convenient to have easier models in the real case based on Linear Algebra, obtained from our models of symplectic triple systems.

The classification of   symplectic triple systems for an arbitrary algebraically closed field of characteristic different from 2 and 3 is obtained in \cite[Theorem~2.30]{Eld06}. This is not the case for the real simple symplectic triple systems, despite  the large amount of information scattered among the  papers mentioned above on Differential Geometry, where the concrete expression for the ternary product is rarely made explicit.   

\medskip

The structure of the paper is as follows.  
Section 2   deals with the necessary background on symplectic triple systems, including the concepts of isomorphism and weak isomorphism, which in the real case are not equivalent in an obvious way. Each simple symplectic triple system has two related Lie algebras, namely, the   simple and $\ZZ/2$-graded standard enveloping algebra and the inner derivation algebra. 
Some results on $\ZZ/2$-graded Lie algebras will be stated  in Section 3 over fields of zero characteristic, which will be crucial for the classification. 
Section 4   gives precise models of all the split simple symplectic triple systems, valid over any field of zero characteristic, in particular over the real field. Precisely, special, symplectic and orthogonal models are recalled, while for the exceptional cases, new models are given using only Linear Algebra tools, that is, vector spaces, exterior powers,  symmetric or alternating bilinear forms and so on, instead of cubic norms or Jordan algebras, which appear usually in the exceptional descriptions.
In Section 5 specific non-split models for the real case will be described: unitarian and quaternionic symplectic triple systems among the symplectic triple systems with classical envelope, as well as several symplectic triple systems with exceptional envelope.  Again the provided exceptional examples    make use only of Linear Algebra tools.
The rest of the work will show that the list of examples is exhaustive. Section 6 contains the 
classification of the real simple symplectic triple 
systems up to weak isomorphism and 
Section 7 is devoted to proving that two weakly isomorphic simple real symplectic triple systems are necessarily isomorphic. It must be remarked here that over an algebraically closed field, two weakly isomorphic triple systems are isomorphic in an obvious way, but this is not the case for arbitrary fields (see Definition \ref{df_weakisom} and the paragraph preceeding it). The final result (Theorem \ref{th_main}) summarizes the classification of the simple real symplectic triple systems up to isomorphism.

%------------------------------------------------------------------------------------------------------------

\section{  Symplectic Triple Systems} 

%------------------

 This material goes back essentially to \cite{ternarias}, although we follow the approach in \cite{Eld06} and \cite{hol}.
 
 Throughout this work $\FF$ will always denote a field of zero characteristic, although some results are valid in a more general setting. All vector spaces and algebras considered will be assumed to be finite-dimensional.

\begin{df}\label{df:STS}  
Let    $T$ be a vector space over a field $\FF$ endowed with 
a non-zero alternating bilinear form $(\cdot\vert\cdot)\colon T\times T\to\mathbb F$, 
  and a trilinear product $[\cdot,\cdot,\cdot]\colon T\times T\times T\to T$.
The triple $\bigl(T,[\cdot,\cdot,\cdot],(\cdot\vert\cdot )\bigr)$ is said to be a \emph{symplectic triple system} if the following conditions are satisfied 
\begin{gather}
[x,y,z]=[y,x,z],\label{eq_uno}\\
[x,y,z]-[x,z,y]=(x\vert z)y-(x\vert y)z+2(y\vert z)x,\label{eq_dos}\\
[x,y,[u,v,w]]=[[x,y,u],v,w]+[u,[x,y,v],w]+[u,v,[x,y,w]],\label{eq_tres}\\
([x,y,u]\vert v)=-(u\vert [x,y,v]),\label{eq_cuatro}
  \end{gather}
  for all $x,y,z,u,v,w\in T$. 
\end{df} 

   An \emph{ideal} of the symplectic triple system $T$ is a subspace $I$ of $T$ such that $[T,T,I]+[T,I,T]\subset I$; and the system is said to be  \emph{simple} if $[T,T,T]\ne0$ and it contains no proper ideal. The simplicity of $T$ is equivalent to the non-degeneracy of the bilinear form (\cite[Proposition~2.4]{Eld06}), and this implies that the simplicity is preserved under scalar extensions(!). 
   
Of course two symplectic triple systems $\bigl(T,[\cdot,\cdot,\cdot],(\cdot\vert\cdot )\bigr)$ and $\bigl(T',[\cdot,\cdot,\cdot]',(\cdot\vert\cdot )'\bigr)$ are said \emph{isomorphic} when there is a bijective map $f\colon T\to T'$ such that $(x\vert y)=(f(x)\vert f(y))'$ and $f([x,y,z])=[f(x),f(y),f(z)]'$ for any $x,y,z\in T$.
  \medskip
  
  There is a close relationship between symplectic triple systems and a particular kind of $\ZZ/2$-graded Lie algebras. 
  We denote by $d_{x,y}$ the linear map
  $$
  d_{x,y}:=[x,y,\cdot ]\in {\rm{End}}_{\FF}(T),
  $$
for $x,y\in T$. 
   Observe that the above set of identities can be read in the following way. 
  By \eqref{eq_cuatro}, $d_{x,y}$ belongs to the symplectic Lie algebra
 $$
  \spf(T,(\cdot,\cdot  ))=\{d\in\gl(T): (d(u)\vert v)+(u\vert d(v))=0\ \forall u,v\in T\},
  $$
 which is a subalgebra of the general linear algebra $\gl(T)=(\mathop{\rm{End}}_{\FF}(T),[\cdot,\cdot])$;
and by \eqref{eq_tres}, $d_{x,y}$ belongs to the   Lie algebra of derivations of the triple too, i.e.,
\begin{multline*}
   \der(T,[\cdot,\cdot,\cdot ])\bydef\{d\in\gl(T): \\
   d([u,v,w])=[d(u),v,w]+[u,d(v),w]+[u,v,d(w)]\ \forall u,v,w\in T\},
\end{multline*}
  which is also a Lie subalgebra of   $\gl(T)$.
  The set of \emph{inner derivations} is the linear span 
  $$
  \inder(T):=\left\{\sum_{i=1}^nd_{x_i,y_i}:x_i,y_i\in T,\,n\ge1\right\},
  $$
  which is a Lie subalgebra of $\der(T,[\cdot,\cdot,\cdot ])$, again taking into account \eqref{eq_tres}. Now, consider
  $(V,\langle\cdot\vert\cdot\rangle)$ a two-dimensional vector space endowed with a non-zero alternating bilinear form, and the vector space
  $$
  \g(T):=\spf(V,\langle\cdot\vert\cdot\rangle)\oplus \inder(T)\oplus\, (V\otimes T).
  $$ 
  Then $ \g(T)$ is endowed with a $\ZZ/2$-graded Lie algebra structure (\cite[Theorem~2.9]{Eld06}) such that the even part is 
  $\g(T)_{\bar0}:=\spf(V,\langle\cdot\vert\cdot\rangle)\oplus \inder(T)$, a direct sum of two ideals, and the odd part is $\g(T)_{\bar1}:=  V\otimes T$. The anticommutative product is given by 
  \begin{itemize}
  \item the usual bracket on $\g(T)_{\bar0}$, which in this way is a Lie subalgebra of $ \g(T)$;
  \item the natural action of $\g(T)_{\bar0}$ on $\g(T)_{\bar1}$, that is,
  $$
   [\gamma+d,a\otimes x]=\gamma(a)\otimes x+a\otimes d(x),
  $$ 
  for $\gamma \in\spf(V,\langle\cdot\vert\cdot\rangle)$, $d\in\inder(T)$, $a\in V$, $x\in T$; and
 \item the product of two odd elements as follows
  \begin{equation}\label{eq_sympproducto}
  [a\otimes x,b\otimes y]=(x\vert y)\gamma_{a,b}+\langle a\vert b\rangle d_{x,y}\in \g(T)_{\bar0},
\end{equation}
  for $a,b\in V$ and $x,y\in T$, where $\gamma_{a,b}\in\spf(V,\langle\cdot\vert\cdot\rangle)$ is defined by
\begin{equation}\label{de_gamma}
  \gamma_{a,b}:=\langle a\vert\cdot\rangle b+\langle b\vert\cdot\rangle a.
  \end{equation}
  \end{itemize}

  The Lie algebra $ \g(T)$ is called the \emph{standard enveloping algebra} related to the symplectic triple system $T$. Moreover,
  the $\ZZ/2$-graded Lie algebra $ \g(T)$ is  simple if and only if so is the symplectic triple system $\bigl(T,[\cdot,\cdot,\cdot],(\cdot\vert \cdot)\bigr)$.
Actually, the standard enveloping algebras of symplectic triple systems are 
  $\ZZ/2$-graded Lie algebras of a very specific type. 
  
We will need the next well known result.

\begin{lemma}\cite[Lemma~1.4]{Vinberg}\label{le_Vinberg}
For any $i=1,2,3$, take $X_i$ a trivial $L$-module and  $U_i$ an irreducible $L$-module such that $\dim\Hom_L(U_1\otimes U_2,U_3)=1$. Fix $p\colon U_1\times U_2\to U_3$ a non-zero $L$-invariant map. Then, for any $\varphi\colon  (U_1\otimes X_1)\times (U_2\otimes X_2)\to (U_3\otimes X_3)$ an $L$-invariant map, there exists a bilinear map $\eta\colon X_1\times X_2\to X_3$ such that $\varphi(u_1\otimes x_1,u_2\otimes x_2)=p(u_1,u_2)\eta(x_1,x_2)$.
\end{lemma}

  In consequence,  if $\g=\g_{\bar0}\oplus\g_{\bar1}$ is a $\ZZ/2$-graded Lie algebra such that there are
   a two-dimensional vector space endowed with a non-zero alternating bilinear form
   $(V,\langle\cdot\vert\cdot\rangle)$, an ideal $\mathfrak{s}$ of $\g_{\bar0}$, and  an $\mathfrak{s}$-module $T$, such that
\begin{equation}\label{eq_tipoenvolvente}
 \g_{\bar0}=\spf(V,\langle\cdot\vert\cdot\rangle)\oplus\mathfrak s,\qquad \g_{\bar1}=V\otimes T
\end{equation}
 and the action of $\g_{\bar0}$ on $\g_{\bar1}$ is the natural one, then the invariance of the Lie bracket in $\g$ under the $\spf(V,\langle\cdot\vert\cdot\rangle)$-action provides the existence of
  \begin{itemize}
  \item an alternating bilinear form 
  $(\cdot\vert \cdot)\colon T\times T\to\mathbb F$, and
  \item a symmetric bilinear map $d\colon T\times T \to \mathfrak{s}$, 
  \end{itemize}
  such that \eqref{eq_sympproducto} holds for any $x,y\in T$, $a,b\in V$. Under these assumptions,  (see, for instance, the arguments in \cite[Theorem~4.4]{alb_sts} interpreting \cite{ternarias}), if $(\cdot\vert \cdot)\ne0$, then
  $\bigl(T,[\cdot,\cdot,\cdot],(\cdot\vert \cdot )\bigr)$ is a symplectic triple system for the triple product on $T$ defined by $[x,y,z]:=d_{x,y}(z)\equiv d({x,y}). z\in\mathfrak{s}. T\subset T$. If 
 $\g$ is simple, then $T$ is a simple symplectic triple system with 
 $$
 \g\cong\g(T),\quad \mathfrak{s}\cong\inder(T). 
 $$ 
 
From here it is not clear whether the pair $(\g(T),\inder(T))$ determines $T$ in some way.

If $\bigl(T,[\cdot,\cdot,\cdot],(\cdot\vert\cdot )\bigr)$ is a  symplectic triple system and $0\ne\alpha\in\mathbb F$, then we can consider another symplectic triple system 
 $T^{[\alpha]}:=\bigl(T,[\cdot,\cdot,\cdot]^\alpha,(\cdot\vert\cdot )^\alpha\bigr)$ by means of
\begin{equation}\label{eq_shift}
 (x\vert y)^\alpha=\alpha(x\vert y);\qquad [x,y,z]^\alpha=\alpha[x,y,z];
 \end{equation}
 for any $x,y,z\in T$, which will be called the \emph{$\alpha$-shift} of $T$.
  Observe that   $\inder(T^{[\alpha]})=\inder(T)$ and $\g(T^{[\alpha]})\cong\g(T)$ because the inner derivations of  $T^{[\alpha]}$ are trivially the maps
 $d_{x,y}^\alpha:=[x,y,-]^\alpha=\alpha d_{x,y}$.

%%%%%%%%%%%%%%%%%%%%%%%%%%%%%%%%%%%%%%%%%%%%%%%
%%%%%%%%%%%%%%%%%%%%%%%%%%%%%%%%%%%%%%%%%%%%%%%

\bigskip

\section{$\ZZ/2$-graded simple Lie algebras}

The standard enveloping algebra of a symplectic triple system $T$ belongs to a
very specific class of $\ZZ/2$-graded Lie algebras (see
\eqref{eq_tipoenvolvente}). In this section, properties of $\ZZ/2$-graded Lie algebras will be exploited to ensure the existence, under certain conditions, of symplectic triple systems
over non algebraically closed fields.  (See Corollary \ref{co:Z2_STS}.)

Our first result is well-known.

\begin{lemma}\label{le:Z2}
Let $\cL=\cL\subo\oplus\cL\subuno$ be a $\ZZ/2$-graded finite-dimensional simple Lie algebra over an algebraically closed field $\FF$ of characteristic $0$, with 
$\cL\subuno\neq 0$.
\begin{enumerate}
\item 
$\cL\subo$ is reductive and it acts faithfully and completely reducibly on 
$\cL\subuno$, which is a sum of at most two irreducible
$\cL\subo$-modules.
\item
The center of $\cL\subo$: $Z(\cL\subo)$, is non-zero if and only if 
$\cL\subuno$ is not irreducible as a module for $\cL\subo$. In this case 
$\cL\subuno=\cV\oplus\cW$ for irreducible $\cL\subo$-modules $\cV$ 
and $\cW$ and the direct sum $\cW\oplus\cL\subo\oplus\cV$ is a 
$\ZZ$-grading of $\cL$, with $\cL_{-1}=\cW$, $\cL_0=\cL\subo$, and 
$\cL_1=\cV$. Moreover, $Z(\cL\subo)=\FF z$ with 
\[
\ad_z\vert_\cV=\id,\qquad \ad_z\vert_\cW=-\id.
\]
\item
The space of symmetric, $\cL\subo$-invariant, bilinear forms 
$\cL\subuno\times\cL\subuno\rightarrow \FF$ has dimension one.
\end{enumerate}
\end{lemma}

We include a proof, due to the lack of a suitable reference. 

Note that the uniqueness, up to scalars, of invariant metrics, is important quite often in arguments in Differential Geometry.

\begin{proof}
The adjoint representation restricts to a faithful representation 
$\rho\colon \cL\subo\rightarrow \frgl(\cL)$. The subspaces $\cL\subo$ and 
$\cL\subuno$ are
orthogonal relative to the Killing form $\kappa$ of $\cL$. In particular,
$\kappa\subo=\kappa\vert_{\cL\subo}$ is non-degenerate, and this is
the trace form $(x,y)\mapsto \trace\rho(x)\rho(y)$. By 
\cite[p.~99]{EldLA}, we conclude that 
$\cL\subo=Z(\cL\subo)\oplus [\cL\subo,\cL\subo]$, with $Z(\cL\subo)$ 
toral, $ [\cL\subo,\cL\subo]$ semisimple, and that  $\cL\subuno$ is a
 completely reducible $\cL\subo$-module.

If $\cV$ is an $\cL\subo$-submodule of $\cL\subuno$ and $[\cV,\cV]\neq 0$, then $[\cV,\cV]\oplus [\cV,[\cV,\cV]]$ is an ideal of $\cL$:
\begin{itemize}
\item 
It is $\cL\subo$-invariant.
\item 
$[\cL\subuno,[\cV,[\cV,\cV]]]\subseteq
[\cL\subo,[\cV,\cV]]+[\cV,[\cL\subo,\cV]]\subseteq [\cV,\cV]$.
\item 
By complete reducibility, $\cL\subuno=\cV\oplus\cW$ for some 
$\cL\subo$-submodule $\cW$, and hence  
$[\cL\subuno,[\cV,\cV]]=[\cV,[\cV,\cV]]+[\cW,[\cV,\cV]]$. As
\[
[\cW,[\cV,\cV]]\subseteq
\begin{cases}
[\cW,\cL\subo]=\cW,\\
[[\cW,\cV],\cV]\subseteq [\cL\subo,\cV]\subseteq\cV,
\end{cases}
\]
then $[\cW,[\cV,\cV]]=0$, so we get
$[\cL\subuno,[\cV,\cV]]=[\cV,[\cV,\cV]]$.
\end{itemize}
As $\cL$ is simple, it follows that $\cL=[\cV,\cV]\oplus [\cV,[\cV,\cV]]$, and in particular $\cL\subo=[\cV,\cV]$ and $\cL\subuno=\cV$. 

Now, if $\cL\subuno=\cV_1\oplus\cV_2\oplus\cV_3$ for non-zero 
$\cL\subo$-modules $\cV_1$, $\cV_2$, $\cV_3$, then the above shows 
that $[\cV_i\oplus\cV_j,\cV_i\oplus\cV_j]=0$ for any $i\neq j$, but this
gives $[\cL\subuno,\cL\subuno]=0$, a contradiction. This proves the first assertion.

If $\cL\subuno$ is an irreducible $\cL\subo$-module, then Schur's Lemma 
shows that any $0\neq z\in Z(\cL\subo)$ satisfies 
$\ad_z\vert_{\cL\subuno}=\alpha\id$ for some $0\neq\alpha\in\FF$. Hence 
we get $0=\ad_z\vert_{\cL\subo}
=\ad_z\vert_{[\cL\subuno,\cL\subuno]}=2\alpha\id$, a contradiction. Hence 
the center of $\cL\subo$ is trivial if $\cL\subuno$ is irreducible.

Otherwise, $\cL\subuno=\cV\oplus\cW$ for irreducible 
$\cL\subo$-modules $\cV$ and $\cW$. By the arguments above, 
$[\cV,\cV]=0=[\cW,\cW]$, so that 
$\cL\subo=[\cL\subuno,\cL\subuno]=[\cV,\cW]$, and the decomposition 
$\cV\oplus\cL\subo\oplus\cW$ is a $\ZZ$-grading of $\cL$ such that $\cV$, $\cL\subo$ and $\cW$
are the homogeneous components of degrees $1$, $0$ and $-1$ respectively. Thus the 
endomorphism $d$ that acts trivially on $\cL\subo$, as the identity on 
$\cV$, and as minus the identity on $\cW$, is a derivation of $\cL$. But 
$\cL $ is simple, so $d$ is inner, and hence there exists an element 
$z\in\cL$ with $d=\ad_z$. It follows that $z\in Z(\cL\subo)$. Schur's 
Lemma shows that $Z(\cL\subo)=\FF z$. This proves part (2).

As $\kappa({\cL\subo}\,,\, \cL\subuno)=0$,  the restriction 
$\kappa\vert_{\cL\subuno}\colon \cL\subuno\times
\cL\subuno\rightarrow \FF$ is a   non-degenerate symmetric $\cL\subo$-invariant
bilinear form   and, in particular, $\cL\subuno$ is a  self-dual $\cL\subo$-module.
 If $\cL\subuno$ is irreducible, then 
$\Hom_{\cL\subo}(\cL\subuno\otimes_\FF \cL\subuno,\FF)\simeq
\Hom_{\cL\subo}(\cL\subuno,\cL\subuno)$ has dimension one by 
Schur's Lemma, so that 
any $\cL\subo$-invariant bilinear form 
$\cL\subuno\times\cL\subuno\rightarrow \FF$ is scalar multiple of 
 $\kappa\vert_{\cL\subuno}$.

Otherwise, $\cL\subuno=\cV\oplus\cW$ as above and 
$\Hom_{\cL\subo}(\cV\otimes_\FF\cV,\FF)=0
=\Hom_{\cL\subo}(\cW\otimes_\FF\cW,\FF)$, because of the action of 
the element $z\in Z(\cL\subo)$, which acts trivially on $\FF$. This shows that the 
$\cL\subo$-modules $\cV$ and $\cW$ are not  self-dual, and they are dual 
of each other. Also we have that
$\Hom_{\cL\subo}(\cL\subuno\otimes_\FF\cL\subuno,\FF)=
\Hom_{\cL\subo}\bigl((\cV\oplus\cW)\otimes_\FF(\cV\oplus\cW),\FF\bigr)
{\az \simeq}\Hom_{\cL\subo}(\cV\otimes_\FF \cW,\FF)\oplus
\Hom_{\cL\subo}(\cW\otimes_\FF\cV,\FF)$ has dimension two, and again the 
space of symmetric, $\cL\subo$-invariant, bilinear forms 
$\cL\subuno\times\cL\subuno\rightarrow \FF$ has dimension one.
\end{proof}

\begin{corollary}\label{co:Z2}
Let $\cL$ be a finite-dimensional Lie algebra over a field $\FF$ of characteristic $0$, let 
$\rho\colon \cL\rightarrow\frgl(\cV)$ be a finite-dimensional, faithful, and 
completely reducible, representation of $\cL$. Then, up to scalars, there 
is at most one $\cL$-invariant bilinear map 
$\mu\colon \cV\times\cV\rightarrow\cL$ such that $\cL\oplus\cV$, with 
multiplication given by:
\begin{equation}\label{eq:LV}
[x+v,y+w]=\bigl([x,y]+\mu(v,w)\bigr)+\bigl(\rho_x(w)-\rho_y(v)\bigr)
\end{equation}
for $x,y\in\cL$ and $v,w\in\cV$, is a central simple Lie algebra.
\end{corollary}
\begin{proof}
If $\mu$ is such a map, $\frg=\cL\oplus\cV$ is a form of a simple 
$\ZZ/2$-graded  Lie algebra $\frg_{\overline{\FF}}=\frg\otimes_\FF {\overline{\FF}}$, where 
$\overline{\FF}$ is an algebraic closure of $\FF$. Let $\kappa$ be the 
Killing form of $\frg$. For any $x,y\in\cL$, 
$\kappa(x,y)=\trace(\ad_x\ad_y)+\trace(\rho_x\rho_y)$  
depends only on $\cL$ and $\rho$, so 
it gives a fixed non-degenerate invariant symmetric bilinear form on 
$\cL$.

Note also that the restriction of $\kappa$ to $\cV$ is a symmetric, 
non-degenerate, $\cL$-invariant, bilinear form on $\cV$. By   Lemma~\ref{le:Z2}, there is, up to scalars, a unique non-zero symmetric,
$\cL$-invariant, bilinear form $\bup\colon\cV\times\cV\rightarrow \FF$, and hence there is a non-zero $\alpha\in\FF$ such that 
$\kappa\vert_\cV=\alpha \bup$.

But for any $v,w\in\cV$ and $x\in\cL$, we have
\begin{equation}\label{eq:kappa_b}
\kappa\vert_\cL\bigl(x,\mu(v,w)\bigr)=\kappa\vert_\cV(\rho_x(v),w\bigr)
=\alpha \bup\bigl(\rho_x(v),w\bigr).
\end{equation}
The non-degeneracy of $\kappa\vert_\cL$ shows that $\mu$ is unique up to scalars.
\end{proof}

\begin{remark}\label{re:Z2}
In the conditions of Corollary \ref{co:Z2}, if there is an 
$\cL\otimes_\FF {\overline{\FF}}$-invariant bilinear map 
$\nu\colon\bigl(\cV\otimes_\FF {\overline{\FF}}\bigr)\times
\bigl(\cV\otimes_\FF {\overline{\FF}}\bigr)
\rightarrow\cL\otimes_\FF {\overline{\FF}}$ making 
$(\cL\oplus\cV)\otimes_\FF {\overline{\FF}}$ a simple Lie algebra with the bracket as 
in \eqref{eq:LV}, then the same is valid for $\cL$ and $\cV$.
\end{remark}
\begin{proof}
The existence of such a bilinear map 
$ \nu$ forces that the space of symmetric $\cL$-invariant bilinear forms on $\cV$ is one-dimensional   by Lemma~\ref{le:Z2}, since this dimension remains invariant under scalar extension.
Hence %\eqref{eq:kappa_b} defines 
there is a bilinear $\cL$-invariant map $\mu\colon\cV\times\cV\rightarrow \cL$ such that its complexification is a scalar multiple of $\nu$. 
The Jacobi identity for the bracket in $(\cL\oplus\cV)\otimes_\FF {\overline{\FF}}$ given by \eqref{eq:LV} for the map $\nu$ is also true for any scalar multiple of $\nu$.
Therefore, the Jacobi identity for the bracket in 
 $\cL\oplus\cV$ given by \eqref{eq:LV} follows from the Jacobi identity in $(\cL\oplus\cV)\otimes_\FF {\overline{\FF}}$ using the complexification of $\mu$.
\end{proof}

\begin{corollary}\label{co:Z2_STS}
Let $T $ be a simple symplectic triple system, over an algebraic closure 
${\overline{\FF}}$ of the field $\FF$ of characteristic $0$, 
 with inner derivation algebra $\inder(T )$. Let $\frs$ be a    form of $\inder(T )$ and let $S$ be a module for 
$\frs$ such 
that $S_{\overline{\FF}}\bydef S\otimes_\FF {\overline{\FF}}$ is isomorphic to $T $ as a 
module for $\frs\otimes_\FF{\overline{\FF}}\simeq \inder(T )$. Then, up to 
scalars, there is a unique non-degenerate alternating bilinear form 
$S\times S\rightarrow \FF$: $(x,y)\mapsto (x\vert y)$, and a unique 
triple product $S\times S\times S\rightarrow S$: 
$(x,y,z)\mapsto [x,y,z]$, that makes $S$ a symplectic triple system, with inner derivation algebra equal to the image in $\frgl(S)$ of 
$\frs$.
\end{corollary}
\begin{proof}
Let $V$ be a two-dimensional vector space over $\FF$ endowed with a non-zero 
alternating bilinear form $\langle u\vert v\rangle$. Consider the Lie
algebra $\cL=\frsl(V)\oplus \frs$, and its module $\cV=V\otimes_\FF S$.
By Corollary \ref{co:Z2} and Remark \ref{re:Z2} there is a unique, 
$\cL$-invariant 
bilinear map $\mu\colon \cV\times\cV\rightarrow \cL$, up to scalars, making 
$\cL\oplus\cV$ into a central simple Lie algebra with bracket given by
\eqref{eq:LV}. 

Recall that $\Hom_{\frsl(V)}(V\otimes V,\FF)=\FF\langle\cdot\vert\cdot\rangle$   
and that $\Hom_{\frsl(V)}(V\otimes V,\frsl(V))=\FF\gamma$, with $\gamma(u\otimes v)=\gamma_{u,v}$ the zero trace map defined by 
$\gamma_{u,v}(w)=\langle u\vert w\rangle v+\langle v\vert w\rangle u$. This fact, 
together with the
 $\frsl(V)$-invariance of $\mu$, implies that this  unique product has the form
$\mu(u\otimes x,v\otimes y)=(x\vert y)\gamma_{u,v}+
\langle u\vert v\rangle d_{x,y}$, for an alternating bilinear form 
$(x\vert y)$ and a bilinear map $S\times S\rightarrow \frs$, 
$(x,y)\mapsto d_{x,y}$,  both being completely determined by $\mu$.

Now it is enough to define $[x,y,z]=d_{x,y}(z)$ (the action of $d_{x,y}\in\frs$ on the element $z\in S$).
\end{proof}

\medskip

  The above corollary is a key tool for the classification of the  real simple symplectic triple systems. We simply have to find, for each complex simple symplectic triple system $T$, the real forms of the Lie algebra $\inder(T )$ which admit a representation whose complexification is $T$, and in such case the
  triple product is defined up to scalars.
In other words, we look for real forms of the pairs $\bigl(\inder(T),T\bigr)$, for $T$ any simple symplectic triple system  over $\CC$.
 Here a pair $(\frh,U)$ consisting of a real Lie algebra $\frh$ and 
an $\frh$-module $U$, with corresponding representation 
$\rho\colon \frh\rightarrow \frgl(U)$, is said to be a \emph{real form} of the pair $(\frg,V)$ for a complex Lie algebra $\frg$ and a $\frg$-module $V$, with representation $\mu\colon \frg\rightarrow \frgl(V)$, if there is an isomorphism of complex Lie algebras $\varphi\colon \frh\otimes_\RR\CC\rightarrow \frg$, and a linear isomorphism of complex vector spaces $\Phi\colon U\otimes_\RR\CC\rightarrow V$ such that 
\[
\Phi(\rho_x(u)\otimes \alpha)
=\alpha\mu_{\varphi(x\otimes 1)}\bigl(\Phi(u\otimes 1)\bigr)
\]
for any $x\in \frh$, $u\in U$ and $\alpha\in\CC$. 

\smallskip

Note that  the
 map $x\mapsto \beta x$ gives an isomorphism 
$T^{[\beta^2]}\rightarrow T$ for any $0\ne\beta\in\mathbb F$. In particular, over $\CC$
any shift of $T$ is 
isomorphic  to $T$, while
 over the reals, any shift of $T$ is 
isomorphic either to $T$ or to $T^{[-1]}$.  (It is not clear whether $T$ and $T^{[-1]}$ are isomorphic.) So we define:

\begin{df}\label{df_weakisom}
Two real symplectic triple systems $T$ and $T'$ are said to be 
\emph{weakly isomorphic} if $T'$ is isomorphic either to $T$ or to $T^{[-1]}$.
\end{df}

To summarize,  the uniqueness only up to scalars of the triple product in Corollary \ref{co:Z2_STS} tells us that  the lists of the real forms in   
Section~\ref{se:classification}  of the pairs $\bigl(\inder(T),T\bigr)$, for the simple symplectic triple systems over $\CC$, complete the classification of the simple real symplectic triple 
systems up to \emph{weak isomorphism}.  Finally, 
in Section~\ref{se:classificationII} it will be proven that two simple weakly isomorphic real symplectic
triple systems are actually isomorphic. (This is not trivial at all!)

%%%%%%%%%%%%%%%%%%%%%%%%%%%%%%%%%%%%%%%%%%%%%%%
%%%%%%%%%%%%%%%%%%%%%%%%%%%%%%%%%%%%%%%%%%%%%%%

\section{The split simple symplectic triple systems}\label{se:split}

The uniqueness in Corollary \ref{co:Z2_STS} shows that the 
classification of the simple symplectic triple systems over an 
algebraically closed field $\FF$ of characteristic $0$ is equivalent to the 
classification of order $2$ automorphisms (i.e., gradings by $\ZZ/2$) of 
the simple Lie algebras over $\FF$ such that the even part is the direct 
sum of a copy of $\frsl_2$ and a reductive Lie algebra $\frs$, and the 
odd part is, as a module for the even part, the tensor product of the 
two-dimensional module for $\frsl_2$ and a module $S$ for $\frs$. The possible pairs $(\frs,S)$, up to isomorphism, can be read from the classification of finite order automorphisms (see, e.g., 
\cite[Chapter 8]{Kac} and \cite{Eld06}), where two pairs $(\frs,S)$ and $(\frs',S')$ are \emph{isomorphic} if and only if there is an isomorphism 
$\varphi\colon \frs\rightarrow\frs'$ of Lie algebras and a linear isomorphism
$\psi\colon S\rightarrow S'$ such that $\psi(s.x)=\varphi(s).\psi(x)$ for all 
$s\in\frs$ and $x\in S$. 

The complete list of these pairs, up to isomorphism, is given in the following list, whose items are labeled 
according to the type of the envelope:

\begin{description}
\item[Special] 
$\bigl(\frgl(W),W\oplus W^*\bigr)$, for a non-zero vector space $W$.

\item[Orthogonal] 
$\bigl(\frsp(V)\oplus\frso(W),V\otimes W\bigr)$, where $V$ is a vector space of dimension $2$ endowed with a non-zero alternating bilinear form $\langle\cdot\vert\cdot\rangle$, and $W$ is a vector space of dimension $\geq 3$ endowed with a 
non-degenerate symmetric bilinear form.

\item[Symplectic]
$\bigl(\frsp(W),W\bigr)$ for a non-zero even-dimensional vector space $W$ endowed with a non-degenerate alternating bilinear form.

\item[$G_2$-envelope] $(\fra_1,V(\varpi_3))$.

\item[$F_4$-envelope] $(\frc_3,V(\varpi_3))$.

\item[$E_6$-envelope] $(\fra_5,V(\varpi_3))$.

\item[$E_7$-envelope] $(\frd_6,V(\varpi_6))$.

\item[$E_8$-envelope] $(\fre_7,V(\varpi_1))$.
\end{description}

In this list, the simple Lie 
algebra over $\FF$ of type $L_n$ is 
denoted $\mathfrak{l}_n$, a Cartan subalgebra and system 
$\Pi=\{\alpha_1,\ldots,\alpha_n\}$ of simple roots are fixed, with the 
ordering used in \cite{Oni},  $\varpi_1,\ldots,\varpi_n$ are the 
corresponding fundamental weights, and for any dominant weight 
$\Lambda$, $V(\Lambda)$ denotes the irreducible module of highest
weight $\Lambda$.

For each case, let us describe explicitly the triple product and the
alternating bilinear form. Actually, this works over an arbitrary field
of characteristic $0$, thus providing the list of the \emph{split simple symplectic triple systems}. Thus, in what follows, $\FF$ will denote an arbitrary field of characteristic $0$, $V$ will denote a 
two-dimensional vector space endowed with a non-zero alternating bilinear form $\langle u\vert v\rangle$. Details on the classical cases can be found in \cite[Examples 2.26]{Eld06}, while for the exceptional cases,  instead of the models based on Jordan algebras in \cite{Eld06}, new models based on Linear Algebra will be given.

\medskip

\subsection{Special type}\label{ss:special}\quad 
Let $W$ be a non-zero finite-dimensional
vector space over our ground field $\FF$, and let $T$ be the direct 
sum of $W$ and its dual: $T=W\oplus W^*$. 
Then $\bigl(T,[\cdot,\cdot,\cdot],(\cdot\vert\cdot)\bigr)$ is a simple symplectic triple system with
\begin{equation}\label{eq:special}
\begin{array}{c}
[x,f,y]=f(x)y+2f(y)x,\quad [f,x,g]:=-f(x)g-2g(x)f,\quad
[x,y,\cdot]=0=[f,g,\cdot],\\[2pt]
(f\vert x)=-(x\vert f)=f(x),\qquad (x\vert y)=0=(f\vert g)
\end{array}
\end{equation}
for any $x,y\in W$ and $f,g\in W^*$.

The inner derivation algebra and the standard enveloping algebra are the following:
\[
\inder(T)\cong\frgl(W),\qquad \frg(T)\cong\frsl(V\oplus W).
\]
Actually, this symplectic triple system is related to the natural grading by $\ZZ/2$ of $\frsl(V\oplus W)$  obtained by splitting $\frgl(V\oplus W)$ into blocks corresponding to $V$ and $W$. The even part consists of the block diagonal endomorphisms, and the odd part of the off block diagonal endomorphisms.

\medskip

\subsection{Orthogonal type}\label{ss:orthogonal}\quad
Let $(W,\bup)$ be a   vector space  of  dimension at least $3$  endowed with a
non-degenerate symmetric bilinear form $\bup$. Consider the vector
space $T=V\otimes W$. Then 
$\bigl(T,[\cdot,\cdot,\cdot],(\cdot\vert\cdot)\bigr)$ is a simple symplectic triple system with
\begin{equation}\label{eq:orthogonal}
\begin{split}
&[u\otimes x,v\otimes y,w\otimes z]=
   \frac12\bigl(\langle u\vert w\rangle v+\langle  v\vert w\rangle u\bigr)
   \otimes \bup(x,y) z\\
&\hspace*{2in}
 +\langle u\vert v\rangle w\otimes\bigl(\bup(x,z)y-\bup(y,z)x\bigr),\\[2pt] 
&(u\otimes x\vert v\otimes y)=\frac12\langle u\vert v\rangle \bup(x,y),
\end{split}
\end{equation}
for any $u,v,w\in V$ and $x,y,z\in W$. 

  Moreover, the inner derivation algebra and the standard enveloping algebra are
\[
\inder(T) \cong \frsl(V)\oplus\frso(W,\bup),\qquad
\frg(T)\cong  \frso\bigl((V\otimes V)\oplus W,\tilde\bup\bigr).
\]
Here $\tilde\bup$ is the non-degenerate symmetric bilinear form on the
direct sum $(V\otimes V)\oplus W$ that restricts to $\bup$ on $W$, satisfies $\tilde\bup(V\otimes V,W)=0$, and restricts to 
\[
(u_1\otimes v_1,u_2\otimes v_2)\mapsto 
\langle u_1\vert u_2\rangle \langle v_1\vert v_2\rangle
\]
on $V\otimes V$. Again, this corresponds to the natural $\ZZ/2$-grading
on $\frso\bigl((V\otimes V)\perp W,\tilde\bup\bigr)$ that is naturally associated to the orthogonal sum (relative to $\tilde\bup$) 
$(V\otimes V)\perp W$.

\medskip

\subsection{Symplectic type}\label{ss:symplectic}\quad
Let $T $ be an even-dimensional vector space endowed with a  
 non-degenerate alternating bilinear form   $(\cdot\vert\cdot)$. 
Then $(T,[\cdot,\cdot,\cdot],(\cdot\vert\cdot))$ is a simple symplectic triple system with the triple product given by
\begin{equation}\label{eq:symplectic} 
[x,y,z]=(x\vert z)y+(y\vert z)x,
\end{equation}
for $x,y,z\in T$. 

Moreover, the inner derivation algebra and the standard enveloping algebra are
\[
\inder(T) \cong  \frsp\bigl(T, (\cdot\vert\cdot)\bigr),\qquad
  \frg(T)\cong \frsp\bigl(V\oplus T,(\cdot\vert\cdot)'\bigr).
\]
Here $(\cdot\vert\cdot)'$ is the non-degenerate alternating bilinear 
form on the
direct sum $V\oplus T$ that restricts to $\langle\cdot\vert\cdot\rangle$
 on $V$, to $(\cdot\vert\cdot)$ on $T$, and that satisfies 
$(V\vert T)'=0$. This corresponds to the natural $\ZZ/2$-grading
on $\frsp\bigl(V\perp W,(\cdot\vert\cdot)'\bigr)$ that is naturally associated to the orthogonal sum  
$V\perp W$.

\medskip

\subsection{$G_2$-type}\label{ss:G2}\quad
Denote by  $V_n$ the vector space of the homogeneous polynomials of degree $n$ in two variables $X$ and $Y$. 
For any $f\in V_n$, $g\in V_m$, consider the \emph{transvection}
\[ 
{(f,g)_q=\frac{(n-q)!}{n!}\frac{(m-q)!}{m!}
\sum_{i=0}^q(-1)^{i}\binom{q}{i}
\frac{\partial^q f}{\partial X^{q-i}\partial Y^i}\frac{\partial^q g}{\partial X^i\partial Y^{q-i}}} \in V_{m+n-2q}.
\]

Then the vector space $T=V_3$
of the homogeneous polynomials of degree $3$, endowed with the alternating form and the triple product given by:
\begin{equation}\label{eq:G2}
(f\vert g)=(f,g)_3,\qquad
 [f,g,h]=6((f,g)_2,h)_1,
\end{equation}
for any $f,g,h\in T$,
    is a simple  symplectic triple system, whose inner derivation algebra
is isomorphic to $\frsl_2(\FF)$  and its
standard enveloping Lie algebra is isomorphic to the split simple Lie algebra of type $G_2$. 

There is a misprint in \cite[p.~210]{Eld06} so, for the convenience of the reader, we will justify the assertions above.

 The key is to use the classical Gordan identities for transvections
  \cite[p.~56--57]{GraceYoung}. For $f\in V_m$, $g\in V_n$, $h\in V_p$ and
$\alpha_1,\alpha_2,\alpha_3$ non-negative integers such that 
$\alpha_1+\alpha_2\leq p$, $\alpha_2+\alpha_3\leq m$, 
$\alpha_3+\alpha_1\leq n$, and such that either $\alpha_1=0$ or 
$\alpha_2+\alpha_3= m$, one has
\begin{multline*}
  \sum_{i\geq 0}
  \frac{\binom{n-\alpha_1-\alpha_3}{i}\binom{\alpha_2}{i}}{\binom{m+n-2\alpha_3-i+1}{i}}
             ((f,g)_{\alpha_3+i}, h)_{\alpha_1+\alpha_2-i}     
  \\[-18pt] 
    =(-1)^{\alpha_1} \sum_{i\geq 0}
    \frac{\binom{p-\alpha_1-\alpha_2}{i}\binom{\alpha_3}{i}}
    {\binom{m+p-2\alpha_2-i+1}{i}} ((f,h)_{\alpha_2+i},
    g)_{\alpha_1+\alpha_3-i}.
\end{multline*}
This identity is usually denoted by  
$\left(\begin{smallmatrix}
f& g& h\\
m&n&p\\
\alpha_1&\alpha_2&\alpha_3\end{smallmatrix}\right)$.

Gordan identity  $\left(\begin{smallmatrix} f&g&h\\ 3&3&3\\
0&1&2\end{smallmatrix}\right)$ gives
\begin{equation}\label{eq:eq_paraTdeg2}
   ((f,g)_2,h)_1+\frac 12 (f,g)_3 h = ((f,h)_1, g)_2 + ((f,h)_2, g)_1 +
   \frac 13 (f,h)_3 g.
\end{equation}
Replace $f$ with $h$ in \eqref{eq:eq_paraTdeg2}, and sum both  identities to get
\begin{equation}\label{eq:eq2_paraTdeg2}
     ((f,g)_2, h)_1 + ((h,g)_2, f)_1 + \frac 12 ((f,g)_3 h + (h,g)_3 f)   = 2((f,h)_2, g)_1\,.
\end{equation}
Replace $f$ with $g$ in \eqref{eq:eq2_paraTdeg2}, and substract both  identities to get
\begin{equation}\label{eq:eq3_paraTdeg2}
    2(f,g)_3 h +( h,g)_3f - (h,f)_3g =  
6\left( ((f,h)_2,g)_1-((g,h)_2,f)_1\right)\,.
\end{equation}
Now,   the only scalar $\alpha$  such that the triple product 
${[f,g,h]}:=\alpha((f,g)_2,h)_1$ satisfies the   identities  of a
symplectic triple system in Definition~\ref{df:STS} is $\alpha=6$. 

Over an algebraically closed field, this simple symplectic triple system corresponds to the unique order
two automorphism, up to conjugacy, of the simple Lie algebra of type $G_2$.

\medskip

\subsection{$F_4$-type}\label{ss:F4}\quad
Let $W$ be a six-dimensional vector space, endowed with a 
non-degenerate alternating bilinear form $\baup$. 
 Let 
$\{u_1,u_2,u_3,v_1,v_2,v_3\}$ be a symplectic basis: 
$\baup(u_i,v_i)=-\baup(v_i,u_i)=1$, and all the other values of $\baup$ on 
basic elements are $0$. Consider the corresponding symplectic Lie algebra
$\frsp(W,\baup)$, which is spanned by the endomorphisms 
$\gamma_{u,v}\colon w\mapsto \baup(u,w)v+\baup(v,w)u$. The subalgebra 
$\frh$ of diagonal elements, relative to the basis above,  
is a Cartan subalgebra, and the root system is 
$\Phi=\{\pm\epsilon_i\pm\epsilon_j,\pm 2\epsilon_i: 
1\leq i\neq j\leq 3\}$, where $\epsilon_i$ denotes the weight of $u_i$, 
$i=1,2,3$, under the natural action. A system of simple roots is given by 
$\Pi=\{\epsilon_1-\epsilon_2,\epsilon_2-\epsilon_3,2\epsilon_3\}$.

Let $T$ be the kernel of the  $\frsp(W,\baup)$-invariant map $\bigwedge^3W\rightarrow W$, 
$x_1\wedge x_2\wedge x_3\mapsto 
\baup(x_1,x_2)x_3+\baup(x_2,x_3)x_1+\baup(x_3,x_1)x_2$. 
  The dimension
of $T$ is $\binom63-6=14$, and $T$ is the irreducible module for $\frsp(W,\baup)$
with highest weight 
$\varpi_3=\epsilon_1+\epsilon_2+\epsilon_3$. The element
$u_1\wedge u_2\wedge u_3$ lies in $T_{\varpi_3}$.

Because of Corollary \ref{co:Z2_STS}, up to scalars, there are a unique
$\frsp(W,\baup)$-invariant alternating form $(\cdot\vert\cdot)$ on $T$, and a bilinear map $T\times T\rightarrow \frsp(W,\baup)$, $(x,y)\mapsto d_{x,y}$, such that 
$\bigl(T,[\cdot,\cdot,\cdot],(\cdot\vert\cdot)\bigr)$ is a simple symplectic triple system, for $[x,y,z]=d_{x,y}.z$, that corresponds to the pair $\bigl(\frc_3,V(\varpi_3)\bigr)$ ($F_4$-envelope).

Note that the weights of $W$ are $\pm\epsilon_i$, $1\leq i\leq 3$; and the weights of $\bigwedge^3W$ are 
$\pm\epsilon_1\pm\epsilon_2\pm\epsilon_3$, with multiplicity
$1$, and $\pm\epsilon_i$, $i=1,2,3$, with multiplicity $2$. Hence
the weights of $T$ are 
$\pm\epsilon_1\pm\epsilon_2\pm\epsilon_3$ and 
$\pm\epsilon_i$, $i=1,2,3$, all with multiplicity $1$.

 The subspace $\frsp(W,\baup)\otimes T$ is generated, as a module for $\frsp(W,\baup)$, by 
$\gamma_{u_1,u_1}\otimes (v_1\wedge v_2\wedge v_3)$, which is the tensor 
product of a highest root vector: $\gamma_{u_1,u_1}\in\frsp(W,\baup)_{2\epsilon_1}$, and a lowest 
weight vector: $v_1\wedge v_2\wedge v_3\in T_{-\varpi_3}$, in the ordering imposed by 
$\Pi$. The image of 
$\gamma_{u_1,u_1}\otimes (v_1\wedge v_2\wedge v_3)$ under any 
$\frsp(W,\baup)$-invariant linear map to $T$ lies in the weight space
$T_{\epsilon_1-\epsilon_2-\epsilon_3}$, which is
one-dimensional. We conclude that 
$\Hom_{\frsp(W,\baup)}\bigl(\frsp(W,\baup)\otimes T,T\bigr)$ is 
one-dimensional, and hence spanned by the action of $\frsp(W,\baup)$ on $T$.

Fix the determinant map $\det\colon \bigwedge^6W\rightarrow \FF$ given by 
\[
\det(u_1\wedge u_2\wedge u_3\wedge v_1\wedge v_2\wedge v_3)=1.
\]
This map $\det$ determines an alternating bilinear form on 
$\bigwedge^3 W$:
\begin{equation}\label{eq:F4alternating_form}
( x\vert y)=\det(x\wedge y),
\end{equation} 
which restricts to a 
non-degenerate alternating $\frsp(W,\baup)$-invariant bilinear map on $T$,
also denoted by $(\cdot\vert \cdot)$. In particular, $T$ is self-dual
and $\Hom_{\frsp(W,\baup)}\bigl(T\otimes T,\FF\bigr)$ is spanned by 
$(\cdot\vert\cdot)$ by 
Schur's Lemma.

Because of the self-duality  as 
$\frsp(W,\baup)$-modules of both $T$ and  $\frsp(W,\baup)$, and since the dimension of
 $\Hom_{\frsp(W,\baup)}\bigl(\frsp(W,\baup)\otimes T,T\bigr)$ is $1$, 
so is the dimension of 
$\Hom_{\frsp(W,\baup)}\bigl(T\otimes T,\frsp(W,\baup)\bigr)$.
(Note that for a Lie algebra $\cL$ and $\cL$-modules $U$ and $V$, there are isomorphisms
 $\Hom_{\cL}(U,V)\cong (U^*\otimes V)^{\cL}\cong\Hom_{\cL}(U\otimes V^*,\FF)
 \cong\Hom_{\cL}(V^*,U^*)$, where $U^\cL=\{u\in U:x.u=0\ \forall x\in\cL\}$.)

Up to scalars, the unique $\frsp(W,\baup)$-invariant linear map 
$T\otimes T\rightarrow \frsp(W,\baup)$, 
$x\otimes y\mapsto d_{x,y}$, is given by the formula
\begin{equation}\label{eq:F4dxy}
\trace(fd_{x,y})=-2(f.x\vert y)
\end{equation}
for any $f\in \frsp(W,\baup)$, where $f.x$ denotes the action of 
$f\in\frsp(W,\baup)$ on $x\in T$.

Consider the triple product $[x,y,z]=d_{x,y}.z$ and the alternating form 
$(\cdot\vert\cdot)$ on $T$.
By the uniqueness, up to scalars, given by Corollary 
\ref{co:Z2_STS}, there is a scalar $0\neq\alpha\in\FF$ such that
\begin{equation}\label{eq:F4_alpha}
d_{x,y}.z-d_{x,z}.y=\alpha\bigl(( x\vert z) y 
-( x\vert y) z+2( y\vert z) x\bigr).
\end{equation}
If we prove that $\alpha=1$, then we will have proved that $\bigl(T,[\cdot,\cdot,\cdot],(\cdot\vert\cdot)\bigr)$ is a symplectic triple system.

For $x=u_1\wedge u_2\wedge u_3=y$, and $z=v_1\wedge v_2\wedge v_3$, which belong to $T$,
we have $(x\vert y)=0$, and 
$d_{x,y}=0$, since $2(\epsilon_1+\epsilon_2+\epsilon_3)\notin\Phi$.
 Besides $d_{x,z}$ belongs to the weight space $\frsp(W,\baup)_0$, 
 so
$d_{x,z}= \sum_{i=1}^3\alpha_i\gamma_{u_i,v_i}$, for some scalars 
$\alpha_i\in\FF$. Note that $\gamma_{u_1,v_1}$ takes $u_1$ to 
$-u_1$, leaves $v_1$ fixed and vanishes on the remaining basic elements. Hence
\[
\trace\bigl(\gamma_{u_1,v_1}d_{x,z}\bigr)=
 \alpha_1\trace\bigl(\gamma_{u_1,v_1}^2\bigr)=2\alpha_1
\]
is equal to
\[
\begin{split}
-2(\gamma_{u_1,v_1}.x\vert z)
  &=-2( \gamma_{u_1,v_1}.(u_1\wedge u_2\wedge u_3)\vert
          v_1\wedge v_2\wedge v_3)\\
 &= -2( -u_1\wedge u_2\wedge u_3\vert 
       v_1\wedge v_2\wedge v_3)=2.
\end{split}
\]
We conclude that $\alpha_1=1$, and analogously  $\alpha_2=\alpha_3=1$. 
Thus we get
\[
-d_{x,z}.y=-\Bigl(\sum_{i=1}^3\gamma_{u_i,v_i}\Bigr).
  (u_1\wedge u_2\wedge u_3)=-(-1-1-1)(u_1\wedge u_2\wedge u_3)=3y,
\]
while $(x\vert z) y-( x\vert y) z
+2( y\vert z) x=y+0+2x=3y$ too.  This shows that indeed $\alpha=1$.

 To summarize, $\bigl(T,[\cdot,\cdot,\cdot],(\cdot\vert\cdot)\bigr)$ is a simple symplectic triple system, 
for the alternating form $(\cdot\vert\cdot)$ defined in \eqref{eq:F4alternating_form} and the triple product $[x,y,z]=d_{x,y}.z$
for $d_{x,y}$ defined by
\eqref{eq:F4dxy}.
Its inner derivation algebra is isomorphic to $\frsp(W,\baup)\cong\frsp_6(\FF)$ and its standard 
enveloping algebra is isomorphic to the split exceptional simple Lie algebra of type $F_4$.

\medskip

\subsection{$E_6$-type}\label{ss:E6}\quad
This case has been considered in \cite[Lemma 6.45]{EKmon}. For the benefit of the reader, we give the details.

As for the $F_4$-type, let $W$ be a vector space of dimension $6$.
Fix  $\{e_1,e_2,e_3,e_4,e_5,e_6\}$ a basis of $W$. The subalgebra $\frh$ of
diagonal elements of $\frsl(W)$, relative to this basis, is a Cartan subalgebra, with root system $\Phi=\{\pm(\epsilon_i-\epsilon_j): 1\leq i<j\leq 6\}$. 
Here $\epsilon_i$ denotes the weight of $e_i$  under the natural $\frsl(W)$-action on $W$, for any $i=1,\dots,6$.
A system of simple roots   of $\Phi$  is given by $\Pi=\{\epsilon_i-\epsilon_{i+1}: 1\leq i\leq 5\}$.

The space $T=\bigwedge^3W$ is the irreducible module for $\frsl(W)$
with highest weight $\varpi_3=\epsilon_1+\epsilon_2+\epsilon_3$. 
The element $e_{123}\bydef e_1\wedge e_2\wedge e_3$ lies in the weight space $T_{\varpi_3}$. We will use the notation $e_{i_1\ldots i_r}=e_{i_1}\wedge\cdots\wedge e_{i_r}$ throughout.

Fix the determinant map $\det\colon \bigwedge^6W\rightarrow \FF$ given by $\det(e_{123456})=1$ and, as in \eqref{eq:F4alternating_form}, 
consider the non-degenerate alternating bilinear form $(\cdot\vert\cdot)$
 on $T=\bigwedge^3W$ given by 
\begin{equation}\label{eq:E6_alternating}
(x\vert y)=\det(x\wedge y).
\end{equation} 
This is
the unique, up to scalars, $\frsl(W)$-invariant bilinear form on $T$.

With the same arguments as for the $F_4$-type, the unique, up to scalars, $\frsl(W)$-invariant linear map 
$T\otimes T\rightarrow \frsl(W)$: $x\otimes y\mapsto d_{x,y}$, is given by the formula
\begin{equation}\label{eq:E6dxy}
\trace(fd_{x,y})=-2(f.x\vert y)\,.
\end{equation}
(There is a misprint in \cite[Lemma 6.45]{EKmon}, where $24$ appears instead of $2$ in the formula above.)

Again, as for the $F_4$-type, there is a scalar $0\neq \alpha\in\FF$
such that \eqref{eq:F4_alpha} holds. For $x=e_{123}=y$ and $z=e_{456}$ we have $d_{x,y}=0$, and $d_{x,z}$  is the endomorphism of $W$ with coordinate matrix (in the given basis) $\diag(-1,-1,-1,1,1,1)$ due to 
\eqref{eq:E6dxy}, so that $d_{x,z}.y=-3y$. Also 
$(x\vert y)=0$ and $(x\vert z)=1=(y\vert z)$. It follows that
$\alpha=1$.

Therefore $\left(T=\bigwedge^3W,[\cdot,\cdot,\cdot],(\cdot\vert\cdot)\right)$ is a simple symplectic triple system, with inner derivation algebra isomorphic to $\frsl(W)\cong\frsl_6(\FF)$ and standard 
enveloping algebra isomorphic to the split exceptional simple Lie algebra of type $E_6$.

\medskip

\subsection{$E_7$-type}\label{ss:E7}\quad
Let $W$ be a six-dimensional vector space and denote by $W^*$ its dual. Fix a basis $\{e_i: 1\leq i\leq 6\}$ of $W$, and consider its
dual basis $\{e^i: 1\leq i\leq 6\}$ in $W^*$. The direct sum $W\oplus W^*$ is endowed with the non-degenerate quadratic form 
given by $\qup(u+f)=f(u)$ for all $u\in W$ and $f\in W^*$.
The corresponding orthogonal Lie algebra $\frso(W\oplus W^*,\qup)$ is
the split simple Lie algebra of type $D_6$. The subalgebra $\frh$ of diagonal elements, relative to the basis $\{e_1,\ldots,e_6,e^1,\ldots,e^6\}$ of $W\oplus W^*$, is a Cartan subalgebra. The corresponding root system is $\Phi=\{\pm\epsilon_i\pm\epsilon_j:
1\leq i<j\leq 6\}$, where $\epsilon_i$ is the weight of $e_i$ in the natural action of $\frso(W\oplus W^*,\qup)$ on $W\oplus W^*$. A system of simple roots is $\Pi=\{\epsilon_1-\epsilon_2,
\epsilon_2-\epsilon_3,\epsilon_3-\epsilon_4,\epsilon_4-\epsilon_5,
\epsilon_5-\epsilon_6,\epsilon_5+\epsilon_6\}$. The fundamental dominant weight $\varpi_6$, relative to this system, is 
$\varpi_6=\frac12(\epsilon_1+\cdots+\epsilon_6)$.

  On the exterior algebra $\bigwedge W$ write, as above, 
  $ e_{i_1\cdots i_r}=e_{i_1}\wedge \cdots\wedge e_{i_r}$ if each $1\le i_j\le6$, and shorten it to $e_I=e_{i_1\cdots i_r}$ for any 
 sequence
$I=(i_1,\ldots,i_r)$ with $1\leq i_1<\cdots<i_r\leq 6$.    
 
Consider the linear form $\det\colon \bigwedge W\rightarrow \FF$ which is trivial on $\bigwedge^iW$ for $i<6$ and such that 
$\det(e_{123456})=1$. Also, let $s\mapsto \widehat{s}$ be
the involution of $\bigwedge W$ which is the identity on $W=\bigwedge^1W$, and hence, for $I=(i_1,\ldots,i_r)$ with $1\leq i_1<\cdots <i_r\leq 6$, 
$\widehat{e_I}=e_{i_r\cdots i_1}=(-1)^{\binom{r}{2}}e_I$.

Finally, consider the non-degenerate bilinear form 
\begin{equation}\label{eq:E7_b}
\begin{split}
\baup\colon  \bigwedge W\times\bigwedge W&\longrightarrow \FF\\
(s,t)\quad &\mapsto \ \det(\widehat{s}\wedge t).
\end{split}
\end{equation}
This is alternating, and the even and odd parts: $\bigwedge\subo W$ and $\bigwedge\subuno W$, are orthogonal.

Denote by $\tau_{\baup}$ the involution of the endomorphism algebra 
$\End_\FF\bigl(\bigwedge W\bigr)$ induced by $\baup$:
\[
\baup\bigr(\varphi(s),t\bigr)=\baup\bigr(s,\tau_{\baup}(\varphi)(t)\bigr),
\]
for all $s,t\in\bigwedge W$ and $\varphi\in
\End_\FF\bigl(\bigwedge W\bigr)$.

  Consider  the linear map $W\oplus W^*\rightarrow 
\End_\FF\bigl(\bigwedge W\bigr)$ that sends any $u\in W$ to the left
multiplication $l_u$ by $u$ on $\bigwedge W$, and any $f\in W^*$ to
the odd superderivation $\delta_f$ of $\bigwedge W$ such that
$\delta_f(u)=f(u)$ for all $u\in W$. (The fact that $\delta_f$ is an odd
superderivation means that 
$\delta_f(s\wedge t)=\delta_f(s)\wedge t+(-1)^{\deg(s)}s\wedge \delta_f(t)$ for any homogeneous (even or odd) elements $s,t$ of $\bigwedge W$.) 
Note that $l_u^2=0=\delta_f^2$, that 
$\delta_f(\widehat{s})=-(-1)^{\deg(s)}\widehat{\delta_f(s)}$, and that 
$\tau_{\baup}(l_u)=l_u$ and $\tau_{\baup}(\delta_f)=\delta_f$ for any
$u\in W$ and $f\in W^*$.
It follows that $l_u\delta_f+\delta_fl_u=f(u)\id$ for all $u\in W$ and
$f\in W^*$,  and hence this linear map extends to a homomorphism
of algebras with involution (see \cite[\S 8]{KMRT} or \cite{Eld07}): 
\begin{equation}\label{eq:E7_Lambda}
\Lambda\colon  \bigl(\Cl(W\oplus W^*,\qup),\tau\bigr)\longrightarrow
\bigl(\End_\FF(\bigwedge W),\tau_{\baup}\bigr),
\end{equation}
where $\Cl(W\oplus W^*,\qup)$ is the Clifford algebra associated to the
quadratic form $\qup$, and $\tau$ is its canonical involution, that is, it restricts to
the identity on $W\oplus W^*$. By dimension count, $\Lambda$ is an 
isomorphism. Moreover, $\Lambda$ restricts to an isomorphism, also denoted by 
$\Lambda$, of the 
even subalgebras:
\[ 
\textstyle{\Lambda\colon  \bigl(\Cl\subo(W\oplus W^*,\qup),\tau\bigr)\longrightarrow
\bigl(\End_\FF(\bigwedge\subo W)\times 
\End_\FF(\bigwedge\subuno W),\tau_{\baup}\bigr).}
\]

The orthogonal Lie algebra $\frso(W\oplus W^*,\qup)$ lives inside
the even Clifford algebra $\Cl\subo(W\oplus W^*,\qup)$: The endomorphism $\sigma_{x,y}:z\mapsto \qup(x,z)y-\qup(y,z)x$ in
$\frso(W\oplus W^*,\qup)$ corresponds to $\frac{ -1}2(xy-yx)$ in 
$\Cl\subo(W\oplus W^*,\qup)$, for all $x,y\in W\oplus W^*$, because
in the Clifford algebra we have $[[x,y],z]=xyz-yxz-zxy+zyx=x(yz+zy)-y(xz+zx)-(zx+xz)y+(zy+yz)x
=2\qup(y,z)x-2\qup(x,z)x=-2\sigma_{x,y}(z)$.
(Throughout this paper   the product in the Clifford algebra is denoted by juxtaposition, and the polar
form of $\qup$ is defined by $\qup(x,y)=\qup(x+y)-\qup(x)-\qup(y)$, which
becomes $xy+yx$  in the Clifford algebra.) Note that the element 
$\sigma_{e^i,e_i}\in\frso(W\oplus W^*\qup)$ satisfies $\sigma_{e^i,e_i}(e_i)=e_i$, 
 $\sigma_{e^i,e_i}(e_j)=0$,  $\sigma_{e^i,e_i}(e^i)=-e^i$, and  $\sigma_{e^i,e_i}(e^j)=0$
 for $j\neq i$.

 Take the half-spin module
$T=\bigwedge\subo W$  of $\frso(W\oplus W^*,\qup)$, that is, the
representation of $\frso(W\oplus W^*,\qup)$  obtained by first
embedding it on $\Cl\subo(W\oplus W^*,\qup)$ and then composing 
with $\Lambda$.  Recall, for any increasing sequence $J$,  that 
$\sigma_{e^i,e_i}.e_J=    
\frac{1}2\big(e_i\wedge \delta_{e^i}(e_J)-\delta_{e^i}(e_i\wedge e_J) \big)$ 
equals $\frac12e_J$ if $i\in J$  and $-\frac12e_J$ if $i\notin J$. Thus, for 
$I=(1,2,3,4,5,6)$, the element $e_I\in T_{\frac12(\epsilon_1+\cdots+\epsilon_6)}$ is a highest weight vector in $T$.
Thus $T$ is the irreducible module with highest weight $\varpi_6$. 
(Note that had we swapped the last two elements in our system of 
simple roots $\Pi$, we would have had to consider the other half-spin 
module: $\bigwedge\subuno W$, of highest weight $\frac12(\epsilon_1+\cdots+\epsilon_5-\epsilon_6)$, as  the irreducible module with highest weight $\varpi_6$.)

By irreducibility, the restriction of $\baup$  is the unique, up to scalars,
$\frso(W\oplus W^*,\qup)$-invariant non-zero bilinear form on $T$, and
 the unique, up to scalars,
$\frso(W\oplus W^*,\qup)$-invariant bilinear map 
$T\times T\rightarrow \frso(W\oplus W^*,\qup)$ is given by the formula
\begin{equation}\label{eq:E7dxy}
\trace\bigl(\sigma d_{x,y})=-4\baup(\sigma.x,y),
\end{equation}
for $\sigma\in \frso(W\oplus W^*,\qup)$ and $x,y\in T$. (See, e.g., \cite[Proposition 2.19]{Eld07}.)

Again, as for the $F_4$ or $E_6$ types, with $(x\vert y)=\baup(x,y)$, there is a scalar $0\neq \alpha\in\FF$
such that \eqref{eq:F4_alpha} holds. For $x=y=1(=e_\emptyset)$ and 
$z=e_{123456}$, we have $d_{x,y}=0$, while $d_{x,z}$, which belongs to the Cartan subalgebra $\frh$, coincides with  the endomorphism 
$\sum_{i=1}^6\sigma_{e^i,e_i}$  due to 
\eqref{eq:E7dxy}. 
   Thus 
$d_{x,z}.y= \sum_{i=1}^6\,\sigma_{e^i,e_i} .e_\emptyset
=\sum_{i=1}^6\frac{-1}2e_\emptyset=-3y  $. 
 Also 
$(x\vert y)=0$ and $(x\vert z)=1=(y\vert z)$. It follows that
$\alpha=1$.

Therefore, $\bigl(T=\bigwedge\subo W,[\cdot,\cdot,\cdot],(\cdot\vert\cdot)\bigr)$ is a simple symplectic triple system, with inner derivation algebra isomorphic to $\frso(W\oplus W^*,\qup)$ and standard 
enveloping algebra isomorphic to the split exceptional simple Lie algebra of type $E_7$.

\medskip

\subsection{$E_8$-type}\label{ss:E8}\quad
Let $U$ be an eight-dimensional vector space. Fix a basis 
$\{e_i: 1\leq i\leq 8\}$ and its dual basis 
$\{e^i: 1\leq i\leq 8\}$ in $U^*$. For any sequence 
$I=(i_1,\ldots,i_r)$ with $1\leq i_1<\cdots<i_r\leq 8$, write 
$e_I=e_{i_1}\wedge\cdots\wedge e_{i_r}$ in $\bigwedge U$, and
$e^I=e^{i_1}\wedge \cdots\wedge e^{i_r}$ in $\bigwedge U^*$.

Consider the linear map $\det\colon \bigwedge U\rightarrow \FF$ that 
annihilates $\bigwedge^i U$, $0\leq i\leq 7$, and such that 
$\det(e_{12345678})=1$. It induces a non-degenerate bilinear form
$(\cdot\vert\cdot)_\wedge$ on $\bigwedge U$ by means of 
$(x\vert y)_\wedge=\det(x\wedge y)$ for all $x,y$.  Its restriction to 
$\bigwedge^4U$ is a (non-degenerate) symmetric bilinear form.

Using these ingredients, Adams gives in \cite[Chapter 12]{Adams} the following explicit construction of the  split simple Lie algebra of type $E_7$
and of its irreducible $56$-dimensional representation. This  $56$-dimensional 
module  is the vector space on which
   the last split simple symplectic triple system, with $E_7$ as its inner derivation algebra, is built. More precisely, 
the vector space
$\cL=\frsl(U)\oplus\bigwedge^4U$ is the split simple Lie algebra of type $E_7$, 
with bracket given by
\begin{itemize}
\item 
the usual bracket in $\frsl(U)$,
\item 
$[f,x]=f.x$, the natural action of $f\in\frsl(U)$ on 
$x\in\bigwedge^4 U$,
\item 
for $x,y\in\bigwedge^4U$, $[x,y]$ is the element in $\frsl(U)$ determined by the condition 
\[
\trace(f[x,y])=(f.x\vert y)_\wedge
\] 
for all $f\in\frsl(U)$.
\end{itemize}
The decomposition
$\cL=\frsl(U)\oplus\bigwedge^4U$ is a grading by $\ZZ/2$.

The Killing form of $\cL$ is $36(\cdot\vert\cdot)_{\cL}$, where
$(\frsl(U)\vert\bigwedge^4U)_{\cL}=0$, 
$(f\vert g)_{\cL}=\trace(fg)$, and 
$(x\vert y)_{\cL}=(x\vert y)_\wedge\bigl(=\det(x\wedge y)\bigr)$, for 
$f,g\in\frsl(U)$ and $x,y\in\bigwedge^4 U$.  

Moreover, $(\cdot\vert\cdot)_\wedge$ gives $\frsl(U)$-invariant linear maps, for $0\leq i\leq 8$:
\begin{equation}\label{eq:E8Phi}
\begin{split}
\textstyle{\Phi_i\colon \bigwedge^i U}&\textstyle{\longrightarrow 
     \left(\bigwedge^{8-i}U\right)^*\cong\bigwedge^{8-i}U^*}\\
x\ &\mapsto\quad (x\vert\cdot)_\wedge,
\end{split}
\end{equation}
where $\bigwedge^iU^*$ is identified with 
$\left(\bigwedge^iU\right)^*$ naturally:
\begin{equation}\label{eq:exterior_dual}
f_{1}\wedge\cdots\wedge f_i\leftrightarrow
\Bigl(v_1\wedge\dots \wedge v_i\mapsto \det\bigl(f_k(v_j)\bigr)_{k,j}\Bigr).
\end{equation} 
Finally, take $T=\bigwedge^2U\oplus\bigwedge^2U^*$, with the action of 
$\cL$ given by:
\begin{itemize}
\item the natural action of $\frsl(U)$ on both $\bigwedge^2U$ and
$\bigwedge^2U^*$,

\item $[x,p]=\Phi_6(x\wedge p)\in\bigwedge^2U^*$, for $x\in\bigwedge^4U$ and
$p\in\bigwedge^2U$,

\item $[x,q]=\Phi_2^{-1}\bigl(\Phi_4(x)\wedge q\bigr)\in\bigwedge^2U$, for $x\in \bigwedge^4U$ and $q\in\bigwedge^2U^*$.
\end{itemize}
Then $T$ is the only $56$-dimensional irreducible module for $\cL$, that is, the 
irreducible module with highest weight $\varpi_1$ (once a Cartan
subalgebra and a system of simple roots is chosen). 

\smallskip

Using these previous results from \cite{Adams}, we proceed as follows to determine the structure of simple symplectic triple system on the
irreducible $\cL$-module $T$.

Given two disjoint increasing sequences $I=(i_1,\ldots,i_r)$ and 
$J=(j_1,\ldots,j_s)$, its union $I\cup J$ gives another increasing 
sequence. Let $(-1)^{IJ}$ be $1$ or $-1$ according to the rule 
$e_Ie_J=(-1)^{IJ}e_{I\cup J}$ (that is, the sign of the permutation). Also denote by $\ol{I}$ be the 
increasing sequence whose underlying set is 
$\{1,\ldots,8\}\setminus I$. The size of a sequence $I$ will be denoted by
$\lvert I\rvert$. The isomorphisms $\Phi_i$ in \eqref{eq:E8Phi} are then given by $e_I\mapsto (-1)^{I\ol{I}}e^{\ol{I}}$  (where $i=\lvert I\rvert$).

With these notations, the action of $\bigwedge^4U$ on $T$ works
as follows, for $I$ and $J$ increasing sequences with $\lvert I\rvert=4$ and $\lvert J\rvert=2$,
\[
e_I.e_J=\begin{cases} 
	0&\text{if $I\cap J\neq \emptyset$,}\\
	(-1)^{IJ}(-1)^{(I\cup J)(\ol{I\cup J})}e^{\ol{I\cup J}}
	&\text{otherwise;}
\end{cases}
\]
and
\[
e_I.e^J=\begin{cases}
	0&\text{if $J\not\subseteq I$,}\\
	(-1)^{I\ol{I}}(-1)^{\ol{I}J}
	(-1)^{(\ol{I}\cup J)(\ol{\ol{I}\cup J})}
	e_{I\cap\ol{J}}&\text{otherwise.}
\end{cases}
\]
(Note that $\ol{\ol{I}\cup J}=I\cap\ol{J}$.)

Endow $T=\bigwedge^2U\oplus\bigwedge^2U^*$ with the 
non-degenerate alternating bilinear form $(\cdot\vert\cdot)$ such that $\bigwedge^2U$
and $\bigwedge^2U^*$ are totally isotropic subspaces, and such
that $(u_1\wedge u_2\vert\omega_1\wedge\omega_2)=
\det\bigl(\omega_i(u_j)\bigr)$. This alternating form is $\cL$-invariant.

Adams also shows in \cite[Theorem 12.4]{Adams} that if the map 
$\circ\colon T\times T\rightarrow \cL$, $(x,y)\mapsto x\circ y$ is defined by
\[
(l\vert x\circ y)_\cL=(l.x\vert y)
\]
for $l\in\cL$ and $x,y\in T$, then $x\circ y$ is symmetric and
\begin{itemize}
\item $x\circ y=-x\wedge y\in\bigwedge^4U$ for
 $x,y\in\bigwedge^2U$,
\item $x\circ y=\Phi_4^{-1}(x\wedge y)\in\bigwedge^4U$ for 
$x,y\in\bigwedge^2U^*$,
\item for $u_1,u_2\in U$ and $\omega_1,\omega_2\in U^*$,
\[
\begin{split}
(u_1\wedge u_2)\circ(\omega_1\wedge \omega_2)
	&=\omega_1(u_1)u_2\otimes\omega_2
		-\omega_1(u_2)u_1\otimes\omega_2
		-\omega_2(u_1)u_2\otimes\omega_1\\
&\qquad +\omega_2(u_2)   u_1\otimes\omega_1
		-\frac14\det\bigl(\omega_i(u_j)\bigr)\id_U\,\in\frsl(U),
\end{split}
\]
where $u\otimes\omega\in\frgl(U)$ denotes the map 
$u'\mapsto \omega(u')u$, for $u,u'\in U$ and $\omega\in U^*$.
\end{itemize}

Define $d_{x,y}$ for $x,y\in T$ by $d_{x,y}=-2x\circ y$, that is,
$d_{x,y}$ is defined by the equation
\begin{equation}\label{eq:E8split_dxy}
\trace(ld_{x,y})=-2(l.x\vert y)
\end{equation}
for $l\in\cL$ and $x,y\in T$. (Recall that $\trace(ld_{x,y})=(l\vert d_{x,y})_\cL$.)

As in the previous cases, $(\cdot\vert\cdot)$ is the unique, up to scalars, non-zero $\cL$-invariant bilinear form, and $d_{x,y}$ is the
unique, up to scalars, non-zero $\cL$-invariant bilinear map
$T\times T\rightarrow \cL$. Therefore there is a scalar
$0\neq \alpha\in\FF$ such that \eqref{eq:F4_alpha} holds. Let us find it.

Take $x=z=e_{12}$ and $y=e^{12}$. On one hand, $d_{x,z}=0$ and 
$d_{x,y}=-2\bigl(e_2\otimes e^2+e_1\otimes e^1-\frac14\id\bigr)$,
so $d_{x,y}.z=-3e_{12}$ (since $\id_U.z=2z$). On the other hand, $(x\vert z)=0$, 
$(x\vert y)=(z\vert y)=1$, so 
$(x\vert z)y-(x\vert y)z+2(y\vert z)x=-3e_{12}$. We thus conclude that $\alpha=1$, and hence,  with $[x,y,z]=d_{x,y}.z$, the triple
$\bigl(T=\bigwedge^2U\oplus\bigwedge^2U^*,[\cdot,\cdot,\cdot],(\cdot\vert\cdot)\bigr)$ is a simple
symplectic triple system with inner derivation algebra the split exceptional simple Lie algebra of type $E_7$ (isomorphic to $\cL=\frsl(U)\oplus\bigwedge^4U$),
and standard enveloping algebra isomorphic to the split exceptional
simple Lie algebra of type $E_8$.

\smallskip

\begin{remark}
It turns out \cite{Eld06} that this same classification and list of  split examples work, with minor modifications,  over fields of characteristic 
$\neq 2,3$.
\end{remark}

\medskip

%%%%%%%%%%%%%%%%%%%%%%%%%%%%%%%%%%%%%%%%%%%%%%%
%%%%%%%%%%%%%%%%%%%%%%%%%%%%%%%%%%%%%%%%%%%%%%%

\section{Non-split simple real symplectic triple 
systems}\label{se:nonsplit}

Our aim is to exhibit a list of some real non-split symplectic triple systems that  will be proved later on to exhaust all the possibilities. 

For convenience, first we will relate the signatures of the Killing forms of the Lie algebras $ \g(T)$ and $ \inder(T)$ associated to a real simple symplectic triple system $\bigl(T,[\cdot,\cdot,\cdot],(\cdot\vert\cdot )\bigr)$. 

\subsection{Some remarks on signatures}\label{ss:signature}

Let us denote by $\sig(\mathfrak{a})$   the signature of the Killing form 
$\kappa_{\mathfrak{a}}$ of any semisimple Lie algebra $\mathfrak{a}$.
The signature is understood here as the difference between the number of 
elements in an orthogonal basis with positive `norm' and the number of them
with negative `norm'.

\begin{lemma}\label{co_sigparteimpar}
Let $T$ be a real simple symplectic triple system.
If   $\kappa$ denotes the Killing form of $ \g(T)$, then
there exists a skew-symmetric bilinear map $\eta\colon T\times T\to\RR$  such that
\begin{equation}\label{eq_signneutra}
\kappa(a\otimes x,b\otimes y)=\langle a\vert b\rangle \eta(x,y),
\end{equation}
for any $a,b\in V$ and $x,y\in T$.
In particular, the signature of the restriction 
$\kappa\vert_{\g(T)_{\bar1}}$ is zero.
\end{lemma}

\begin{proof} 
First note that $\dim_\RR\Hom_{\spf(V)}(V\otimes V,\RR)$ equals $1$, because this
is the case after complexification:
 $\dim_\CC\Hom_{\slf_2(\CC)}( \CC^2\otimes\CC^2,\CC)=
 \dim_\CC\Hom_{\slf_2(\CC)}(\mathfrak{gl}_2(\CC),  \CC)=1$. As 
 $\langle\cdot\vert\cdot\rangle\colon V\times V\to\RR$ is a non-zero 
 $\spf(V)$-invariant map, then Lemma~\ref{le_Vinberg} gives the existence of a bilinear map $\eta\colon T\times T\to\RR$ satisfying
\eqref{eq_signneutra}. But it is easy to check that $\eta$ must be $\inder(T)$-invariant, by using the invariance of $\kappa$. Also, as $\kappa$ is symmetric and 
$\langle\cdot\vert\cdot\rangle$ is skew-symmetric, then $\eta$ is skew-symmetric, and we may think of it as an element in
  $\Hom_{\inder(T)}(\bigwedge^2T,\RR)$.  
  
This space of homomorphisms has dimension $1$  and it is spanned by 
$(\cdot\vert\cdot)$. Indeed,   each $\eta'\in\Hom_{\inder(T)}(\bigwedge^2T,\RR)$ would give $\eta'\otimes \langle\cdot\vert\cdot\rangle$, which would belong to  
$\Hom_{\g\subo}(S^2(\g{\subuno}),\RR)$, and this is one-dimensional, because so is its complexification according to part (3) in Lemma~\ref{le:Z2}.

Then there is a scalar $\alpha\in\RR$ such that 
$\kappa(a\otimes x,b\otimes y)
=\alpha\langle a\vert b\rangle (x\vert y)$.
Now, if $\{x_i,y_i:i=1,\dots,n\}$ is a symplectic basis of $\bigl(T,(\cdot\vert \cdot)\bigr)$, that is, 
$(x_i\vert y_i)=1=-(y_i\vert x_i)$, and $(x_i\vert y_j)=0=(y_j\vert x_i)$ for $i\neq j$, and if 
$\{e_1,e_2\}$ is a symplectic basis of $V$, then the family
$$
\{e_1\otimes x_i+e_2\otimes y_i,e_1\otimes y_i-e_2\otimes x_i,
e_1\otimes x_i-e_2\otimes y_i,e_1\otimes y_i+e_2\otimes x_i:i=1,\dots,n\}
$$ 
is a $\kappa$-orthogonal basis of $\g(T)_{\bar1}=V\otimes T$ such that the `length'
$\kappa(z,z)$ of any of the first $2n$ elements is $2\alpha$ and the `length' of any of the last $2n$ elements is $-2\alpha$. Therefore, the signature of 
$\kappa\vert_{\g(T)_{\bar1}}$ is $0$.
 \end{proof}
 
 A classical result will be useful for us too.
 
 \begin{lemma}
 If $L\subset\mathfrak{gl}(U)$ is a complex simple Lie algebra and $\kappa\colon L\times L\to\CC$ denotes its Killing form, then there is a scalar 
 $\alpha\in\mathbb Q$, $\alpha>0$,
 such that $\kappa(f,g)=\alpha\tr(fg)$ for all $f,g\in L$.
 \end{lemma}
 
 \begin{proof}
 The vector space  $\Hom_L(S^2(L),\CC)$ is contained in 
 $\Hom_L(L\otimes_\CC L,\CC)$, which is linearly isomorphic to $\Hom_L(L,L^*)$, and hence to $\Hom_L(L,L)=\CC\id_L$, because the adjoint module $L$ is  self-dual  
 ($\kappa $ is non-degenerate).  We conclude that $\Hom_L(S^2(L),\CC)$ has dimension one and it is spanned by the Killing form $\kappa$.
 
 The symmetric bilinear form 
 $\bup\colon L\times L\to\CC$, $\bup(f,g)=\tr(fg)$,
 is $L$-invariant, so there is a scalar $\alpha\in\CC$ such that 
 $\kappa(f,g)=\alpha\tr(fg)$ for all $f,g\in L$. 
 Take now a subalgebra of $L$ isomorphic to $\slf_2(\CC)$ and let $\{h,e,f\}$
 be  a standard basis.  As $L$ is a completely reducible module, it is a sum of 
 $\slf_2(\CC)$-modules of type $V(n \varpi_1)$ so that $\ad h$ acts diagonally on 
 $L$ and all its eigenvalues are integer numbers. This means that $\kappa(h,h)=\tr(\ad^2h)$ is in $\ZZ_{\ge0}$, in fact it is strictly positive.
In the same vein, $U$ is completely reducible as a module for this subalgebra
isomorphic to  $\slf_2(\CC)$, and all the eigenvalues of 
$h\,\bigl(\in\mathfrak{gl}(U)\bigr)$ are integers, so we get  
$\bup(h,h)\in\ZZ_{\ge0}$ too. From $\kappa(h,h)=\alpha \bup(h,h)$ we obtain that the scalar $\alpha$ is a (positive) rational number.
 \end{proof}

 \begin{proposition}\label{pr_signat}
 Let $\bigl(T,[\cdot,\cdot,\cdot],(\cdot\vert\cdot )\bigr)$ be  a real simple symplectic triple system. Write
  $\g=\g(T)$ and $\frh=\inder(T)$. Then, if $\frh$ is simple, the difference between the signatures of
  $\frg$ and $\frh$ is $1$: $\sig(\g)-\sig(\frh )=1$.
\end{proposition}

\begin{proof}
Denote by $\kappa$ the Killing form of $\g$. Lemma~\ref{co_sigparteimpar} gives 
$\sig(\g)=\sig(\kappa\vert_{\g(T)_{\bar0}})$. Now note that 
$\g(T)_{\bar0}=\spf(V)\oplus\inder(T)$ is an orthogonal sum, so that 
$\sig(\g)=\sig(\kappa\vert_{ \spf(V)})+\sig(\kappa\vert_{ \inder(T)})$.

Let us prove first that $\kappa\vert_{ \spf(V)}$ is a positive multiple of
the Killing form $\kappa_{ \spf(V)}$, which has signature $1$ ($\slf_2(\RR)= \spf(V)$ is the split algebra of type $A_1$).
Indeed, if $\gamma,\gamma'\in\spf(V)$, then $[\gamma,[\gamma',a\otimes t]]=\gamma\gamma'(a)\otimes t$, so that
$$
\kappa(\gamma,\gamma')=\kappa_{\spf(V)}(\gamma,\gamma')+\tr(\gamma\gamma')\dim T.
$$
But $\tr(h^2)=2$, while $\kappa_{\spf(V)}(h,h)=8$, so that we have 
$\tr(\gamma\gamma')=\frac{1}{4}\kappa_{\spf(V)}(\gamma,\gamma')$, and this implies
 $\kappa(\gamma,\gamma')
 =\left(1+\frac{\dim T}{4}\right)\kappa_{\spf(V)}(\gamma,\gamma')$. In particular, we get 
 $\sig(\kappa\vert_{ \spf(V)})=1$. 

For $d,d'\in\inder(T)=\frh $, we get $[d,[d',a\otimes t]]=a\otimes dd'(t)$ and
$$
\kappa(d,d')=\kappa_\frh (d,d')+2\tr(dd').
$$
Our $\frh $ is simple so, as in the  lemma above, we have 
$\alpha\kappa_\frh (d,d')=\tr(dd')$ for a positive $\alpha\in\mathbb Q$, 
because the same happens after complexification. 
It follows that $\kappa(d,d')$ equals $(1+2\alpha)\kappa_\frh (d,d')$ for all $d,d'\in\frh$ and 
that $\sig(\kappa\vert_{ \frh})$ equals $\sig(\frh)$.
\end{proof}

\medskip
 
Let us now  describe some non-split real symplectic triple systems with classical enveloping algebra. Namely, we will consider one family of triple systems  with special envelope and another one with orthogonal envelope.

\subsection{Unitarian type}\label{ss:unitarian}

Let $W$ be a non-zero complex vector space endowed with a 
non-degenerate hermitian form $\hup_W\colon W\times W\to \CC$. Hence $\hup_W$ is linear in the 
first component, conjugate linear in the second, 
$\hup_W(x,y)=\overline{\hup_W(y,x)}$ for all $x,y\in W$, and the kernel
$\{x\in W: \hup_W(x,W)=0\}$ is trivial. Recall that there are always
orthogonal bases relative to $\hup_W$ and, as in Subsection \ref{ss:signature}, the signature of 
$\hup_W$ is defined as the difference between the number of elements $w$ in
any of these basis  with $\hup_W(w,w)>0$ and the number of those with $\hup_W(w,w)<0$.

For any $x,y\in W$, we may consider the real and imaginary parts:
\begin{equation}\label{eq:h_re_im}   
\hup_W(x,y)=\{x\vert y\}+\bi(x\vert y).
\end{equation}
Then $\{\cdot\vert \cdot\}$ is a non-degenerate symmetric real 
bilinear form
on $W$ (considered as a real vector space), and $(\cdot\vert \cdot)$
is a non-degenerate alternating real bilinear form.

The linear operators 
$$
\eta_{a,b}=\hup_W(\cdot,a)b-\hup_W(\cdot,b)a,
$$
 for $a,b\in W$, span the
unitary Lie algebra           
\[
\fru(W,\hup_W)=\{f\in\End_\CC(W): 
\hup_W\bigl(f(a),b\bigr)=-\hup_W\bigl(a,f(b)\bigr)\ \forall a,b\in W\}.
\]
\noindent Moreover,  we have
$\eta_{\alpha a,b}=\eta_{a,\overline{\alpha}b}$ and
$[\sigma,\eta_{a,b}]=\eta_{\sigma(a),b}+\eta_{a,\sigma(b)}$
for all $a,b\in W$, $\alpha\in\CC$ and $\sigma\in\fru(W,\hup_W)$.

Let now $U$ be a two-dimensional vector space over $\CC$ endowed
with a non-degenerate hermitian form $\hup_U\colon U\times U\to \CC$ of
signature $0$. Then there is a basis $\{u,v\}$ of $U$ with 
$\hup_U(u,u)=0=\hup_U(v,v)$ and $\hup_U(u,v)=\bi$. The special unitary Lie algebra
$\frsu(U,\hup_U)$ is $\RR$-spanned  by 
$\eta_{u,\bi u},\eta_{v,\bi v},\eta_{u,\bi v}$ (note that 
$\eta_{u,v}=-\bi\,\id$). It leaves invariant the two-dimensional
vector space $V=\RR u\oplus \RR v$, and it is isomorphic, by restriction
to $V$,  to $\frsl(V)$.
 Indeed, for
 any $a,b\in V$, $\hup_U(a,b)=\bi \langle a\vert b\rangle$, where 
$\langle\cdot\vert\cdot\rangle$ is alternating with $\langle u\vert v\rangle =1$. For $a,b,c\in V$ we have:
\begin{equation}\label{eq:etaygamma}
\begin{split}
\eta_{a,\bi b}(c)&=\hup_U(c,a)\bi b-\hup_U(c,\bi b)a\\
	&=-\langle c\vert a\rangle b-\langle c\vert   b\rangle a\\
	&=\langle a\vert c\rangle b+\langle b\vert c\rangle a
	=\gamma_{a,b}(c).
\end{split}
\end{equation}
This means that we have the isomorphism 
$\frsu(U,\hup_U)\cong\frsp(V,\langle\cdot\vert\cdot\rangle)=\frsl(V)$.
Also, using that $\langle a\vert b\rangle c+\langle b\vert c\rangle a
+\langle c\vert a\rangle b=0$, because $\dim_\RR V=2$, we obtain
\[
\begin{split}
\eta_{a,b}(c)&=\hup_U(c,a)b-\hup_U(c,b)a\\
	&=\bi \langle c\vert a\rangle b-\bi\langle c\vert b\rangle a
	=-\bi\langle a\vert b\rangle c,
\end{split}
\]
so we get 
\begin{equation}\label{eq:eta_ab}
\eta_{a,b}=-\langle a\vert b\rangle L_\bi,
\end{equation} 
where $L_\bi$ denotes the scalar multiplication by the imaginary unit 
$\bi$.

\smallskip

The direct sum $U\oplus W$ is endowed with the non-degenerate
hermitian form given by the `orthogonal sum' $\hup=\hup_U\perp \hup_W$ of $\hup_U$ and $\hup_W$. 
 It is $\ZZ/2$-graded with even part $U$ and odd part $W$, which induces a $\ZZ/2$-grading on the endomorphism algebra and hence a $\ZZ/2$-grading on the  corresponding special unitary Lie algebra, where the even part consists of the endomorphisms preserving $U$ and $W$, and the odd part  consists of the endomorphisms interchanging $U$ with $W$. In this way, such
  special unitary Lie algebra 
decomposes, with the natural embeddings, as follows:
\begin{equation}\label{eq:uUW}
\frsu(U\perp W,\hup)=\bigl(\frsu(U,\hup_U)\oplus\frsu(W,\hup_W)\oplus \RR J\bigr)
\oplus \eta_{U,W},
\end{equation}
where $J$ is the trace zero endomorphism determined by:
\[
J(x)=\begin{cases} 
	(\dim_\CC W)\bi x&\text{for $x\in U$,}\\
	-2\bi x&\text{for $x\in W$.}
\end{cases}
\]
It must be remarked here that the space $\eta_{U,W}$ is the same as
$\eta_{V,W}$ and this is isomorphic to $V\otimes_\RR W$
($v\otimes x\mapsto \eta_{v,x}$). Also, $\frsu(W,\hup_W)\oplus\RR J$ is isomorphic to the unitary Lie algebra $\fru(W,\hup_W)$, as this is the direct sum of $\frsu(W,\hup_W)$ and a one-dimensional center. Thus we can write,
using the natural identifications:
\begin{equation}\label{eq:uUW2}
\frsu(U\perp W,\hup)=\bigl(\frsp(V,\langle\cdot\vert\cdot\rangle)\oplus
 \fru(W,\hup_W)\bigr)\oplus \bigl(V\otimes_\RR W\bigr),
\end{equation}
where  the product of the even part with the odd one is given by the natural action.

For any $a,b\in V$ and $x,y\in W$, easy computations give:
\[
\begin{split}
[\eta_{a,x},\eta_{b,y}]
	&=\eta_{\eta_{a,x}(b),y}+\eta_{b,\eta_{a,x}(y)}\\
	&=\eta_{\hup_U(b,a)x,y}+\eta_{\hup_W(y,x)a,b}\\
	&=\Bigl(\langle a\vert b\rangle \eta_{x,\bi y}
		+\{x\vert y\}\eta_{a,b}\Bigr)
		+(x\vert y)\eta_{a,\bi b}\\
	&=\langle a\vert b\rangle\Bigl(\eta_{x,\bi y}
	   -\{x\vert y\}L_\bi\vert_U\Bigr)+(x\vert y)\gamma_{a,b},
\end{split}
\]
where we have used  \eqref{eq:h_re_im}, \eqref{eq:etaygamma}, and  \eqref{eq:eta_ab}, and where $L_\bi\vert_U$
denotes the endomorphism that is trivial on $W$ and equals the 
scalar multiplication by $\bi$ on $U$.

Note that the trace of $\eta_{x,\bi y}$ is obtained as follows:
\[
\trace(\eta_{x,\bi y})=\hup_W(\bi y,x)-\hup_W(x,\bi y)
=\bi\bigl(\hup_W(y,x)+\hup_W(x,y)\bigr)
=2\{x\vert y\}\bi,
\]
and hence $\eta_{x,\bi y}-\{x\vert y\}L_\bi\vert_U$ is traceless so belonging to 
$\frsu(W,\hup_W)\oplus \RR J$. 
Moreover, for $c\in V$ and $z\in W$, we have:
\[
\begin{split}
[\eta_{x,\bi y}-\{x\vert y\}L_\bi\vert_U,\eta_{c,z}]
	&=\eta_{c,\eta_{x,\bi y}(z)}-\{x\vert y\}\eta_{\bi c,z}\\
	&=\eta_{c,\eta_{x,\bi y}(z)+\bi\{x\vert y\}z}\\
	&=\eta_{c,\hup_W(z,x)\bi y-\hup_W(z,\bi y)x+\bi\{x\vert y\}z}
		=\eta_{c,[x,y,z]},
\end{split}
\]
for the triple product on $W$ defined by 
\begin{equation}\label{eq_triplecasounitario}
[x,y,z]=\bi\bigl(\hup(z,x)y+\hup(z,y)x+\{x\vert y\}z\bigr).
\end{equation}
Thus we have checked that  
$[\proj_{\frsu(W,\hup_W)\oplus \RR J}([\eta_{a,x},\eta_{b,y}]),\eta_{c,z}]=\eta_{\langle a\vert b\rangle c,[x,y,z]}$. (Here we denote by $\proj_{\frsu(W,\hup_W)\oplus \RR J}$ the projection onto $\frsu(W,\hup_W)\oplus \RR J$ according to the decomposition in \eqref{eq:uUW}.)

Then, Equations  \eqref{eq_sympproducto} and \eqref{eq:uUW2} gives at once that 
$\bigl( T=W,[\cdot,\cdot,\cdot],(\cdot\vert\cdot)\bigr)$ is a simple symplectic triple system, with inner derivation algebra isomorphic to
$\fru(W,\hup_W)$, and standard enveloping algebra isomorphic to 
$\frsu(U\perp W,\hup)$.

These symplectic triple systems will be said to be of \emph{unitarian type}.
The signatures of $\hup_W$ and $\hup$ coincide, because the signature of $\hup_U$ is $0$.
Hence we get the isomorphism
$(\g(T),\inder(T))\cong ( \mathfrak{su}_{p+1,n+1-p},\mathfrak{u}_{p,n-p}) $ for $n=\dim_\CC W$ and $p\geq \frac{n}{2}$.

\medskip

\subsection{Quaternionic type}\label{ss:quaternionic}\quad

Let here $U$ be a two-dimensional right $\HH$-module, where 
$\HH=\RR\oplus\RR\bi\oplus\RR\bj\oplus\RR\bk$ is the real division
algebra of quaternions, endowed with a skew-hermitian non-degenerate
form $\hup_U\colon U\times U\to \HH$. That is, $\hup_U$ is $\RR$-bilinear,
$\hup_U(a,bq)=\hup_U(a,b)q$,  $\hup_U(aq,b)=\overline{q} \hup_U(a,b)$, and 
$\hup_U(a,b)=-\overline{\hup_U(b,a)}$, for all $a,b\in U$
and $q\in \HH$, where $q\mapsto \overline{q}$ is the canonical 
conjugation in $\HH$. Note that non-degenerate skew-hermitian forms
are unique up to `isometry'.   
Denote by $\HH_0$ the subspace
of zero trace   quaternions: $\HH_0=\RR\bi\oplus\RR\bj\oplus\RR\bk$.

Pick a basis $\{u,v\}$ of $U$ with $\hup_U(u,u)=0=\hup_U(v,v)$, $\hup_U(u,v)=1$ and,
as in the unitarian case, consider the two-dimensional vector space
$V=\RR u\oplus\RR v$. The restriction of $\hup_U$ to $V$ is a non-zero
alternating bilinear form $\langle\cdot\vert\cdot\rangle$.

Denote by $\frso^*(U,\hup_U)$ the Lie algebra of
skew-symmetric endomorphisms relative to $\hup_U$:
\[
\frso^*(U,\hup_U)=\{f\in\End_\HH(U): 
\hup_U\bigl(f(a),b\bigr)=-\hup_U\bigl(a,f(b)\bigr)\ \forall a,b\in U\}.
\]
Note that   $\frso^*(U,\hup_U)$ is the ($\RR$-)linear span of the operators
\begin{equation}\label{eq_sigmas}
\sigma_{a,b}=a\,\hup_U(b,\cdot)+b\,\hup_U(a,\cdot),
\end{equation}
 for $a,b\in U$.

\begin{lemma}\label{le:so*H}
The Lie algebra $\frso^*(U,\hup_U)$
 is the direct sum of two three-dimensional simple ideals:
\[
\frso^*(U,\hup_U)=\sigma_{V,V}\oplus\sigma_{u,v\HH_0}.
\]
Moreover, $\sigma_{V,V}$ is isomorphic to the split simple Lie algebra 
$\frsl_2(\RR)$, while $\sigma_{u,v\HH_0}$ is isomorphic to the 
compact Lie algebra $\frsu_2$.
\end{lemma}
\begin{proof}
For any $f\in\frso^*(U,\hup_U)$, $a,b\in U$ and $q\in\HH$, we have 
$[f,\sigma_{a,b}]=\sigma_{f(a),b}+\sigma_{a,f(b)}$ and 
$\sigma_{aq,b}=\sigma_{a,b\overline{q}}$. 
Also, for $a\in U$ and
$q\in \HH_0$:
\[
\sigma_{a,aq}=\begin{cases}
	\sigma_{a\overline{q},a}=-\sigma_{aq,a}&\text{as $\overline{q}=-q$,}\\
	\sigma_{aq,a}&\text{as $\sigma$ is symmetric on its arguments.}
\end{cases}
\] 
Hence $\sigma_{a,aq}=0$. 
From here it is easy to obtain  
\begin{equation}\label{eq:sigma_abq}
\sigma_{a,bq}=\langle a\vert b\rangle\sigma_{u,vq}
\end{equation}
for any $a,b\in V$ and $q\in\HH_0$. We then get
\[
\frso^*(U,\hup_U)=\sigma_{U,U}=\sigma_{V,V\HH}=
\sigma_{V,V}+\sigma_{V,V\HH_0}=\sigma_{V,V}+\sigma_{u,v\HH_0}.
\]

For $a,b\in V$, we have $\sigma_{a,b}\vert_V=\gamma_{a,b}\in\frsp\bigl(V,\langle\cdot\vert\cdot\rangle\bigr)$, and hence $\sigma_{V,V}$ is 
a Lie subalgebra isomorphic to 
$\frsp\bigl(V,\langle\cdot\vert\cdot\rangle\bigr)\cong\frsl_2(\RR)$.

For $a,b\in V$ and $q,q'\in\HH_0$ we compute:
\begin{equation}\label{eq:sigmas}
\begin{split}
[\sigma_{a,b},\sigma_{u,vq}]&=
	\sigma_{\sigma_{a,b}(u),vq}+\sigma_{u,\sigma_{a,b}(v)q}\\
	&=\Bigl(\langle\sigma_{a,b}(u)\vert v\rangle +
	  \langle u\vert\sigma_{a,b}(v)\rangle\Bigr)\sigma_{u,vq}=0,
					\\[4pt]
\sigma_{u,vq}(u)&=u\hup_U(vq,u)=-u\overline{q}=uq,\\
\sigma_{u,vq}(v)&=vq\hup_U(u,v)=vq,\\[4pt]
[\sigma_{u,vq},\sigma_{u,vq'}]&=
	\sigma_{\sigma_{u,vq}(u),vq'}+\sigma_{u,\sigma_{u,vq}(v)q'}\\
	&=\sigma_{uq,vq'}+\sigma_{u,vqq'}=\sigma_{u,v[q,q']}.
\end{split}
\end{equation}
This shows that $\sigma_{V,V}$ and $\sigma_{u,v\HH_0}$ are two
ideals of $\frso^*(U,\hup_U)$, and that $\sigma_{u,v\HH_0}$ is 
isomorphic to the compact Lie algebra $\HH_0$ (with the usual bracket).
\end{proof}

Let now $W$ be a non-zero right $\HH$-module endowed with a 
non-degenerate skew-hermitian form $\hup_W\colon W\times W\to\HH$. As for
the unitarian type, the direct sum $U\oplus W$ is endowed with the
non-degenerate skew-hermitian form $\hup=\hup_U\perp \hup_W$ and there
appears a natural $\ZZ/2$-grading on the Lie algebra 
$\frso^*(U\perp W,\hup)$,  where the endomorphisms in the even part preserve the subspaces $U$ and $W$, and those in the odd part swap them: 
\begin{equation}\label{eq:Z2_quaternionic}
\begin{split}
\frso^*(U\perp W,\hup)
 &=\bigl(\frso^*(U,\hup_U)\oplus\frso^*(W,\hup_W)\bigr)\oplus\sigma_{U,W}\\
 &=\bigl(\sigma_{V,V}\oplus\bigl(\sigma_{u,v\HH_0}
   \oplus\sigma_{W,W}\bigr)\bigr)    \oplus\sigma_{V,W}.
\end{split}
\end{equation}
  (The notation used  here for $\sigma_{x,y}$ when $x,y\in W$, extends  the one in \eqref{eq_sigmas}.)
Besides, $\sigma_{U,W}=\sigma_{V,W}$ is isomorphic to 
$V\otimes_\RR W$ by means of $a\otimes x\mapsto \sigma_{a,x}$.

For any $x,y\in W$, consider the real and imaginary parts: 
\begin{equation}\label{eq:formacaso_quaternionic}
\hup_W(x,y)=(x\vert y)+\{x\vert y\}, 
\end{equation}
with $(x\vert y)\in\RR$ and $\{x\vert y\}\in\HH_0$. Then 
$(\cdot\vert\cdot)$ is a non-degenerate   alternating (real)  bilinear form on 
$W$, while $\{\cdot\vert\cdot\}\colon W\times W\to\HH_0$ is a symmetric (real) bilinear map.

For $a,b\in V$ and $x,y\in W$ we compute, using \eqref{eq:sigma_abq}:
\[
\begin{split}
[\sigma_{a,x},\sigma_{b,y}]&=
	\sigma_{\sigma_{a,x}(b),y}+\sigma_{b,\sigma_{a,x}(y)}\\
	&=\sigma_{\langle a\vert b\rangle x,y}+\sigma_{b,a\hup_W(x,y)}\\
	&=\langle a\vert b\rangle\sigma_{x,y}+(x\vert y)\sigma_{a,b}
		+\sigma_{b,a\{x\vert y\}}\\
	&=\langle a\vert b\rangle\bigl(\sigma_{x,y}
	-\sigma_{u,v\{x\vert y\}}\bigr)+(x\vert y)\sigma_{a,b}.
\end{split}
\]
Moreover, for $a\in V$ and $q\in\HH_0$, \eqref{eq:sigmas} gives
$\sigma_{u,vq}(a)=aq$, and hence    
\[
[\sigma_{u,vq},\sigma_{a,z}]=\sigma_{\sigma_{u,vq}(a),z}
=\sigma_{aq,z}=-\sigma_{a,zq},
\]
for any $z\in W$. It follows, for $x,y,z\in W$ and $a\in V$, that:
\[
[\sigma_{x,y}-\sigma_{u,v\{x\vert y\}},\sigma_{a,z}]
 =\sigma_{a,\sigma_{x,y}(z)+z\{x\vert y\}}=\sigma_{a,[x,y,z]},
\]
 for the triple product in $W$ defined by 
\begin{equation}\label{eq:triplecaso_quaternionic}
[x,y,z]=\sigma_{x,y}(z)+z\{x\vert y\}.
\end{equation}

Then Equation \eqref{eq:Z2_quaternionic} gives at once that
$\bigl(W,[\cdot,\cdot,\cdot],(\cdot\vert\cdot)\bigr)$ is a simple 
symplectic triple system, with inner derivation algebra isomorphic to
$\frsu_2\oplus\frso^*(W,\hup_W)$, and standard enveloping algebra isomorphic to 
$\frso^*(U\perp W,\hup)$.

These symplectic triple systems will be said to be of \emph{quaternionic type}.

\bigskip

Let us proceed now with examples with exceptional envelopes. Our constructions are much in the spirit of \cite{Adams} and provide
models of some of the exceptional simple real Lie algebras.

\subsection{Non-split $E_6$ types}\label{ss:E6nonsplit}\quad

As in  Subsection~\ref{ss:E6}, let $W$ be a six-dimensional vector space,
now over $\CC$ and endowed with a non-degenerate hermitian form
$\hup\colon W\times W\rightarrow \CC$. Changing $\hup$ by $-\hup$ if necessary, we will assume that the signature  of $\hup$ is $\geq 0$. Fix then an orthogonal basis $\{e_i: 1\leq i\leq 6\}$ of $W$, with $\hup(e_i,e_i)=1$ for 
$1\leq i\leq p$, and $\hup(e_i,e_i)=-1$ for $p+1\leq i\leq 6$, with $p\geq 3$. 
Use this basis to define the determinant map $\det\colon \bigwedge^6W\rightarrow \CC$ given by $\det(e_{123456})=1$, and consider the alternating 
bilinear form $(\cdot\vert\cdot)$ on $\bigwedge^3W$ given by \eqref{eq:E6_alternating}: $(x\vert y)=\det(x\wedge y)$. Consider
too the dual basis $\{e^i: 1\leq i\leq 6\}$ in $W^*$.

     Analogously to \eqref{eq:E8Phi}, this induces a $\frsl(W)$-invariant linear
isomorphism $\bigwedge^3W\rightarrow \bigl(\bigwedge^3W\bigr)^*$,
$x\mapsto (x\vert\cdot)$. As $\bigwedge^3 W^*$ is canonically
isomorphic, similarly to 
\eqref{eq:exterior_dual}, to $\bigl(\bigwedge^3W\bigr)^*$, we obtain
the following $\frsl(W)$-invariant linear isomorphism
\[
\begin{split}
\Phi_3\colon \textstyle{\bigwedge^3W}&\longrightarrow 
\textstyle{\bigwedge^3W^*}\\
  e_{\sigma(1)\sigma(2)\sigma(3)}&\mapsto
 (-1)^\sigma e^{\sigma(4)\sigma(5)\sigma(6)}
\end{split}
\]
for any permutation $\sigma$ of $\{1,\ldots,6\}$, where $(-1)^\sigma$
denotes the signature of $\sigma$.

On the other hand, consider the conjugate linear $\frsu(W,\hup)$-invariant
map $W\to W^*$: $u\mapsto \hup(\cdot,u)$, which induces the following 
 conjugate linear $\frsu(W,\hup)$-invariant map:
\[
\begin{split}
\Psi\colon \textstyle{\bigwedge^3W}&\longrightarrow 
\textstyle{\bigwedge^3W^*}\\
  u_1\wedge u_2\wedge u_3&\mapsto
 \hup(\cdot,u_1)\wedge \hup(\cdot,u_2)\wedge \hup(\cdot,u_3).
\end{split}
\]
For any $1\leq i<j<k\leq 6$ one has
\[
\Psi(e_{ijk})=(-1)^{\lvert \{i,j,k\}{\az\cap}\{p+1,\ldots,6\}\rvert}e^{ijk}.
\]

The composition $\Gamma=\Phi_3^{-1}\Psi\colon \bigwedge^3W
\to\bigwedge^3W$ is bijective, conjugate linear, and 
$\frsu(W,\hup)$-invariant. For any permutation $\sigma$, we get
\[
\Gamma(e_{\sigma(1)\sigma(2)\sigma(3)})=
-(-1)^\sigma(-1)^{\lvert \{\sigma(1),\sigma(2),\sigma(3)\}
\cap\{p+1,\ldots,6\}\rvert}e_{\sigma(4)\sigma(5)\sigma(6)}.
\]
Then $\Gamma^2$ is $\CC$-linear and $\frsu(W,\hup)$-invariant, so it is
$\frsl(W)$-invariant. Schur's Lemma shows then that $\Gamma^2$ is a scalar multiple of the identity. But we have:
\[
\begin{split}
\Gamma(e_{123})&=-e_{456},\\
\Gamma(e_{456})
	&=(-1)^{\lvert\{4,5,6\}\cap\{p+1,\ldots,6\}\rvert}e_{123},
\end{split}
\]
which gives:
\[
\Gamma^2=\begin{cases}
\id&\text{if $p=3$ or $5$,}\\
-\id&\text{if $p=4$ or $6$.}
\end{cases}
\]

For the rest of the subsection, we assume $\Gamma^2=\id$, that is,
$p=3$ or $5$.

As $\Gamma^2=\id$, the subspace of fixed elements, 
$\bigl(\bigwedge^3W\bigr)^\Gamma$, is an irreducible 
$\frsu(W,\hup)$-module, and 
$\bigwedge^3W=\bigl(\bigwedge^3W\bigr)^\Gamma\oplus\bi
\bigl(\bigwedge^3W\bigr)^\Gamma$. 
As   
$$
\textstyle{\dim_\CC\Hom_{\frsl(W)}\bigl(\bigwedge^3W\otimes_\CC\bigwedge^3W,\CC\bigr)=1,}
$$
this complex vector space is spanned by the map $x\otimes y\mapsto (x\vert y)$. Hence   the (real) dimension of 
$\Hom_{\frsu(W,\hup)}\bigl(\bigl(\bigwedge^3W\bigr)^\Gamma
\otimes_\RR \bigl(\bigwedge^3W\bigr)^\Gamma,\RR\bigr)$ is also $1$ and
there is a complex number $\alpha$ such that $(x\vert y)\in \RR\alpha$
for any $x,y\in \bigl(\bigwedge^3W\bigr)^\Gamma$. 
The elements $e_{123}-e_{456}$ and $\bi(e_{123}+e_{456})$ are in
$\bigl(\bigwedge^3W\bigr)^\Gamma$, and
\[
\bigl(e_{123}-e_{456}\vert\bi(e_{123}+e_{456})\bigr)
  =2\bi.
\]
We conclude that the restriction of $(\cdot\vert\cdot)$ to
$\bigl(\bigwedge^3W\bigr)^\Gamma$ takes values in $\RR\bi$. In 
other words, the map $x\otimes y\mapsto \bi(x\vert y)$ spans
$\Hom_{\frsu(W,\hup)}\bigl(\bigl(\bigwedge^3W\bigr)^\Gamma
\otimes_\RR \bigl(\bigwedge^3W\bigr)^\Gamma,\RR\bigr)$.

Also, we know from Subsection \ref{ss:E6} that the complex vector space
$\Hom_{\frsl(W)}\bigl(\bigwedge^3W\otimes_\CC\bigwedge^3W,\frsl(W)\bigr)$ is one-dimensional and spanned by the map 
$x\otimes y\mapsto d_{x,y}$ in \eqref{eq:E6dxy}. 
Therefore, the real vector
space $\Hom_{\frsu(W,\hup)}\bigl(\bigl(\bigwedge^3W\bigr)^\Gamma
\otimes_\RR \bigl(\bigwedge^3W\bigr)^\Gamma, \frsu(W,\hup)\bigr)$ is
one-dimensional too. We conclude that there is a complex number 
$\beta$ such that the $d_{x,y}\in\beta\frsu(W,\hup)$ for all 
$x,y\in\bigl(\bigwedge^3W\bigr)^\Gamma$. Equation~\eqref{eq_dos} then shows that $d_{x,y}\in\bi\,\frsu(W,\hup)$ for any 
$x,y\in \bigl(\bigwedge^3W\bigr)^\Gamma$.

Therefore, $\Bigl(\bigl(\bigwedge^3W\bigr)^\Gamma,\bi[\cdot,\cdot,\cdot],
\bi(\cdot\vert\cdot)\Bigr)$ is a real simple symplectic triple system, with inner derivation algebra isomorphic to $\frsu(W,\hup)$.

In conclusion, we have constructed real simple symplectic triple systems with inner derivation algebras isomorphic to $\frsu_{3,3}$ and
$\frsu_{5,1}$,  depending on the signature of the hermitian form $\hup$ being  $0$ or $4$, respectively, and with enveloping algebras necessarily real forms of the Lie algebra of type $E_6$.
Proposition~\ref{pr_signat} shows that the signature of the Killing form of the 
enveloping algebra of a symplectic triple system is one plus the signature of the Killing form of its inner derivation algebra. The signature of the Killing form of $\frsu_{p,q}$ is $1-(p-q)^2$, so 
$\sig(\frsu_{3,3})=1$ and $\sig(\frsu_{5,1})=-15$. We conclude that the standard enveloping algebras of the triple
systems above are isomorphic to $\fre_{6,2}$ and $\fre_{6,-14}$, respectively.
\smallskip

For the case of standard enveloping algebra of signature $-14$, the reader can also consult \cite[\S V.A.]{grade6-14}, where the related $\ZZ$-grading on 
$\mathfrak{e}_{6,-14}$ is described in detail in order to look for the fine gradings obtained by refining the  $\ZZ$-grading (gradings which are  compatible with our linear model).

\medskip 

\subsection{Non-split $E_7$ types}\label{ss:E7nonsplit}\quad

The half-spin modules for $\frso_{12}(\CC)$ descend to $\RR$ for
$\frso_{p,q}(\RR)$ ($p+q=12$, $p\geq 6$), if and only if the even Clifford algebra 
$\Cl_{p,q}(\RR)\subo$ is isomorphic to 
$\Mat_{32}(\RR)\times\Mat_{32}(\RR)$, and this is the case if and only if the discriminant of the quadratic form with signature $p-q$ and the
Brauer class of $\Cl_{p,q}(\RR)$ are trivial. This happens if and only 
if $p$ is even and $p-q$ is congruent to $0$ or $2$ modulo $8$ (\cite[p.~125]{Lam}). That is, if and only if $p=6$ or $p=10$. 
In case $p=6$, $\frso_{p,q}(\RR)=\frso_{6,6}(\RR)$ is the split
simple Lie algebra of type $D_6$. Let us consider now the other case:
$p=10$.

Let then $V$ be a real vector space of dimension $12$ endowed with a
non-degenerate quadratic form $\qup_V$ of signature $8$. 
Let
$\{w_1,\ldots,w_{12}\}$ be an orthogonal basis relative to $\qup_V$
with $\qup_V(w_i)=1$ for $1\leq i\leq 10$, and $\qup_V(w_i)=-1$ for 
$i=11,12$.

As in Subsection \ref{ss:E7}, let $W$ be a six-dimensional complex space. The complexified space $V^\CC=V\otimes_\RR\CC$ can be isometrically identified  with $W\oplus W^*$ by means of
\[
\begin{cases}
w_i\leftrightarrow e_i+e^i&\text{for $i=1,\ldots,6$,}\\
w_{6+i}\leftrightarrow \bi(e_i-e^i)&\text{for $i=1,\ldots,4$,}\\
w_{6+i}\leftrightarrow e_i-e^i&\text{for $i=5,6$.}
\end{cases}
\]
In this way, the real Clifford algebra $\Cl(V,\qup_V)$ embeds into
the complex Clifford algebra $\Cl(W\oplus W^*,\qup)$. Consider the isomorphism $\Lambda$ in 
\eqref{eq:E7_Lambda}. The element $x=w_7w_8w_9w_{10}$ in
$\Cl(V,\qup_V)\subseteq \Cl(W\oplus W^*,\qup)$  commutes with $w_i$, for $i\neq 7,8,9,10$, 
 anticommutes with $w_i$, for $i=7,8,9,10$ and it satisfies $x^2=1$,
so $\Lambda(x)\in\End_\CC\bigl(\bigwedge W\bigr)$ is such that $\Lambda(x)^2=\id$.

 Consider the 
conjugate-linear endomorphism $\Delta$ of $\bigwedge W$ that is the identity
on the basic elements $e_I$, for $I\subset\{1,\dots,6\}$. Then $\Delta^2=\id$, $\Delta$ commutes
with $\Lambda(e_i)=l_{e_i}$ and $\Lambda(e^i)=\delta_{e^i}$ for all $i$, and hence
$\Delta$ commutes too with $\Lambda(w_i)$, for $i\neq 7,8,9,10$, and
anticommutes with $\Lambda(w_i)$ for $i=7,8,9,10$. This shows that
$\Delta$ and $\Lambda(x)$ commute.

Take the composition $\Gamma=\Delta\Lambda(x)$. This is a conjugate-linear endomorphism of $\bigwedge W$, it preserves its even and odd parts, its order is $2$: $\Gamma^2=\id$,  and it commutes with
$\Lambda(w_i)$ for all $i=1,\ldots,12$. Hence $\Gamma$ commutes with
the action of $\Cl(V,\qup_V)$ on $\bigwedge W$ through $\Lambda$.

It turns out that the fixed subspace
$\bigl(\bigwedge W\bigr)^\Gamma=\{z\in \bigwedge W: 
\Gamma(z)=z\}$, and its even and odd parts, are invariant under
the action of $\Cl(V,\qup_V)$, and hence the even and odd parts are
the half-spin modules of $\frso(V,\qup_V)$. 

Denote by $T=\bigwedge\subo W$ the complex symplectic triple system 
as in Subsection \ref{ss:E7}. The unique, up to scalars, non-zero 
$\frso(W\oplus W^*,\qup)$-invariant
bilinear form on $T$ is the restriction $(\cdot\vert\cdot)$ of $\bup$ in \eqref{eq:E7_b}. Hence there is a scalar $\alpha\in\CC$ such that
the restriction of $(\cdot\vert\cdot)$ to the fixed subspace satisfies
$(T^\Gamma\vert T^\Gamma)\in \alpha\RR$.
Let us check that this scalar is real.
Note that $\Gamma(1)=e_{1234}$, 
$\Gamma(e_{56})=e_{123456}$, and that   $1+\Gamma(1)$ and 
$e_{56}+\Gamma(e_{56})$ are in $T^\Gamma$. We compute
\[
(1+e_{1234}\vert e_{56}+e_{123456})=
(1\vert e_{123456})+(e_{1234}\vert e_{56})=2\, 
\text{($\in \RR$).}
\]
We conclude, as in Subsection \ref{ss:E6nonsplit},  that
$\bigl(T^\Gamma,[\cdot,\cdot,\cdot],(\cdot\vert\cdot)\bigr)$ is a real
symplectic triple system, with inner derivation algebra isomorphic
to $\frso(V,\qup_V)\cong\frso_{10,2}(\RR)$, and standard enveloping
Lie algebra isomorphic to $\fre_{7,-25}$ (using Proposition~\ref{pr_signat} and that the signature of the 
Killing form of $\frso_{10,2}(\RR)$
is $-26$.)

\medskip

There is another example of real simple symplectic triple system with 
standard enveloping Lie algebra of type $E_7$.

To describe it, let now $\mathcal{U}$ be a right $\HH$-module of dimension $6$
endowed with a skew-hermitian non-degenerate form 
$\hup\colon \mathcal{U}\times \mathcal{U}\to\HH$. Consider the natural copy of $\CC=\RR+\RR\bi$
contained in $\HH=\RR+\RR\bi+\RR\bj+\RR\bk$. By restriction of 
scalars, we may look at $\mathcal{U}$ as a complex vector space. We will write
$\mathcal{U}_\CC$ when we do this. Then $\mathcal{U}_\CC$ has dimension $12$, and it
is endowed with a non-degenerate quadratic form $\qup_\mathcal{U}$ defined by:
\begin{equation}\label{eq:formacuad}
\qup_\mathcal{U}(w)=\frac12\operatorname{proj}_\CC\bigl(\bj \hup(w,w)\bigr),
\end{equation}
where we use the natural projection 
$\operatorname{proj}_\CC(a+b\bi+c\bj+d\bk)=a+b\bi$. Note that
for all $w\in \mathcal{U}$, $\qup_\mathcal{U}( w\bj)$ is the (complex) conjugate of 
$q_\mathcal{U}(w)$:
\begin{equation}\label{eq:qW}
q_\mathcal{U}(w\bj)=\overline{\qup_\mathcal{U}(w)}.
\end{equation}
Denote by $\bup_\mathcal{U}$ the associated (bilinear, complex) polar form:
$\bup_\mathcal{U}(u,v)
=\qup_\mathcal{U}(u+v)-\qup_\mathcal{U}(u)-\qup_\mathcal{U}(v)$. 
Note that for any
$u,v\in \mathcal{U}$ we have $\bj \hup(u,v)=\bup_\mathcal{U}(u,v)+\{u\vert v\}\bj$ for some
$\{u\vert v\}\in\CC$, and hence:
\[
\hup(u,v)=\begin{cases}
 \bj \hup(u\bj,v)=\bup_\mathcal{U}(u\bj,v)+\{u\bj\vert v\}\bj,&\\
   -\bj^2 \hup(u,v)=-\bj\bigl(\bup_\mathcal{U}(u,v)+\{u\vert v\}\bj\bigr)
   =\overline{\{u\vert v\}}-\overline{\bup_\mathcal{U}(u,v)}\bj,&
\end{cases}
\]
because $\bj\alpha=\overline{\alpha}\bj$ for all $\alpha\in\CC$.
Thus we recover $\hup$ as 
$\hup(u,v)=\bup_\mathcal{U}(u\bj,v){\az-}\overline{\bup_\mathcal{U}(u,v)}\bj$.

The simple Lie algebra $\frso^*(\mathcal{U},\hup)$ appears now as follows, where 
$R_\bj$ denotes the 
multiplication by $\bj$ on $\mathcal{U} $:
\begin{equation}\label{eq_meterelso}
\begin{split}
\frso^*(\mathcal{U},\hup)&=\{f\in\End_\HH(\mathcal{U}): 
     \hup\bigl(f(u),v\bigr)+\hup\bigl(u,f(v)\bigr)=0\ \forall u,v\in \mathcal{U}\}\\
   &=\{f\in\End_\CC(\mathcal{U}):  fR_\bj=R_{\bj}f, \  \bup_\mathcal{U}\bigl(f(u),v\bigr)+\bup_\mathcal{U}\bigl(u,f(v)\bigr)=0\ \forall u,v\in \mathcal{U}\}\\
   &=\frso(\mathcal{U}_\CC,\qup_\mathcal{U})\cap
          \{f\in\End_\CC(\mathcal{U}):  fR_\bj=R_{\bj}f\}.
\end{split}
\end{equation}

Pick a basis $\{v_1,v_2,v_3,w_1,w_2,w_3\}$ of $\mathcal{U}$ (over $\HH$) with
$\hup(v_i,v_j)=0=\hup(w_i,w_j)$, and $\hup(v_i,w_j)=\delta_{ij}$, for
$1\leq i,j\leq 3$. As in Subsection \ref{ss:E7}, let $W$ be a six-dimensional complex space. The (complex) vector space $\mathcal{U}_\CC$, endowed
with $\qup_\mathcal{U}$, can be identified isometrically with $W\oplus W^*$, 
endowed with the quadratic form $\qup$ in Subsection~\ref{ss:E7}, by means of
\begin{align*}
v_i&\leftrightarrow e_i,& v_i\bj&\leftrightarrow e_{i+3},\\
-w_i\bj&\leftrightarrow e^i,& w_i&\leftrightarrow e^{i+3},
\end{align*}
for $i=1,2,3$. That is, $W$ is identified with $v_1\HH+v_2\HH+v_3\HH$ and $W^*$ 
with $w_1\HH+w_2\HH+w_3\HH$. In this way, the (complex) Clifford
algebra $\Cl(\mathcal{U}_\CC,\qup_\mathcal{U})$ is identified with 
$\Cl(W\oplus W^*,\qup)$ in Subsection \ref{ss:E7}.

The conjugate-linear map $R_\bj$ on $\mathcal{U}_\CC$ induces a   conjugate-linear algebra automorphism $\widehat{R_\bj}$ of 
$\Cl(\mathcal{U}_\CC,\qup_\mathcal{U})$: 
$\widehat{R_\bj}(u_1\cdots u_r)=(u_1\bj)\cdots (u_r\bj)$
for any $r\geq 0$ and $u_1,\ldots,u_r\in \mathcal{U}$. In the same vein, it 
induces a conjugate-linear algebra automorphism $\widetilde{R_\bj}$ of
the exterior algebra $\bigwedge W$:
$\widetilde{R_\bj}(u_1\wedge\cdots\wedge u_r)=
(u_1\bj)\wedge\cdots\wedge (u_r\bj)$
for any $r\geq 0$ and $u_1,\ldots,u_r\in W$. Note that $\widehat{R_\bj}$ has order $2$ in the even Clifford algebra
$\Cl\subo(\mathcal{U}_\CC,\qup_\mathcal{U})$, and order $4$ in the odd part. Also,
$\widetilde{R_\bj}$ has order $2$ in $T=\bigwedge\subo W$, and order $4$ in $\bigwedge\subuno W$.

Using \eqref{eq:qW}, it follows that the algebra isomorphism 
$\Lambda$ in \eqref{eq:E7_Lambda} satisfies 
$\Lambda\widehat{R_\bj}(x)\widetilde{R_\bj}
=\Lambda(x)\widetilde{R_\bj}$ for all $x\in\Cl(\mathcal{U}_\CC,\qup_\mathcal{U})$, 
that is, $\Lambda\widehat{R_\bj}=\Ad_{\widetilde{R_\bj}}\Lambda$.  Thus we have 
$\widetilde{R_\bj}(x.s)=\widehat{R_\bj}(x).\widetilde{R_\bj}(s)$ for 
$x\in\Cl(\mathcal{U}_\CC,\qup_\mathcal{U})$
and $s\in T=\bigwedge\subo W$. We conclude that the space of fixed elements 
$T^{\widetilde{R_\bj}}$ 
is a module for the fixed subalgebra 
$\widehat\cC=\left(\Cl\subo(\mathcal{U}_\CC,\qup_\mathcal{U})\right)^{\widehat{R_\bj}}$, whose complexification is the half-spin module $T=\bigwedge\subo W$.

 Note that $\Ad_{R_\bj}(\sigma_{x,y})= \sigma_{x\bj,y\bj}$, and this coincides with 
 $\widehat{R_\bj}(\sigma_{x,y}) $  if we see the element $\sigma_{x,y}\in \frso(\mathcal{U}_\CC,\qup_\mathcal{U}) $ as the element $\frac{-1}{2}[x,y]$ in  
 $\Cl\subo(\mathcal{U}_\CC,\qup_\mathcal{U})$. Hence $\frso^*(\mathcal{U},\hup)$ is contained in 
 $\widehat\cC$ by \eqref{eq_meterelso}, and acts in $T^{\widetilde{R_\bj}}$.

The unique, up to scalars, non-zero $\frso(\mathcal{U}_\CC,\qup_\mathcal{U})$-invariant
bilinear form on $T$ is the restriction $(\cdot\vert\cdot)$ of $\baup$ in
\eqref{eq:E7_b}. Hence there is a scalar $\alpha\in\CC$ such that
the restriction of $(\cdot\vert\cdot)$ to the fixed subspace satisfies
$(T^{\widetilde{R_\bj}}\vert T^{\widetilde{R_\bj}})\subseteq\alpha\RR$.
Note that the element $e_{123456}
=v_1\wedge v_2\wedge v_3\wedge v_1\bj\wedge v_2\bj\wedge v_3\bj$
is fixed by $\widetilde{R_\bj}$, as well as the element $1$. Since
$(1\vert e_{123456})=1$, it follows that $\alpha$ belongs to $\RR$, 
 and we conclude, as above, that 
$\bigl(T^{\widetilde{R_\bj}},[\cdot,\cdot,\cdot],
(\cdot\vert\cdot)\bigr)$ is a real symplectic triple system, with inner derivation algebra isomorphic to $\frso^*(\mathcal{U},\hup)\cong\frso^*_{12}$, 
and standard enveloping Lie algebra isomorphic to $\fre_{7,-5}$ 
(since the signature of the Killing form of $\frso^*_{12}$ is $-6$).

\begin{remark}
The conjugate-linear map $\widetilde{R_\bj}$ squares to $-\id$ on
$\bigwedge\subuno U$, commuting with the action of the fixed 
subalgebra $\widehat\cC $.  This
gives $\bigwedge\subuno W$ the structure of a $\HH$-module of dimension $16$ and proves that $\widehat\cC$ is isomorphic to
$\Mat_{32}(\RR)\times\Mat_{16}(\HH)$. 

(Even) Clifford algebras can be defined for any finite-dimensional central simple algebra endowed with an orthogonal involution over a field of characteristic not two. According to \cite[\S 3]{Jacobson}, 
$\widehat\cC$ is the even Clifford 
algebra of the pair $\bigl(\End_\HH(\mathcal{U}),*)$, where $*$ denotes the 
involution obtained as the
conjugation associated to the skew-hermitian form $\hup$.
\end{remark}

\pagebreak
%\smallskip

\subsection{Non-split $E_8$ type}\label{ss:E8nonsplit}\quad

We will follow here the notation and results in Subsection \ref{ss:E8} over $\CC$.
Thus $U$ is an eight-dimensional complex vector space, but this time endowed with a non-degenerate hermitian form $\hup\colon U\times U\to\CC$. Again, changing $\hup$ by 
$-\hup$ if necessary, we assume the signature of $\hup$ is $\geq 0$. We fix an orthogonal basis $\{e_i: 1\leq i\leq 8\}$ with
$\hup(e_i,e_i)=1$ for $1\leq i\leq p$, and $\hup(e_i,e_i)=-1$ for $p+1\leq i\leq 8$, with $p\geq 4$. Use this basis to define the determinant map 
$\det\colon \bigwedge U\to \CC$, so that $\det(\bigwedge^iU)=0$ for $i<8$, and $\det(e_{12345678})=1$.

The hermitian form $\hup$ induces a conjugate-linear automorphism 
$\Upsilon$ of $\cL=\frsl(U)\oplus\bigwedge^4U$ as follows:
\begin{itemize}
\item 
$\Upsilon(l)=-l^*$ for $l\in\frsl(U)$, where $l^*$ is the adjoint
of $l$ relative to $\hup$: $\hup\bigl(l(v),w\bigr)=\hup\bigl(v, l^*(w)\bigr)$ for all
$v,w\in U$.   

\item
$\hup$ induces  a conjugate-linear map for any $s=0,\ldots,8$:
\[
\begin{split}
\widehat{\hup}_s\colon \textstyle{\bigwedge^sU}&\longrightarrow \textstyle{\bigwedge^s U^*}\\
w_1\wedge\cdots\wedge w_s
	&\mapsto \hup(\cdot,w_1)\wedge\cdots\wedge \hup(\cdot,w_s).
\end{split}
\]
Then for all $x\in\bigwedge^4U$, we define, using \eqref{eq:E8Phi},
\[
\Upsilon(x)\bydef \bigl(\widehat{\hup}_4\bigr)^{-1}\Phi_4(x).
\]
In particular, for any increasing sequence $I$ of size $4$, we get:
\begin{equation}\label{eq:E8nonsplitUpsilon}
\Upsilon(e_I)=(-1)^{I\,\overline{I}}
(-1)^{\lvert I\cap I_p\rvert}e_{\overline{I}},
\end{equation}
where $I_p=\{p+1,\ldots,8\}$, because $\hup(\cdot,e_i)$ is $e^i$ for $i\leq p$ and $-e^i$ for  $i>p$.
\end{itemize}

\begin{lemma}\label{le:E8nonsplit}
The map $\Upsilon\colon \cL\to\cL$ thus defined is a Lie algebra automorphism over 
$\RR$ with $\Upsilon^2=\id$ if and only if $p$ is even.
\end{lemma}
\begin{proof}
Since $l\mapsto l^*$ is an involution of $\End_\CC(U)$, the restriction
$\Upsilon\vert_{\frsl(U)}$ is a Lie algebra automorphism whose square
is the identity. For any increasing sequence $I$ of size $4$, we have
\begin{multline*}
\Upsilon^2(e_I)
=(-1)^{I\overline{I}}(-1)^{\lvert I\cap I_p\rvert}\Upsilon(e_{\overline{I}})
=
(-1)^{\lvert I\cap I_p\rvert}(-1)^{\lvert\overline{I}\cap I_p\rvert}e_I
=(-1)^{\lvert I_p\rvert}e_I=(-1)^{8-p}e_I,
\end{multline*}
so that $\Upsilon^2=\id$ if and only if $p$ is even.

Assume hence from now on that $p$ is even. For $l\in \frsl(U)$ and $w_1,\ldots,w_4\in U$, we get
\begin{equation}
\begin{split}\label{eq_accioncompat}
\widehat{\hup}_4\bigl(l.(w_1\wedge\cdots\wedge w_4)\bigr)
	&=\hup(\cdot,l.w_1)\wedge \hup(\cdot,w_2)\wedge \hup(\cdot,w_3)\wedge \hup(\cdot, w_4)+\cdots\\
&\qquad +\hup(\cdot,w_1)\wedge \hup(\cdot,w_2)\wedge \hup(\cdot,w_3)\wedge \hup(\cdot, l.w_4)\\[4pt]
&=\hup\bigl(l^*(\cdot),w_1)\wedge\cdots\wedge \hup(\cdot,w_4)+\cdots\\
&= (-l^*.\hup(\cdot,w_1))\wedge\cdots\wedge \hup(\cdot,w_4)+\cdots\\
&=\Upsilon(l).\widehat{\hup}_4(w_1\wedge\cdots\wedge w_4).
\end{split}
\end{equation}
This proves $\widehat{\hup}_4\bigl(\Upsilon(l).\Upsilon(x)\bigr)
=l.\bigl(\widehat{\hup}_4\Upsilon(x)\bigr)=l.\Phi_4(x)$, and hence,
for $l\in \frsl(U)$ and $x\in\bigwedge^4U$, we get:
\[
\begin{split}
\Upsilon\bigl([l,x]\bigr)&=\Upsilon(l.x)
=\bigl(\widehat{\hup}_4\bigr)^{-1}\Phi_4(l.x)\\
	&=(\widehat{\hup}_4\bigr)^{-1}\bigl(l.\Phi_4(x)\bigr)\quad
\text{(as $\Phi_4$ is $\frsl(U)$-invariant)}\\
	&=\Upsilon(l).\Upsilon(x)=[\Upsilon(l),\Upsilon(x)].
\end{split}
\]
Note that for any increasing sequences $I,J$ with $\lvert I\rvert=4=\lvert J\rvert$, \eqref{eq:E8nonsplitUpsilon} gives
\[
(\Upsilon(e_I)\vert\Upsilon(e_{\overline{I}}))_\wedge
= (-1)^{\lvert I\cap I_p\rvert}(-1)^{\lvert\overline{I}\cap I_p\rvert}(e_I\vert e_{\overline{I}})_\wedge=(e_I\vert e_{\overline{I}})_\wedge=\overline{(e_I\vert e_{\overline{I}})_\wedge}
\]
and $(\Upsilon(e_I)\vert\Upsilon(e_J))_\wedge=0= \overline{(e_I\vert e_J)_\wedge}$ if 
$J\neq \overline{I}$. Hence we get
\[
(\Upsilon(x)\vert\Upsilon(y))_\wedge=\overline{(x\vert y)_\wedge}
\]
for all $x,y\in \bigwedge^4U$. Finally, for $l\in\frsl(U)$ and 
$x,y\in\bigwedge^4U$, we get
\[
\begin{split}
&\trace\Bigl(\Upsilon(l)\Upsilon([x,y])\Bigr)
   =\trace(l^*[x,y]^*)=\trace\bigl(([x,y]l)^*\bigr)=
   \overline{\trace([x,y]l)}\\
&\qquad =\overline{\trace(l[x,y])}=\overline{(l.x\vert y)_\wedge}
   =(\Upsilon(l.x)\vert\Upsilon(y))_\wedge\\
&\qquad =(\Upsilon(l).\Upsilon(x)\vert\Upsilon(y))_\wedge
  =\trace\bigl(\Upsilon(l)[\Upsilon(x),\Upsilon(y)]\bigr),
\end{split}
\]
so that $\Upsilon([x,y])=[\Upsilon(x),\Upsilon(y)]$, and $\Upsilon$ is
a conjugate-linear bijection and an automorphism of $\cL$ as a real Lie algebra. 
\end{proof}

\begin{proposition}\label{pr:E8nonsplit}
With the notations above and even $p$, the fixed subalgebra of 
$\Upsilon$: $\cS=\cL^\Upsilon\bydef\{l\in\cL: \Upsilon(l)=l\}$ is a form of $\cL$ 
(i.e., $\cS$ is a real subalgebra and $\cL=\cS\oplus \bi\,\cS$). Moreover,
\begin{itemize}
\item If $p=4$ or $8$, then $\cS$ is, up to isomorphism, the simple 
exceptional real Lie algebra $\fre_{7,7}$ (the split simple real Lie 
algebra of type $E_7$).
\item If $p=6$, then $\cS$ is, up to isomorphism, the simple
exceptional real Lie algebra $\fre_{7,-25}$.
\end{itemize}
\end{proposition}
\begin{proof}
It is clear that $\cS$ is a real form of $\cL$, since it is the fixed subspace by an order 2 conjugate-linear automorphism.
The bilinear form $(\cdot\vert\cdot)_\cL$ is $\frac{1}{36}$ times the 
Killing form of $\cL$. Hence the restriction of $(\cdot\vert\cdot)_\cL$ to $\cS$ is
$\frac{1}{36}$ times the 
Killing form of $\cS$ and, in particular, it takes values in $\RR$.

Taking into account that
the signature of the Killing form of $\frsu_{p,8-p}$  is 
$1-(p-(8-p))^2=1-(2p-8)^2$, and that the number of increasing
sequences $I=(1,i,j,k)$ with even $\lvert I\cap I_p\rvert$ is:
\begin{itemize}
\item $1+3\binom{4}{2}=19$ if $p=4$ (so $I_p=\{5,6,7,8\}$),
\item $\binom{5}{3}+5=15$ if $p=6$,
\item $\binom{7}{3}=35$ if $p=8$,
\end{itemize}
we conclude that the signature of $\kappa_\cS$ is:
\begin{itemize}
\item $1+(38-32)=7$ for $p=4$,
\item $-15+(30-40)=-25$ for $p=6$,
\item $-63+70=7$ for $p=8$.
\end{itemize}
This finishes the proof.
\end{proof}

\begin{remark}
We may consider too the real forms of $\cL$ obtained as 
$\overline{\cS}=\frsu(U,\hup)\oplus\bi \bigl(\bigwedge^4U\bigr)^\Upsilon$. 
The corresponding signatures are:
\begin{itemize}
\item $1-(38-32)=-5$ for $p=4$,
\item $-15-(30-40)=-5$ for $p=6$,
\item $-63-70=-133$ for $p=8$.
\end{itemize}
Therefore we obtain  precise linear models of   the other two real simple Lie algebras of type
$E_7$: $\fre_{7,-5}$ and the compact one, $\fre_{7,-133}$.
This is a nice complement to \cite[Chapter 12]{Adams}.
\end{remark}

From now on, we restrict to the case $p=6$ (as for $p=4$ or $p=8$, 
$\cS$ is the split algebra $\fre_{7,7}$, and the attached symplectic triple 
system is, up to weak isomorphism, the split one in Subsection \ref{ss:E7}).
 The aim is to get a symplectic triple 
system with inner derivation algebra $\fre_{7,-25}$.

On the $\cL$-module $T=\bigwedge^2U\oplus\bigwedge^2U^*$ consider
the order $2$ conjugate-linear isomorphism $\Upsilon_T$ given by:
\begin{itemize}
\item $\Upsilon_T(x)=\widehat{\hup}_2(x)\in\bigwedge^2U^*$ for 
$x\in\bigwedge^2U$,

\item $\Upsilon_T(y)=\bigl(\widehat{\hup}_2\bigr)^{-1}(y)\in\bigwedge^2U$
for $y\in\bigwedge^2U^*$.
\end{itemize}
As in \eqref{eq_accioncompat},
 we have, for $l\in\frsl(U)$ and
$x\in\bigwedge^2U$, 
\[
\Upsilon_T(l.x)=\widehat{\hup}_2(l.x)=\Upsilon(l).\widehat{\hup}_2(x)
=\Upsilon(l).\Upsilon_T(x),
\]
and, similarly, $\Upsilon_T(l.y)=\Upsilon(l).\Upsilon_T(y)$   for $y\in\bigwedge^2U^*$. Also, for increasing sequences $I$ of 
size $4$ and $J$ of size $2$, $I\cap J=\emptyset$,
\[
\begin{split}
\Upsilon_T(e_I.e_J)&= 
  \bigl(\widehat{\hup}_2\bigr)^{-1}\bigl((-1)^{IJ}
    (-1)^{(I\cup J)(\overline{I\cup J})}e^{\overline{I\cup J}}\bigr)\\
	&=(-1)^{IJ}
    (-1)^{(I\cup J)(\overline{I\cup J})}
    (-1)^{(\overline{I\cup J})\cap I_6}e_{\overline{I\cup J}},
\end{split}
\]
while
\[
\begin{split}
&\Upsilon(e_I).\Upsilon_T(e_J)=(-1)^{I\overline{I}}
    (-1)^{\lvert I\cap I_6\rvert}e_{\overline{I}}.
      \bigl((-1)^{\lvert J\cap I_6\rvert}e^J\bigr)\\
&\qquad =(-1)^{I\overline{I}}(-1)^{\lvert (I\cup J)\cap I_6\rvert}
     e_{\overline{I}}.e^J\\
&\qquad =(-1)^{I\overline{I}}(-1)^{\lvert (I\cup J)\cap I_6\rvert}
   (-1)^{\overline{I}I}(-1)^{IJ}
     (-1)^{(I\cup J)(\overline{I\cup J})}e_{\overline{I}\cap\overline{J}},
\end{split}
\]
so we get $\Upsilon_T(e_I.e_J)=\Upsilon(e_I).\Upsilon_T(e_J)$, 
and similarly for $e_I.e^J$. Therefore we obtain
\[
\Upsilon_T(l.x)=\Upsilon(l).\Upsilon_T(x)
\]
for any $l\in \cL$ and $x\in T$, and hence 
the subspace of fixed elements $T^\Upsilon\bydef\{x\in T: \Upsilon_T(x)=x\}$ 
is invariant under the action of $\cS$, with 
$T=T^\Upsilon\oplus\bi T^\Upsilon$.

We know that both the (complex) dimensions of
$\Hom_{\cL}(T\otimes_\CC T,\cL)$ and 
$\Hom_{\cL}(T\otimes_\CC T,\CC)$ are equal to $1$.
Therefore, the (real) dimensions of 
$\Hom_\cS(T^\Upsilon\otimes_\RR T^\Upsilon,\cS)$
and $\Hom_\cS(T^\Upsilon\otimes_\RR T^\Upsilon,\RR)$ are equal to  $1$ too.  

It follows, as in Subsection \ref{ss:E6nonsplit}, that there is a non-zero
scalar $\alpha\in\CC$ such that $d_{x,y}\in\alpha\cS$ for any $x,y\in T^\Upsilon$,
where $d_{x,y}$ is given by \eqref{eq:E8split_dxy}. 
For $x=e_{12}+e^{12}=e_{12}+\Upsilon_T(e_{12})\in T^\Upsilon$, we get
\[
d_{x,x}=d_{e_{12}+e^{12},e_{12}+e^{12}}=2d_{e_{12},e^{12}}
= -4\bigl( e_2\otimes e^2+e_1\otimes  e^1 -\frac14\id_U\bigr),
\]
that is, $d_{x,x}$ is the element of $\frsl(U)$ whose coordinate matrix in our basis is 
$$
 \diag(-3,-3,1,1,1,1,1,1),
$$
 so that 
$d_{x,x}\in\bi\,\frsu(U,\hup)$. We conclude that 
$d_{T^\Upsilon,T^\Upsilon}\subseteq\bi\,\cS$ and, as in 
Subsection~\ref{ss:E6nonsplit},
that $\bigl(T^\Upsilon,\bi[\cdot,\cdot,\cdot],\bi (\cdot\vert\cdot)\bigr)$ is
a real simple symplectic triple system with inner derivation algebra isomorphic 
to $\fre_{7,-25}$. Moreover,  by Proposition~\ref{pr_signat}, its standard enveloping Lie algebra is, up
to isomorphism, the exceptional Lie algebra $\fre_{8,-24}$.

\bigskip

%%%%%%%%%%%%%%%%%%%%%%%%%%%%%%%%%%%%%%%%%%%%%%%
%%%%%%%%%%%%%%%%%%%%%%%%%%%%%%%%%%%%%%%%%%%%%%%

\section{The classification of the real simple symplectic triple 
systems, up to weak isomorphism}\label{se:classification}

According to the classification of the simple symplectic triple systems 
over $\CC$, reviewed in Section \ref{se:split}, the pairs 
$(\inder(T),T\bigr)$ for a simple symplectic triple system $T$ over $\CC$ with classical envelope $\frg(T)$ are, up to 
isomorphism, those in the following list:
\begin{description}
\item[Special] 
$\bigl(\frgl(W),W\oplus W^*\bigr)$, for a non-zero vector space $W$.

\item[Orthogonal] 
$\bigl(\frsp(V)\oplus\frso(W),V\otimes W\bigr)$, where $V$ is a vector space of dimension $2$ endowed with a non-zero alternating bilinear form $\langle\cdot\vert\cdot\rangle$, and $W$ is a vector space of dimension $\geq 3$ endowed with a 
non-degenerate symmetric bilinear form.

\item[Symplectic]
$\bigl(\frsp(W),W\bigr)$ for a non-zero even-dimensional vector space $W$ endowed with a non-degenerate alternating bilinear form.
\end{description}

Recall that,  by 
Corollary \ref{co:Z2_STS},  in order to classify the real forms up to weak isomorphism (Definition \ref{df_weakisom}), it is enough to classify the real forms of the pairs in the list above.  We will do it now according to the type of $T$.

The results on this section are summarized in the following theorem.

\begin{theorem}\label{th:weak_iso}
Any simple real symplectic triple system is, up to weak isomorphism, one of
the split simple real symplectic triple systems in Section \ref{se:split}: special, orthogonal, symplectic, $G_2$, $F_4$, $E_6$, $E_7$, $E_8$ types, or one of the non-split simple real symplectic triple systems in Section \ref{se:nonsplit}: unitarian and quaternionic types, plus the non-split $E_6$-types (two possibilities), the non-split
$E_7$-types (two possibilities), and the non-split $E_8$-type.
\end{theorem}

The proof of this result is achieved in the next subsections.

\medskip

\subsection{Special type}

If $(\frh,U)$ is a real form of a pair of special type, then $\frh=Z(\frh)\oplus[\frh,\frh]$, with the center $Z(\frh)$ of dimension $1$, and $[\frh,\frh]$ simple (or $0$). Moreover, $Z(\frh)=\RR z$ for an element $z$ such that $\rho_z^2=\pm\id$ (by part (2) of Lemma~\ref{le:Z2}, for any  $0\ne x\in Z(\frh)$, there is
a scalar  $\alpha\in\CC$ such that, in the complexification of the pair, we have 
$\rho_{\alpha x}^2= \id$, so that $\alpha$ is either  real or purely imaginary, and the existence of $z$ above follows).

If $\rho_z^2=\id$,  then $\rho_z$ is diagonalizable and we can consider the eigenspace decomposition $U=U_+\oplus U_-$.  Then $U_+$ and $U_-$ are dual modules for the action of $\frh$, because this is the case after extending scalars to $\CC$. Besides $[\frh,\frh]$ embeds in $\frsl(U_+)$ and $\frh$ in $\frgl(U_+)$. By dimension count $\frh$ is, up to isomorphism, $\frgl(U_+)$, and we get the pair $\bigl(\frgl(U_+),U_+\oplus U_+^*\bigr)$, for a non-zero real vector space $U_+$. (\emph{Special type}.)\smallskip

Otherwise $\rho_z^2=-\id$ and $U$ becomes a complex vector space by defining $\bi u=\rho_z(u)$ for any $u\in U$. There is a unique, up to scalars, non-zero $\frh$-invariant symmetric bilinear form $\bup$ on $U$, because this is so after scalar extension.  
Define the $\RR$-bilinear map
$\hup\colon U\times U\rightarrow \CC$ by 
\[
\hup(u,v)=\bup(u,v)+\bi\bup(u,\rho_z(v)).
\]
The $\frh$-invariance of $\bup$ shows that $\hup$ is hermitian and non-degenerate. Moreover, it is also $\frh$-invariant and, by dimension count, it follows that $\frh$ is, up to isomorphism, the unitary Lie algebra $\fru(U,\hup)$. (\emph{Unitarian type}.)

\medskip

\subsection{Symplectic type}

If $(\frh,U)$ is a real form of a pair of symplectic type, then $U$ is endowed with a unique, up to scalars, non-degenerate $\frh$-invariant alternating form $(\cdot\vert \cdot)$  (because this is the case after scalar extension). Thus $\frh$ embeds in $\frsp\bigl(U,(\cdot\vert \cdot)\bigr)$, and 
by dimension count $\rho$ gives an isomorphism $\frh\simeq \frsp\bigl(U,(\cdot\vert \cdot)\bigr)$. (\emph{Symplectic type}.)

\medskip

\subsection{Orthogonal type}

If $(\frh,U)$ is a real form of a pair of orthogonal type, then 
$\frh_\CC=\frh\otimes_\RR\CC$ is isomorphic to 
$\frsl_2(\CC)\oplus\frso_n(\CC)$, with $n\geq 3$, and the module 
$U\otimes_\RR\CC$ ($\cong T$) is isomorphic to the tensor product of the 
natural (two-dimensional) module for $\frsl_2(\CC)$ and the natural 
($n$-dimensional) module for $\frso_n(\CC)$. The space 
$\Hom_\frh(U\otimes_\RR U,\RR)$ is one-dimensional, and its elements 
are skew-symmetric, due to Lemma~\ref{le:Z2} and its proof.  Indeed, 
the same happens after complexification, and each  element 
$\eta\in\Hom_{\inder(T)}(T\otimes_\CC T,\CC)$ provides 
$\eta\otimes \langle\cdot\vert \cdot\rangle\colon
\frg(T)_{\bar1}\otimes\frg(T)_{\bar1}\to\CC$ a   $\frg(T)_{\bar0}$-invariant map, which is symmetric if $\eta$ is skew-symmetric.

If $n\geq 5$, then $\frso_n(\CC)$ is simple and not isomorphic to 
$\frsl_2(\CC)$. Hence $\frh=\frh_1\oplus \frh_2$, with $\frh_1$ a  real form of $\frsl_2(\CC)$ and $\frh_2$ a real form of $\frso_n(\CC)$. Besides, the centralizer 
$\End_{\frh_2}(U)$ is a real form of $M_2(\CC)$, so it is a quaternion algebra. Also, $\frh_1$ embeds naturally in $\End_{\frh_2}(U)$ and hence it embeds isomorphically in its derived subalgebra.

Two possibilities appear. If $\End_{\frh_2}(U)$ is the split quaternion algebra, then $\frh_1$ is isomorphic to $\frsl_2(\RR)$, and $U$, as a module for $\End_{\frh_2}(U)$, is a direct sum of its unique irreducible module (the natural module for $2\times 2$-matrices). It follows that $U=U_1\otimes_\RR U_2$, with $U_1$ the natural module for $\frh_1\simeq \frsl_2(\RR)$, and $U_2$ an irreducible module for $\frh_2$. Besides, Lemma~\ref{le_Vinberg} shows that $U_2$ is endowed with a unique, up to scalars, 
$\frh_2$-invariant non-degenerate symmetric bilinear form $\bup$ and, by dimension count, the action of $\frh_2$ fills $\frso\bigl(U_2,\bup\bigr)$.

Therefore, up to isomorphism, $\frh$ is the direct sum of $\frsl_2(\RR)$ and of the special orthogonal Lie algebra $\frso\bigl(U_2,\bup\bigr)$, and $U$ is the tensor product of the natural modules for $\frsl_2(\RR)$ and $\frso\bigl(U_2,\bup\bigr)$. (\emph{Orthogonal type}.) \smallskip

However, if $\End_{\frh_2}(U)$ is the real division algebra of quaternions 
$\HH=\RR 1+\RR\bi+\RR\bj+\RR\bk$, then $\frh_1$ embeds isomorphically into $\RR\bi+\RR\bj+\RR\bk$. In other words, 
$\frh_1$ is isomorphic to $\frsu_2$. The action of the centralizer $\End_{\frh_2}(U)$ will be considered on the right (note that $[\frh_1,\frh_2]$ equals $0$). Thus $U$ is a (free) right $\HH$-module. Let $\baup$ be a non-zero element in 
$\Hom_\frh(U\otimes_\RR U,\RR)$, so that $\baup$ is a non-degenerate alternating bilinear form on $U$, and in the same vein as for the
special case, consider the $\RR$-bilinear map $\hup\colon U\times U\rightarrow \HH$ given by:
\[
\hup(u,v)=\baup(u,v)-\baup(u,v\bi)\bi-\baup(u,v\bj)\bj-\baup(u,v\bk)\bk.
\]
As we have $\baup(u,v\bi)=-\baup(u\bi,v)=\baup(v,u\bi )$ for any $u,v\in U$ by 
$\frh_1$-invariance, it follows that $\hup$ is a non-degenerate $\frh_2$-invariant skew-hermitian form. By dimension count, $\frh_2$ is isomorphic to the Lie algebra of skew-adjoint endomorphisms of $\End_\HH(U)$ relative to $\hup$, which is the Lie algebra $\frso ^*(U,\hup)\cong\frso_{2n}^*$.

Therefore, in this case, $\frh$ is, up to isomorphism the direct sum of 
$\frsu_2$ and $\frso_{2n}^*$, and $U$ is the natural module for
$\frso_{2n}^*$, with the natural action of $\frsu_2$ commuting with the action of $\frso_{2n}^*$. (\emph{Quaternionic type}.)\smallskip

Now, for $n=4$, $\frso_4(\CC)$ is isomorphic to 
$\frsl_2(\CC)\oplus\frsl_2(\CC)$, while the tensor product of the natural modules for $\frsl_2(\CC)$ and $\frso_4(\CC)$ is, up to isomorphism, the tensor product of the natural modules for each of the three copies of 
$\frsl_2(\CC)$ in $\frsl_2(\CC)\oplus\frso_4(\CC)\simeq
\frsl_2(\CC)\oplus\frsl_2(\CC)\oplus\frsl_2(\CC)$. The centroid of 
$\frsl_2(\CC)\oplus\frsl_2(\CC)\oplus\frsl_2(\CC)$ is 
$\CC\times\CC\times\CC$, and hence the centroid of $\frh$ is either $\RR\times\RR\times\RR$ or $\RR\times\CC$. Therefore $\frh=\frh_1\oplus\frh_2$ for a three-dimensional simple Lie algebra and a six-dimensional semisimple Lie algebra $\frh_2$ which is either the direct sum of two central simple ideals, or a simple Lie algebra with centroid $\CC$. Thus the arguments for $n\geq 5$ work here, and we get either the orthogonal or  the quaternionic type.

Finally, for $n=3$, $\frso_3(\CC)$ is isomorphic to $\frsl_2(\CC)$, 
while the tensor 
product of the natural modules for $\frsl_2(\CC)$ and $\frso_3(\CC)$ is the tensor 
product of the natural module for the first copy of $\frsl_2(\CC)$ and the 
adjoint module for the second copy of $\frsl_2(\CC)$ in $\frsl_2(\CC)\oplus\frso_3(\CC)\simeq\frsl_2(\CC)\oplus\frsl_2(\CC)$. There are two 
possibilities for $\frh$.  Either $\frh$ is the direct sum of two 
three-dimensional simple Lie algebras, so that  the arguments for $n\geq 5$ work, 
or $\frh$ is simple with centroid $\CC$, and hence it is isomorphic to 
$\frsl_2(\CC)$ considered as a real Lie algebra. Note that this is, up to isomorphism, the Lorentz Lie algebra $\frso_{3,1}(\RR)$. But the real Lie algebra 
$\frsl_2(\CC)$ has no six-dimensional representation whose complexification is the tensor product above. (See, for instance, \cite[Example 8.2]{Oni}.)
Therefore this last situation   cannot occur, and again, for $n=3$, we are either in the orthogonal or  in the quaternionic type.

\bigskip

\subsection{Exceptional envelope}\quad

From the classification of the simple symplectic triple systems over 
$\CC$ with exceptional envelope reviewed in Section \ref{se:split}, 
and with the notation used there, the pairs $\bigl(\inder(T),T\bigr)$ are, up to isomorphism, those in the 
next list:
\begin{description}
\item[$G_2$-envelope] $(\fra_1,V(\varpi_3))$,
\item[$F_4$-envelope] $(\frc_3,V(\varpi_3))$,
\item[$E_6$-envelope] $(\fra_5,V(\varpi_3))$,
\item[$E_7$-envelope] $(\frd_6,V(\varpi_6))$,
\item[$E_8$-envelope] $(\fre_7,V(\varpi_1))$.
\end{description}
Note that in all cases, $T$ is an irreducible module for $\inder(T)$.

Let $(\frh,U)$ be a real form of one of the pairs 
$\bigl(\frg,V(\Lambda)\bigr)$ in the list above.
This means that $\frh$ is a real form of $\frg$, and $U$ is the unique
absolutely irreducible module for $\frh$ whose complexification is $V(\Lambda)$.

Given a real form $\frh$ of $\frg$, such an absolutely irreducible module $U$ exists if and only if  (see \cite[\S 8]{Oni})
\begin{equation}\label{eq:Cartan_index}
s_0(\Lambda)=\Lambda\quad\text{and}\quad \varepsilon(\frh,\Lambda)=1,
\end{equation}
where $s_0$ is the automorphism of the Dynkin diagram attached to the real form $\frh$, and $\varepsilon(\frh,\Lambda)$ the Cartan index of the representation of $\frh$ on $U\otimes_\RR\CC\simeq V(\Lambda)$.

For $\fra_1$, $\frc_3$, and $\fre_7$, the Dynkin diagram has no proper automorphisms, and for $\fra_5$, its fundamental weight $\varpi_3$ is invariant under the non-trivial automorphism of the Dynkin diagram, so only for the $E_7$-envelope we must worry about the first restriction in \eqref{eq:Cartan_index}. 

The information contained in \cite[Table 5]{Oni} gives at once the classification of these real forms $(\frh,U)$, up to isomorphism. The corresponding simple real Lie algebra $\frh$ are given in the following list. The corresponding modules $U$ are uniquely determined from $V(\Lambda)$ and 
have been explicitly constructed in Sections \ref{se:split} and \ref{se:nonsplit}.
\begin{description}
\item[$G_2$-envelope] 
 $\frh=\frsl_2(\RR)$, so only the split case appears;

\item[$F_4$-envelope] 
 $\frh=\frsp_6(\RR)$, again only the split case appears,  because   
 $$\textstyle\varepsilon(\mathfrak{sp}_{1,2},\sum_{i=1}^3 \Lambda_i\varpi_i)=
 \varepsilon(\mathfrak{sp}_{3},\sum_{i=1}^3 \Lambda_i\varpi_i)=
 (-1)^{\Lambda_1+\Lambda_3}$$
 are the Cartan indexes of the irreducible representations of the other real forms of $ \frsp_6(\CC)$ (denoted by $\frsp_{1,2}$ and $\frsp_3$ in \cite{Oni});

\item[$E_6$-envelope] 
 $\frh$ is isomorphic to either $\frsl_6(\RR)$ (split case), $\frsu_{3,3}$, or 
 $\frsu_{5,1}$, because we have
 \[
 \begin{split}
 &\varepsilon(\mathfrak{su}_6,\textstyle\sum_{i=1}^5 \Lambda_i\varpi_i)=(-1)^{\Lambda_3},\\
 &\textstyle \varepsilon(\mathfrak{su}_{4,2},\sum_{i=1}^5 \Lambda_i\varpi_i)=(-1)^{5\Lambda_3},\\
 &\textstyle \varepsilon(\mathfrak{sl}_{3}(\HH),\sum_{i=1}^5 \Lambda_i\varpi_i)=(-1)^{\Lambda_1+\Lambda_3+\Lambda_5};
 \end{split}
\]

\item[$E_7$-envelope] 
 $\frh$ is isomorphic to either $\frso_{p,12-p}(\RR)$, where $p$ and $\frac{6-p}{2}$ are even, or to $\frso^*_{12}$. Hence we obtain three different possibilities: $\frso_{6,6}(\RR)$ (split case), $\frso_{10,2}(\RR)$, and $\frso^*_{12}$
   (denoted by $\mathfrak{u}_6^*(\HH) $ in \cite{Oni}), because we have   
 \[
 \varepsilon(\frso^*_{12},\textstyle\sum_{i=1}^6 \Lambda_i\varpi_i)=(-1)^{\Lambda_1+\Lambda_3+\Lambda_5}
 \] 
 (which equals $1$ for $\varpi_6 $), while
\[
 \begin{split}
 &\textstyle \varepsilon(\frso_{12}(\RR),\sum_{i=1}^6 \Lambda_i\varpi_i)=(-1)^{3(\Lambda_5+\Lambda_6)},\\
 &\textstyle \varepsilon(\frso_{4,8}(\RR),\sum_{i=1}^6 \Lambda_i\varpi_i)=(-1)^{ \Lambda_5+\Lambda_6}.
 \end{split}
\]

\item[$E_8$-envelope] 
 $\frh$ is isomorphic to either $\fre_{7,7}$ (split case) or to 
$\fre_{7,-25}$,  because we have
$$
\textstyle \varepsilon(\fre_{7,-133},\sum_{i=1}^7 \Lambda_i\varpi_i)=
 \varepsilon(\fre_{7,-5},\sum_{i=1}^7 \Lambda_i\varpi_i)=
 (-1)^{ \Lambda_1+\Lambda_3+\Lambda_7}.
$$
\end{description}

\section{The classification of the real simple symplectic triple 
systems }\label{se:classificationII}

In this final section, we will prove that weakly isomorphic simple real symplectic triple systems are actually isomorphic, so the list in Theorem \ref{th:weak_iso} 
gives the classification up to isomorphism too.

Over any field $\FF$ with $\FF=\FF^2$, it is trivial to check that weakly isomorphic systems are isomorphic (see the argument prior to Definition \ref{df_weakisom}). 
However, over $\RR$ the situation is much subtler although, eventually, the same result holds: weakly isomorphic simple symplectic triple systems are isomorphic.

Our proof is based on the following preliminary result.
 
\begin{lemma}\label{le:Z4}
Let $T$ be a real symplectic triple system endowed with a 
$\ZZ/4$-grading with support $\{\bar 1,\bar 3\}$. Then $T$ is 
isomorphic to $T^{[-1]}$.
\end{lemma}
\begin{proof}
Let $T=T_{\bar 1}\oplus T_{\bar 3}$ be the given $\ZZ/4$-grading, i.e., 
$[T_{\bar i},T_{\bar j},T_{\bar k}]\subset T_{\bar i+\bar j+\bar k} $. 
Then the linear map $T\rightarrow T$ given by $x\mapsto x$ for 
$x\in T_{\bar 1}$ and $x\mapsto -x$ for $x\in T_{\bar 3}$ is an 
isomorphism $T\rightarrow T^{[-1]}$.
\end{proof}

\begin{remark}
Using that $[x,y,z]-[x,z,y]=(x\vert y)z-(x\vert z)y+2(y\vert z)x$, it is easy 
to see that any grading by an abelian group of $\bigl(T,[\cdot,\cdot,\cdot]\bigr)$ satisfies 
$(T_g\vert T_h)=0$ unless $g+h=0$.
\end{remark}

\begin{proposition}
Any classical simple real symplectic triple system $T$ is isomorphic to $T^{[-1]}$.
\end{proposition}
\begin{proof}
We will check that Lemma~\ref{le:Z4} applies in all cases.

\smallskip

\noindent\emph{Special type:}\quad
$T=W\oplus W^*$ and $\inder(T)=\frgl(W)$ for a 
vector space $W$. Then $T$ is $\ZZ/4$-graded with $T_{\bar 1}=W$ 
and $T_{\bar 3}=W^*$ according to \eqref{eq:special}.

\smallskip

\noindent\emph{Orthogonal type:}\quad
In this case, $T=V\otimes_\RR W$, where $V$ is a two-dimensional vector space endowed with a non-zero skew-symmetric 
bilinear form, and $W$ is a vector space endowed with a non-degenerate 
symmetric bilinear form. Moreover, $\inder(T)=\frsl(V)\oplus\frso(W)$. 
Let $h$ be a diagonalizable element of $\frsl(V)$ with eigenvalues 
$\pm 1$. Then $V=\RR u\oplus\RR v$, with $h.u=u$ and $h.v=-v$. Again 
$T$ is $\ZZ/4$-graded with $T_{\bar 1}=u\otimes W$ and 
$T_{\bar 3}=v\otimes W$, due to \eqref{eq:orthogonal}.

\smallskip

\noindent\emph{Symplectic type:}\quad
Here $T$ is a real vector space endowed with a non-degenerate skew-symmetric form $(\cdot\vert \cdot)$. Take maximal complementary isotropic subspaces $T_{\bar 1}$ and $T_{\bar 3}$. These give a $\ZZ/4$-grading of $T$ as a vector space. Then $(\cdot\vert \cdot)$ is a homogeneous map of degree $\bar 0$, and hence this grading is a grading as a triple system. 

\smallskip

\noindent\emph{Unitarian type:}\quad
Here $T$ is a complex vector space, endowed with a non-degenerate 
hermitian form $\hup\colon T\times T\rightarrow \CC$. 
The skew-symmetric form and triple product on $T$ are given by the following formulas, from \eqref{eq:h_re_im} and \eqref{eq_triplecasounitario}:
\[
\begin{split}
&(x\vert y)=\frac{-\bi}{2}\bigl(\hup(x,y)-\hup(y,x)\bigr),\\
&[x,y,z]=(x\vert y)z+\bi\bigl(\hup(z,x)y+\hup(z,y)x+\hup(x,y)z\bigr).
\end{split}
\]
Consider the $\ZZ/4$-grading on $\CC$ with $\CC_{\bar 0}=\RR 1$, 
$\CC_{\bar 2}=\RR\bi$. Take an `orthonormal $\CC$-basis' $\{u_1,\ldots,u_n\}$, so $\hup(u_r,u_s)=0$ if $r\neq s$, and $\hup(u_r,u_r)$ is either $1$ or $-1$. Then $T$ is a $\ZZ/4$-graded module over $\CC$ with 
$T=T_{\bar 1}\oplus T_{\bar 3}$, where 
$T_{\bar 1}=\sum_{r=1}^{n}\RR u_r$ and 
$T_{\bar 3}=\bi T_{\bar 1}$. It follows that $\hup(\cdot,\cdot)$, considered as a bilinear form over $\RR$, is homogeneous of degree $\bar 2$, and hence both $\bi \hup(\cdot,\cdot)$ and $(\cdot\vert \cdot)$     are  homogeneous of degree $\bar 0$.
Therefore, this $\ZZ/4$-grading on $T$ is a grading as a triple system.

\smallskip

\noindent\emph{Quaternionic type:}\quad
In this case, $T$ is a right $\HH$-vector space, endowed with a non-degenerate
skew-hermitian form $\hup\colon T\times T\rightarrow \HH$. 
The skew-symmetric form and triple product are given by the following formulas, from \eqref{eq:triplecaso_quaternionic} and \eqref{eq:formacaso_quaternionic}:
\[
\begin{split}
&(x\vert y)=\frac{1}{2}\bigl(\hup(x,y)-\hup(y,x)\bigr),\\
&[x,y,z]=x\hup(y,z)+y\hup(x,z)+z\hup(x,y)-z(x\vert y).
\end{split}
\]
Consider the $\ZZ/4$-grading on $\HH$ with $\HH_{\bar 0}=\RR 1+\RR\bj$, 
$\HH_{\bar 2}=\RR\bi+\RR \bk$. Take an $\HH$-basis 
$\{u_1,\ldots,u_n\}$ with $\hup(u_r,u_s)=0$ if $r\neq s$, and $\hup(u_r,u_r)=\bi$. (This is always possible.) Then $T$ is a $\ZZ/4$-graded module over $\HH$ with 
$T=T_{\bar 1}\oplus T_{\bar 3}$, for the real subspaces 
$T_{\bar 1}=\sum_{r=1}^n u_r\HH_{\bar 0}$ and 
$T_{\bar 3}=\sum_{r=1}^n u_r\HH_{\bar 2}$. It follows that $\hup(\cdot,\cdot)$, considered as a bilinear form over $\RR$, is homogeneous of degree $\bar 0$, and so is  $(\cdot\vert \cdot)$,
and hence this $\ZZ/4$-grading on $T$ is a grading as a triple system.
\end{proof}

\bigskip

\begin{proposition}
Any exceptional simple real symplectic triple system $T$ is isomorphic to 
$T^{[-1]}$.
\end{proposition}
\begin{proof}
Again we will check that Lemma~\ref{le:Z4} applies in all cases. \smallskip

\noindent\emph{Split types:}\quad
In the $G_2$-type, $\inder(T)$ is the Lie algebra $\frsl_2$ and $T$ is its 
four-dimensional irreducible module. Let $h$ be a diagonalizable 
element of $\frsl_2$, chosen as usual so that the eigenvalues of 
$\ad_h$ are $0$ and $\pm 2$. Then the eigenvalues of the action of $h$ 
on $T$ are $\pm 1$ and $\pm 3$: 
$T=T_{-3}\oplus T_{-1}\oplus T_1\oplus T_3$. Then $T$ is 
$\ZZ/4$-graded with $T_{\bar 1}=T_{-3}\oplus T_1$ and 
$T_{\bar 3}=T_{-1}\oplus T_3$.\smallskip

In the $F_4$-type, $\inder(T)=\frsp(W,\bup)$ is the symplectic Lie algebra of 
a six-dimensional vector space $W$ with an alternating form $\baup$, and $T$ is the kernel of the 
$\frsp(W)$-invariant linear map $\bigwedge^3W\rightarrow W$: 
$x_1\wedge x_2\wedge x_3\mapsto 
\baup(x_1,x_2)x_3+\baup(x_2,x_3)x_1+\baup(x_3,x_1)x_2$.
As in Subsection~\ref{ss:F4}, once we fix a Cartan 
subalgebra of $\inder(T)$,  the weights of $W$ relative to it are 
$\pm\epsilon_i$, $1\leq i\leq 3$, and, with the natural ordering with 
$\epsilon_1>\epsilon_2>\epsilon_3>0$, $T$ is the irreducible module 
with highest weight $\epsilon_1+\epsilon_2+\epsilon_3$, whose weights 
 are $\pm\epsilon_1\pm\epsilon_2\pm\epsilon_3$ and 
$\pm\epsilon_i$, $1\leq i\leq 3$, all with multiplicity $1$. The 
homomorphism 
$\ZZ\epsilon_1\oplus\ZZ\epsilon_2\oplus\ZZ\epsilon_3\rightarrow 
\ZZ/4$ given by $\epsilon_i\mapsto \bar 1$ gives a $\ZZ/4$-grading of 
$T$ with support $\{\bar 1,\bar 3\}$.\smallskip

In the $E_6$-type, $\inder(T)=\frsl(W)$ for a six-dimensional vector 
space $W$, and $T$ is its module $\bigwedge^3W$. Fix a basis of $W$ 
and take the element $h\in\frsl(W)$ whose coordinate matrix is 
$\diag(1,1,1,-1,-1,-1)$. The eigenspaces of the action of $h$ give a 
$\ZZ$-grading of $T$ with support $\{-3,-1,1,3\}$. As for the 
$G_2$-type, 
$T$ is 
$\ZZ/4$-graded with $T_{\bar 1}=T_{-3}\oplus T_1$ and 
$T_{\bar 3}=T_{-1}\oplus T_3$.\smallskip

In the $E_7$-type, $\inder(T)=\frso(W\oplus W^*)$ for a six-dimensional 
vector space $W$, and 
$T=\bigwedge_{\bar 0}W=\bigoplus_{i=0}^3\bigwedge^{2i}W$. Now $\inder(T)$ is graded by 
$\ZZ/4$ with $\inder(T)_{\bar 0}=\sigma_{W,W^*}$ and $\inder(T)_{\bar 2}=
\sigma_{W,W}\oplus\sigma_{W^*,W^*}$. Then  $T$ is a $\ZZ/4$-graded vector space with
\begin{equation}\label{eq:E7splitZ4} 
\textstyle{T_{\bar 1}=\bigwedge^0W\oplus\bigwedge^4W\quad\text{and}\quad
T_{\bar 3}=\bigwedge^2W\oplus\bigwedge^6W.}
\end{equation} 
The alternating bilinear form  $\baup$
in \eqref{eq:E7_b} is homogeneous of degree $\bar 0$ on $T$, and
\eqref{eq:E7dxy} shows that this $\ZZ/4$-grading on $T$ is a grading as a triple 
system.\smallskip

Finally, in the $E_8$-type, $\inder(T)$ is the Lie algebra of type $E_7$, 
which appears as $\frsl(U)\oplus\bigwedge^4U$, for an 
eight-dimensional vector space $U$. Our symplectic triple system is 
$T=\bigwedge^2U\oplus\bigwedge^2U^*$. This decomposition is a 
$\ZZ/4$-grading with $T_{\bar 1}=\bigwedge^2U$ and 
$T_{\bar 3}=\bigwedge^2U^*$.

\medskip

We will deal now with the remaining non-split cases.  

\noindent\emph{$E_6$-types:}\quad
Use here the notations in Subsection~\ref{ss:E6nonsplit}. Thus $\{e_i: 1\leq i\leq 6\}$ is an orthogonal basis of the six dimensional complex vector space $W$  relative to a non-degenerate hermitian form 
$\hup\colon W\times W\rightarrow \CC$, with $\hup(e_i,e_i)=1$ if $i\le p$ and $-1$ otherwise, with 
$p$ either $3$ or $5$. As a real vector space, $W$ is $\ZZ/4$-graded with 
$W_{\bar 1}=\sum_{i=1}^6\RR e_i$, and 
$W_{\bar 3}=\bi W_{\bar 1}$. This grading on $W$ induces a grading by $\ZZ/4$ on 
$\bigwedge^3W$ with support $\{\bar 1,\bar 3\}$.

We had fixed the `determinant map' $\det\colon \bigwedge^6W\rightarrow \CC$ by imposing 
$\det(e_{123456})=1$. Then the associated skew-symmetric bilinear 
form $(\cdot\vert \cdot)\colon \bigwedge^3W\times\bigwedge^3W\rightarrow \CC$ is homogeneous of 
degree $\bar 2$ (again the $\ZZ/4$-grading on $\CC$ is given by $\CC_{\bar 0}=\RR 1$, 
$\CC_{\bar 2}=\RR\bi$).
Hence the triple product $[\cdot,\cdot,\cdot]$ on $\bigwedge^3W$  is also homogeneous of degree 
$\bar 2$. Moreover, the grading on $\bigwedge^3W$ induces one on its dual  
 ($e^I\in(\bigwedge^3W^*)_{\bar3}$ and $\bi e^I\in(\bigwedge^3W^*)_{\bar1}$),
and the 
(real) linear isomorphisms $\Phi_3$ and $\Psi$ are homogeneous of degree $\bar 2$. It 
turns out that the composition $\Gamma=\Phi_3^{-1}\Psi\colon \bigwedge^3W\rightarrow \bigwedge^3W$ is 
homogeneous of degree $\bar 0$, so the fixed subspace 
$T=\left(\bigwedge^3W\right)^{\Gamma}$ is a graded subspace: 
$T=T_{\bar 1}\oplus T_{\bar 3}$. 
But $\bi[\cdot,\cdot,\cdot]$ and $
\bi(\cdot\vert\cdot)$ are both homogeneous of degree $\bar 0$, because 
multiplication by $\bi$ has degree $\bar 2$. Hence this $\ZZ/4$-grading of the
(real) vector space $T$ is in fact a grading of the symplectic triple system 
$\bigl(T,\bi[\cdot,\cdot,\cdot],
\bi(\cdot\vert\cdot)\bigr)$.\smallskip

\noindent\emph{First $E_7$-type:}\quad
The first non-split simple real symplectic triple system   with envelope of 
type $E_7$, denoted by $T^\Gamma$ in Subsection~\ref{ss:E7nonsplit},
satisfies that $\inder(T^\Gamma)$ is isomorphic to $\frso_{10,2}(\RR)$. Thus 
$\inder(T^\Gamma)$ is the orthogonal Lie algebra of a real vector space $V$ of dimension 
$12$, endowed with a symmetric bilinear form $\bup$ such that that there is a basis 
$\{u_1,u_2,v_1,v_2,w_1,\ldots,w_8\}$ of $V$, with $\bup(u_1,v_1)=\bup(u_2,v_2)=1= 
\bup(w_r,w_r)=0$, $1\leq r\leq 8$, and the remaining values of $\bup$ on vectors of the 
basis  are either $0$ or obtained from the above by symmetry. The linear 
endomorphism $h$ with $h(u_1)=u_1$, $h(v_1)=-v_1$, and $h(u_2)=h(v_2)=h(w_r)=0$ for 
all $r$, lies in the orthogonal Lie algebra $\frso(V,\bup)$. The diagonalizable element 
$h$ acts with eigenvalues $1$, $-1$ and $0$ both on $V$ and on $\frso(V,\bup)$, and 
with eigenvalues $\frac{1}{2}$ and $-\frac{1}{2}$ on the half-spin modules. To see 
this, note that after complexification, it is easy to get a Cartan subalgebra 
containing $h$,  with corresponding root system $\Phi
=\{\pm\epsilon_r\pm\epsilon_s: 1\leq r\neq s\leq 6\}$ 
with $\epsilon_1(h)=1$, $\epsilon_r(h)=0$ for $r\neq 1$, where $\pm\epsilon_r$, 
$1\leq r\leq 6$, are the weights of the natural module. 

Then we get the following $\ZZ/4$-grading of 
$T^\Gamma$: $T^\Gamma=T^\Gamma_{\bar 1}\oplus T^\Gamma_{\bar 3}$, with 
$T^\Gamma_{\bar 1}=\{x\in T^\Gamma: h.x=\frac{1}{2}x\}$, 
$T^\Gamma_{\bar 3}=\{x\in T^\Gamma: h.x=-\frac{1}{2}x\}$.
\smallskip

\noindent\emph{`Quaternionic' $E_7$-type:}\quad
The other non-split simple real symplectic triple system   with envelope of 
type $E_7$ is related to a skew-hermitian non-degenerate form on a quaternionic vector space.

In this case, let $\mathcal{U}$ be a right $\HH$-module of rank $6$ endowed 
with a skew-hermitian non-degenerate form 
$\hup\colon\mathcal{U}\times\mathcal{U}\rightarrow \HH$ as in Subsection \ref{ss:E7nonsplit}. Using the notation there, our simple real symplectic triple
system is $T^{\widetilde{R_\bj}}$, where $T=\bigwedge\subo W$ with
$W=v_1\HH\oplus v_2\HH\oplus v_3\HH$. Recall from Subsection \ref{ss:E7nonsplit}
that $v_1,v_2,v_3$ span a maximal isotropic subspace for $\hup$.

Now, $T$ is $\ZZ/4$-graded as in \eqref{eq:E7splitZ4}. The conjugate-linear map 
$\widetilde{R_\bj}$ preserves the homogeneous components  
$T_{\bar 1}=\bigwedge^0 W\oplus\bigwedge^4W$ and 
$T_{\bar 3}=\bigwedge^2W\oplus
\bigwedge^6 W$, and therefore this $\ZZ/4$-grading on $T$ is inherited by its
real form $T^{\widetilde{R_\bj}}$.

\smallskip

\noindent{$E_8$-type:}\quad
We use the notations in Subsection~\ref{ss:E8nonsplit}. Thus $U$    denotes a complex 
eight-dimensional vector space, endowed with a non-degenerate hermitian 
form $\hup\colon U\times U\rightarrow \CC$ and an orthogonal basis 
$\{e_r: 1\leq r\leq 8\}$, with $\hup(e_r,e_r)=\pm 1$ for all $r$.  
Recall that $\cL=\frsl(U)\oplus \bigwedge^4U$ is the algebra of inner derivations of the complex symplectic triple system $T=\bigwedge^2U\oplus\bigwedge^2U^*$.

As a real vector space, $\bigwedge^2U$ is $\ZZ/4$-graded with 
\[
\textstyle{\bigl(\bigwedge^2U\bigr)_{\bar 1}}=\espan_{\RR}\left\langle e_{rs}: 1\leq r<s\leq 8\right\rangle,
\] and 
$\bigl(\bigwedge^2U\bigr)_{\bar 3}=\bi\bigl(\bigwedge^2U\bigr)_{\bar 1}$. The same happens with $\bigwedge^2U^*$, and hence $T $ is $\ZZ/4$-graded.

The $\ZZ/4$-grading on $U$ given by 
$U_{\bar 0}=\bigoplus_{r=1}^8\RR e_r$ and $U_{\bar 2}=\bi U_{\bar 0}$ 
 induces naturally a  $\ZZ/4$-grading on   $\frsl(U)$ with support $\{\bar 0,\bar 2\}$.
 Also $\bigwedge^4U$ is 
$\ZZ/4$-graded with 
\[
\textstyle{\bigl(\bigwedge^4U\bigr)_{\bar 0}}
=\espan_{\RR}\left\langle e_{r_1 r_2 r_3 r_4}: 
1\leq r_1<r_2<r_3<r_4\leq 8\right\rangle
\] 
and 
$\bigl(\bigwedge^4U\bigr)_{\bar 2}=\bi\bigl(\bigwedge^4U\bigr)_{\bar 0}$. 
So $\cL$ is $\ZZ/4$-graded.

Now, the action $\cL\times T\rightarrow T$ gives a homogeneous map
of degree $\bar 0$,  the bilinear form
$(\cdot\vert\cdot)_{\cL}$ is also homogeneous of degree $\bar 0$, and the 
bilinear form $(\cdot\vert\cdot)$ on $T$ is homogeneous of degree 
$\bar 2$. It follows that the map $d_{.,.}\colon T\times T\rightarrow \cL$, determined by 
$\bigl(l\vert d_{x,y}\bigr)_{\cL}=-2( l.x\vert y)$ is homogeneous of 
degree $\bar 2$.

Finally, $\Upsilon_{T}$ is homogeneous of degree $\bar 0$, so 
the symplectic triple system $ T^\Upsilon$, given by the elements in $T$ fixed 
by $\Upsilon_{T}$, inherits the $\ZZ/4$-grading on $T$, as a real vector 
space, with support $\{\bar 1,\bar 3\}$.
  As the triple product on $T^\Upsilon$ is given by $[x,y,z]_{T^\Upsilon}=\bi d_{x,y}(z)$, for all 
$x,y,z\in T^\Upsilon$, and the map $\bi d_{.,.}$ is homogeneous of degree $\bar 0$,
so the $\ZZ/4$-grading on $T^\Upsilon$ is actually a $\ZZ/4$-grading as a (real) triple system,   not just as a vector space.\end{proof}

\bigskip

As a consequence, Theorem \ref{th:weak_iso} may be reformulated, replacing the words \emph{`up to weak isomorphism'} by \emph{`up to isomorphism'}, and this completes the classification of the simple real symplectic triple systems up to isomorphism.

We conclude the paper by stating this classification in a different way as follows:

\begin{theorem}\label{th_main}
If $T$ is a real simple  symplectic triple system,  then
  {$(\g(T),\mathfrak{inder}(T))$} belongs to the following list:
 \begin{enumerate}
  \item  {Symplectic}: $(\mathfrak{sp}_{2n+2}(\RR),\mathfrak{sp}_{2n}(\RR)), $

\item Orthogonal:  {$(\mathfrak{so}_{p+2,q+2}(\RR),\mathfrak{so}_{p,q}(\RR)\oplus\mathfrak{sl}_2(\RR)), $}

  \item {Quaternionic}: $(\mathfrak{so}^*_{2n+4},\mathfrak{so}^*_{2n}\oplus\mathfrak{su}_2),$ 
 
   \item  Special:  {$(\mathfrak{sl}_{2+n}(\RR),\mathfrak{gl}_{n}(\RR)),$}
   
   \item Unitarian: $( \mathfrak{su}_{p+1,n+1-p},\mathfrak{u}_{p,n-p}), $\vspace{2pt}
    
  \item  {Exceptional}: $ \left\{\text{\begin{tabular}{l}
   $(\mathfrak{g}_{2,2},\mathfrak{sl}_2(\RR))$, \\
   $(\mathfrak{f}_{4,4},\mathfrak{sp}_6(\RR))$,\\
 $(\mathfrak{e}_{6,6},\mathfrak{sl}_6(\RR))$,
  $(\mathfrak{e}_{6,2},\mathfrak{su}_{3,3})$,
  $(\mathfrak{e}_{6,-14},\mathfrak{su}_{5,1})$,\\
$(\mathfrak{e}_{7,7},\mathfrak{so}_{6,6}(\RR))$,
 $( \mathfrak{e}_{7,-5},\mathfrak{so}_{12}^*)$,
 $(\mathfrak{e}_{7,-25},\mathfrak{so}_{10,2}(\RR))$,\\
$(\mathfrak{e}_{8,8},\mathfrak{e}_{7,7})$,
 $(\mathfrak{e}_{8,-24},\mathfrak{e}_{7,-25})$.\end{tabular}}\right.$
\end{enumerate} 
  Moreover, if  $T$ and $T'$ are real simple  symplectic triple systems with isomorphic inner derivation algebras and isomorphic enveloping algebras, then $T$ and $T'$ are isomorphic too.
 \end{theorem}


\begin{thebibliography}{1}

\bibitem{Adams}
J.F.~Adams, \emph{Lectures on Exceptional Lie Groups}, edited by Zafer Mahmud and Mamoru Mimura. Chicago Lectures in Mathematics. University of Chicago Press, Chicago, IL, 1996.

\bibitem{cortes1}
 D.V.~Alekseevsky, and V.~Cort\'es, 
 \emph{The twistor spaces of a para-quaternionic K\"ahler manifold},
  Osaka J. Math. \textbf{45} (2008), no.~1, 215--251. 
  

\bibitem{Allison78}
 B.N.~Allison,   
 \emph{A class of nonassociative algebras with involution containing the class of Jordan algebras}, 
 Math. Ann. \textbf{237}  (1978), no.~2, 133--156.
 

\bibitem{parabolic}
A.~\v{C}ap and J.~Slov\'ak, 
{\sl    Parabolic geometries. I. Background and general theory.} Mathematical Surveys and Monographs, 154. American Mathematical Society, Providence, RI, 2009. x+628 pp. ISBN: 978-0-8218-2681-2

\bibitem{skew1_a}
 T.~De Medts, 
 \emph{   Structurable algebras of skew-dimension one and hermitian cubic norm structures},
 Comm. Algebra \textbf{47} (2019), no. 1, 154--172.


   \bibitem{Jeroen}
   T.~De Medts and J.~Meulewaeter,
    \emph{Inner ideals and structurable algebras: Moufang sets, triangles and hexagons},
    preprint  arXiv:2008.02700v2.

%
% 
\bibitem{nues3Sas}
C.~Draper, M.~Ortega and F.J.~Palomo,
\emph{Affine Connections on 3-Sasakian Homogeneous Manifolds},
  Math. Z. \textbf{294} (2020), no.~1-2, 817--868.
 
\bibitem{hol}
C.~Draper,
\emph{Holonomy on 3-Sasakian homogeneous manifolds versus symplectic triple systems},
Transformation Groups (2020), on line
DOI: 10.1007/s00031-020-09609-w

\bibitem{CM}
C.~Draper Fontanals,
\emph{Homogeneous Einstein manifolds based on symplectic triple systems},
Communications in Mathematics,
\textbf{28} (2020), 139--154. 
DOI: 10.2478/cm-2020-0016

 \bibitem{grade6-14}
C.~Draper and V.~Guido,
\emph{ Gradings on the real form $\mathfrak{e}_{6,-14}$},
 J. Math. Phys. \textbf{59} (2018), no.~10, 101702, 20 pp. 


 

\bibitem{EldLA}
A.~Elduque, \emph{Lie algebras}, 
\texttt{http://personal.unizar.es/elduque/files/LAElduque.pdf}.

\bibitem{Eld06}  
A.~Elduque, \emph{New simple Lie superalgebras in characteristic $3$}, 
J. Algebra \textbf{296} (2006), no.~1, 196--233.

\bibitem{Eld07}
A.~Elduque, \emph{Some new simple modular Lie superalgebras}, 
Pacific J. Math. \textbf{231} (2007), no.~2, 337--359.

 %  
 \bibitem{alb_sts}
 A.~Elduque, 
 \emph{ The Magic Square and Symmetric Compositions II},
  Rev. Mat. Iberoamericana  \textbf{23} (2007), no.~1, 57--84.

\bibitem{alb_resumen}
 A.~Elduque, 
  \emph{ Symplectic and orthogonal triple systems, and a Freudenthal  Magic Supersquare}, 
   Proceedings of the XVIth Latin American Algebra Colloquium, 253--270, Bibl. Rev. Mat. Iberoamericana, Rev. Mat. Iberoamericana, Madrid, 2007.


\bibitem{EKmon}
A.~Elduque and M.~Kochetov, \emph{Gradings on simple Lie algebras},
Mathematical Surveys and Monographs \textbf{189}, American
Mathematical Society, Providence, RI; Atlantic Association for Research
in the Mathematical Sciences (AARMS), Halifax, NS, 2013.




  \bibitem{Ferrar}
  J.R.~Faulkner,  J.C.~Ferrar, 
   \emph{On the structure of symplectic ternary algebras},
    Nederl. Akad. Wetensch. Proc. Ser. A 75=Indag. Math. \textbf{34} (1972), 247--256.
  
 \bibitem{FreuII}
   H.~Freudenthal,   
    \emph{Beziehungen der $\mathfrak{E}_7$ und $\mathfrak{E}_8$ zur Oktavenebene. II. }(German) 
   Nederl. Akad. Wetensch. Proc. Ser. A. \textbf{57}, (1954) 363--368 = Indag. Math. \textbf{16}, 363--368 (1954). 
  
 
\bibitem{GraceYoung}
J.H.~Grace and A.~Young, \emph{The Algebra of Invariants}, Cambridge, 1903.



\bibitem{Wallach}
 B.H.~Gross, N.R.~Wallach,  
  \emph{On quaternionic discrete series representations, and their continuations}, 
 J. Reine Angew. Math.  \textbf{481} (1996), 73--123.


\bibitem{Jacobson}
N.~Jacobson, \emph{Clifford Algebras for Algebras with Involution of 
Type $D$}, 
J.~Algebra \textbf{1} (1964), 288--300.

\bibitem{Kac}
V.G.~Kac, \emph{Infinite-dimensional Lie algebras}, 
Third edition. Cambridge University Press, Cambridge, 1990.


\bibitem{proyFr}
H.~Kaji and O.~Yasukura,  
\emph{Projective geometry of Freudenthal's varieties of certain type},
Michigan Math. J.  \textbf{52} (2004), no.~3, 515--542. 


\bibitem{kerner}
R.~Kerner, 
\emph{ Ternary and non-associative structures},
 Int. J. Geom. Methods Mod. Phys. \textbf{5} (2008), no.~8, 1265--1294. 


\bibitem{KMRT}
M.-A.~Knus, A.~Merkurjev, M.~Rost, and J.-P.~Tignol, \emph{The book of involutions},
   American Mathematical Society Colloquium Publications \textbf{44},
   American Mathematical Society, Providence, RI, 1998.


\bibitem{Loketesis}
H.Y.~Loke,  
\emph{Quaternionic representations of exceptional Lie groups},
Pacific J. Math.  \textbf{211} (2003), no.~2, 341--367.
   

  

  

\bibitem{Lam}
T.Y.~Lam, 
\emph{The algebraic theory of quadratic forms},
Revised second printing. Mathematics Lecture Note Series. 
Benjamin/Cummings Publishing Co., Inc., Advanced Book Program, Reading, Mass., 1980.


 \bibitem{Mey}
K.E.~Meyberg, 
\emph{  Theorie der Freudenthalschen Tripelsysteme. I, II,} (German) Nederl. Akad. Wetensch. Proc. Ser. A 71=Indag. Math.  \textbf{30} (1968) 162--174, 175--190.
 

\bibitem{Oni}
A.~Onishchik, \emph{Lectures on real semisimple Lie algebras and their representations},
ESI Lectures in Mathematics and Physics. European Mathematical Society (EMS), Z\"urich, 2004.



\bibitem{Vinberg}
 E.B.~Vinberg, 
 \emph{Non-abelian gradings of Lie algebras},
  50th Seminar "Sophus Lie'', 19--38, Banach Center Publ., 113, Polish Acad. Sci. Inst. Math., Warsaw, 2017. 


\bibitem{Yamaguchi}
K.~Yamaguchi,   
\emph{ Differential systems associated with simple graded Lie algebras},
 Progress in differential geometry, 413--494, Adv. Stud. Pure Math., 22, Math. Soc. Japan, Tokyo, 1993.


\bibitem{ternarias}
K.~Yamaguti and H.~Asano,
\emph{ On the Freudenthal's construction of exceptional Lie algebras},
 Proc. Japan Acad. \textbf{51} (1975), no.~4, 253--258.









\end{thebibliography}
\end{document}